\documentclass[a4paper, reqno, 11pt]{amsart}
\usepackage[foot]{amsaddr}

\DeclareSymbolFont{rsfscript}{OMS}{rsfs}{m}{n}
\DeclareSymbolFontAlphabet{\mathrsfs}{rsfscript}

\usepackage{amsfonts}
\usepackage[centertags]{amsmath}
\usepackage{amssymb}
\usepackage{amsthm}
\usepackage{graphicx}
\usepackage{caption}
\usepackage{enumerate}
\usepackage[textwidth=16cm, hmarginratio=1:1]{geometry}
\usepackage[colorlinks]{hyperref}
\usepackage{pstricks,pst-node,pst-text,pst-3d}
\usepackage{color}
\usepackage[all]{xy}
\usepackage{float}
\usepackage{amsmath}
\usepackage{longtable}
\usepackage{resizegather}
\usepackage{framed}
\usepackage{tikz}
\usetikzlibrary{fit}

\usepackage{marginnote}

\usepackage{rotating}
\usepackage{pdflscape}
\usepackage{mathdots}

\newtheorem{thm}{Theorem}[section]
\newtheorem{lem}[thm]{Lemma}
\newtheorem{cor}[thm]{Corollary}
\newtheorem{prop}[thm]{Proposition}
\newtheorem{problem}{Problem}[section]
\newtheorem*{multiprop}{Proposition 3.\arabic{thm}.i}

\theoremstyle{definition}

\theoremstyle{remark}

\newtheorem{example}[thm]{Example}
\newtheorem{step}{Step}
\newtheorem*{remks}{Remarks}
\newtheorem*{remk}{Remark}

\usepackage{multirow}

\DeclareMathOperator{\var}{\mathsf{var}}
\DeclareMathOperator{\con}{\mathsf{con}}
\DeclareMathOperator{\diag}{\mathsf{diag}}

\newcommand{\sgp}{semi\-group}
\newcommand{\sgps}{semi\-groups}
\def\malcev{\mathop{\hbox{$\bigcirc$\kern-9pt\raise1.25pt\hbox{\scriptsize$m$}\kern1.5pt}}}

\usepackage{datetime2}
\begin{document}
\title{Laws for triangular matrices}
\author{Xun Hu, Yanfeng Luo and Mikhail Volkov}

\address[Xun Hu, Yanfeng Luo]{\normalfont School of Mathematics and Statistics, Lanzhou University, Lanzhou 730000, P. R. China}
\address[Xun Hu]{\normalfont Department of Mathematics and Statistics, Chongqing Technology and Business University, Chongqing, $400067$, P. R. China}
\address[Mikhail Volkov]{\normalfont Institute of Natural Sciences and Mathematics, Ural Federal University, Lenina 51, 620000 Ekaterinburg, Russia}
\email{luoyf@lzu.edu.cn, m.v.volkov@urfu.ru}

\date{}
\renewcommand{\subjclassname}{\textup{2020} Mathematics Subject Classification}

\begin{abstract}
We investigate identities satisfied by monoids of triangular matrices over fields and over additively idempotent semirings with 0 and 1, and exhibit several new instances in which these identities admit no finite axiomatization.
\end{abstract}


\maketitle

\tableofcontents

\section*{Introduction}

Matrices\footnote{In this paper, the word `matrix' always means an $n\times n$ matrix; moreover, we always assume that $n\ge 2$.} and matrix operations constitute fundamental  tools in many branches of mathematics. Important properties of matrix operations are often expressed in the form of \emph{laws} or \emph{identities} such as the associative law for multiplication of matrices or the product law for transposition.

If one aims to classify matrix identities of a certain type, then a natural approach is to look for a collection of `basic' identities from which all other identities can be inferred. Such a collection is usually referred to as an \emph{identity basis} or simply a \emph{basis}. For instance, all identities of matrices over an infinite field that involve only multiplication are known to follow from the associative law. Thus, the associative law forms a basis of such `multiplicative' identities. On the other hand, some matrix semigroups may not admit any finite identity basis: a classical example here is the 6-element semigroup formed by the following $2\times 2$ matrices:
\begin{equation}
\label{eq:Brandt}
\begin{pmatrix} 0 & 0\\ 0 & 0\end{pmatrix},\
\begin{pmatrix} 1 & 0\\ 0 & 0\end{pmatrix},\
\begin{pmatrix} 0 & 1\\ 0 & 0\end{pmatrix},\
\begin{pmatrix} 0 & 0\\ 1 & 0\end{pmatrix},\
\begin{pmatrix} 0 & 0\\ 0 & 1\end{pmatrix},\
\begin{pmatrix} 1 & 0\\ 0 & 1\end{pmatrix}.
\end{equation}
(This example is known as the \emph{Brandt monoid} $B_2^1$; it is also sometimes referred to as the \emph{Perkins semigroup}, since the fact that $B_2^1$ has no finite identity basis was discovered by Peter Perkins~\cite{Per69}.) The Finite Basis Problem for matrix semigroups, that is, the natural question of classifying matrix semigroups into those that admit finite identity bases and those that do not, has been intensively investigated since the mid-1980s. The present paper contributes to a special instance of this problem, namely when semigroups of \textbf{triangular}\footnote{Here and throughout the paper, `triangular' means upper triangular.} matrices are considered.

Of course, multiplication is not the only natural operation on matrices. Arguably, the most important operations among the other matrix operations are addition and transposition. Studying matrix identities that involve multiplication and addition is a classical research direction, motivated by several important problems in geometry and algebra (see Shimshon Amitsur's survey~\cite{Ami74} for an overview of the origins of the theory) that has eventually led to the profound theory of PI-rings. For a comprehensive study of matrix laws that involve  multiplication and transposition, we refer to the article by Karl Auinger, Igor Dolinka, and the third-named author~\cite{ADV12}.

Since the set $T_n$ of all $n\times n$ triangular matrices is closed under addition, it makes perfect sense to study laws of triangular matrices that involve both multiplication and addition; such studies have indeed been carried out in ring and semiring theory; see, e.g., \cite{{Mal71,Pol80,Sid81}}. In contrast, the transpose of an upper triangular matrix is a lower triangular matrix, so the set $T_n$ is not closed under transposition. However, it is closed under \emph{skew transposition}, that is, reflection in the secondary diagonal. Identities involving both multiplication and skew transposition have been studied for certain semigroups of triangular matrices, and in the present paper we contribute to this direction as well.

The paper is structured as follows. Section~\ref{sec:preliminaries} collects the necessary preliminaries. Section~\ref{sec:plain} deals with the Finite Basis Problem for `plain' laws of triangular matrices (that is, those involving only multiplication), while Section~\ref{sec:unary} treats the problem for identities  involving both multiplication and skew transposition.

Our arguments use basic concepts of equational logic specialized to semigroups and semigroups with involution. We recall these concepts in Section~\ref{subsec:vocabulary}; a more general treatment, together with additional details, can be found in \cite[Chapter II]{BuSa81}. We assume that the reader is familiar with the notion of a nilpotent group. The few other semigroup and group notions that we need can be found in the early chapters of the standard textbooks \cite{Hall:1959,Howie:1995}. Otherwise, the paper is largely self-contained; in particular, no background in matrix theory is assumed.

\section{Preliminaries}
\label{sec:preliminaries}
\subsection{Which matrices and which matrix operations are considered}
\label{subsec:matrices}
The reader may observe that, in defining the Brandt monoid $B_2^1$, we have not specified the origin of the symbols 0 and 1 appearing in the six matrices in \eqref{eq:Brandt}. In fact, they may come from any field, or from the Boolean semiring $\mathbb{B}$, or, more generally, from any semiring $\mathbb{S}$ with $0\ne 1$ such that $0+0=0\cdot0=1\cdot0=0\cdot1=0$ and $0+1=1+0=1\cdot1=1$ (the value of the sum $1+1$ does not matter since this sum never arises when multiplying matrices in \eqref{eq:Brandt}). The latter configuration is the one we use in this paper as it allows for uniform treatment of matrices with entries of various kinds.

Thus, our base structures are semirings $\mathbb{S}=(S,+,\cdot)$ with two binary operations, addition $+$ and multiplication $\cdot$, such that
\begin{itemize}
\item[(S1)] $(S,+)$ is a commutative semigroup, that is,
\[
a+(b+c)=(a+b)+c\ \text{ and }\  a+b=b+a\  \text{ for all }\ a,b,c\in S;
\]
\item[(S2)] $(S,\cdot)$ is a semigroup, that is,
\[
a(bc)=(ab)c\ \text{ for all }\ a,b,c\in S;
\]
\item[(S3)] the left- and right-hand distributive laws hold, that is,
\[
a(b+c)=ab+ac\ \text{ and }\ (a+b)c=ac+bc\ \text{ for all }\ a,b,c\in S;
\]
\item[(S4)] there exists an element 0 that is neutral for addition and absorbing for multiplication, that is,
\[
0+a=a+0=a\ \text{ and }\  a\cdot0=0\cdot a=0\  \text{ for all }\ a\in S;
\]
\item[(S5)] there exists an element 1 different from 0 that is neutral for multiplication, that is,
\[
a\cdot1=1\cdot a=a\  \text{ for all }\ a\in S.
\]
\end{itemize}

Examples of semirings satisfying (S1)--(S5) include fields and, more generally, rings with $1\ne 0$, the Boolean semiring $\mathbb{B}:=(\{0,1\},+,\cdot)$ in which $1+1=1$, and, more generally, bounded distributive lattices. A very important example is the \emph{tropical semiring} $\mathbb{T}$. The latter is often defined on the set $\mathbb{R}$ of all real numbers augmented with the symbol $-\infty$ with the operations $\max\{a,b\}$ as addition and $a+b$ as multiplication. In this model, $-\infty$ is neutral for addition and absorbing for multiplication, while the number 0 is neutral for multiplication. To comply with the notation adopted in the present paper, we use another model of $\mathbb{T}$ in which the carrier is the set $\mathbb{R}_{\ge0}$ of all nonnegative reals, addition is still defined as $\max\{a,b\}$, and multiplication is the usual multiplication of reals. The two models are clearly isomorphic under the exponentiation map $\mathbb{R}\cup\{-\infty\}\to\mathbb{R}_{\ge0}$ defined by $a\mapsto 2^a$ for $a\in \mathbb{R}$ and $-\infty\mapsto 0$. In the model $(\mathbb{R}_{\ge0},\max,\cdot)$, the numbers 0 and 1 play their `normal' roles as in (S4) and (S5).

For a semiring $\mathbb{S}=(S,+,\cdot)$ satisfying (S1)--(S5), we denote by $M_n(\mathbb{S})$ the set of all $n\times n$ matrices with entries in $S$; this is a semigroup under the usual matrix multiplication
\[
\bigl(\alpha_{ij}\bigr)_{n\times n}\cdot\bigl(\beta_{ij}\bigr)_{n\times n}:=\left(\sum_{k=1}^n\alpha_{ik}\beta_{kj}\right)_{n\times n}.
\]
$M_n(\mathbb{S})$ is actually a monoid with the matrix
\[
E:=\begin{pmatrix}
    1&0&\cdots& 0&0\\
    0&1&\cdots& 0&0\\
    \vdots & \vdots & \ddots &\vdots & \vdots\\
    0&0&\cdots& 1&0\\
    0&0&\cdots& 0&1
\end{pmatrix}_{n\times n}
\]
as its neutral element for multiplication.

A matrix $\bigl(\alpha_{ij}\bigr)_{n\times n}\in M_n(\mathbb{S})$ is called (\emph{upper}) \emph{triangular} if $\alpha_{ij}=0$ for all $1\le j<i\le n$. The set $T_n(\mathbb{S})$ of all $n\times n$ triangular matrices forms a submonoid of $M_n(\mathbb{S})$. We will also consider several submonoids of $T_n(\mathbb{S})$, in particular, the following two, which can be defined over an arbitrary semiring $\mathbb{S}$:
\begin{itemize}
\item the monoid $U^0T_n(\mathbb{S})$ of triangular matrices whose diagonal entries are 0s or 1s,
    \[
    U^0T_n(\mathbb{S}):=\left\{\bigl(\alpha_{ij}\bigr)_{n\times n}\in T_n(\mathbb{S})\;\middle|\; \alpha_{ii}\in\{0,1\}\ \text{for}\ i=1,\dots,n\right\};
    \]
\item the monoid $UT_n(\mathbb{S})$ of \emph{unitriangular} matrices, that is, triangular matrices with all diagonal entries equal to 1,
    \[
    UT_n(\mathbb{S}):=\left\{\bigl(\alpha_{ij}\bigr)_{n\times n}\in T_n(\mathbb{S})\;\middle|\;  \alpha_{ii}=1\ \text{for}\ i=1,\dots,n\right\}.
    \]
\end{itemize}

For a matrix $A=\bigl(\alpha_{ij}\bigr)_{n\times n}$, we denote its \emph{skew transpose} $\bigl(\alpha_{n+1-j\, n+1-i}\bigr)_{n\times n}$ by $A^S$. `Geometrically', the skew transpose of a matrix is its reflection in the secondary diagonal; it is also easy to see that $A^S$ is equal to the usual transpose  $A^T$ conjugated by the matrix
\[
P:=\begin{pmatrix}
0 & 0 & \cdots & 0 & 1\\
0 & 0 & \cdots & 1 & 0\\
\vdots & \vdots & \iddots & \vdots & \vdots\\
0 & 1 & \cdots & 0 & 0\\
1 & 0 & \cdots & 0 & 0\\
\end{pmatrix}.
\]
The equalities $A^S=PA^TP$ and $P^2=E$ readily imply that skew transposition enjoys the same product law as the usual transposition: $(AB)^S=B^SA^S$. Since clearly $(A^S)^S=A$, we see that the skew transposition is an involution on the full matrix monoid $M_n(\mathbb{S})$. Although this involution appears in some studies of $M_n(\mathbb{S})$ (see, e.g. \cite{Lee76,Lee23b}), it plays an essentially secondary role there compared with the usual transposition. However, since skew transposition respects triangularity, it becomes the standard involution for semigroups consisting of triangular matrices. In particular, each of the monoids $UT_n(\mathbb{S})$, $U^0T_n(\mathbb{S})$,  $T_n(\mathbb{S})$ is closed under skew transposition.

For a matrix $A$, let $\diag(A)$ denote the diagonal matrix whose diagonal entries are those of $A$. The map $A\mapsto\diag(A)$ is an endomorphism of the monoid $T_n(\mathbb{S})$, and we will use this fact without explicit reference. Moreover, the map is compatible with skew transposition.

\subsection{Vocabulary of equational logic}
\label{subsec:vocabulary}
Even though the notions of identity and identity basis are intuitively clear, precise reasoning about them requires a formal framework. Such a framework, provided by equational logic, is succinctly presented, for example, in \cite[Chapter~II]{BuSa81}. For the reader’s convenience, we briefly review the basic vocabulary of equational logic in a form adapted to the needs of this paper. When doing so, we closely follow~\cite[Section 4.1]{AV20}.

We fix a countably infinite set $\mathcal{X}$ of \emph{variables} and let $\mathcal{X}^+$ denote the \emph{free semigroup} on $\mathcal{X}$, that is, the set of all non-empty (finite) words formed from the variables of $\mathcal{X}$, endowed with the binary operation of concatenation. Throughout the paper, we adopt the convention that $a,b,c,\dots, x,y,z$ (with or without subscripts) stand for variables, while $\mathbf{a}, \mathbf{b}, \mathbf{c}, \dots, \mathbf{x}, \mathbf{y}, \mathbf{z}$ (with or without subscripts) represent words. We occasionally employ the \emph{empty word}; whenever words under consideration are allowed to be empty, we state this explicitly.

If $S$ is a semigroup, any map $\varphi\colon\mathcal{X}\to S$ is called a \emph{substitution}. Each substitution admits a unique extension to a morphism $\mathcal{X}^+\to {S}$, still denoted by $\varphi$: if  $\mathbf{w}=w_1\cdots w_k$, where $w_1,\dots,w_k$ are variables, $\varphi(\mathbf{w}): =\varphi(w_1)\cdots\varphi(w_k)$. A \emph{semigroup identity}, or simply an \emph{identity} is a pair of words $(\mathbf{u},\mathbf{v})\in \mathcal{X}^+\times \mathcal{X}^+$, written as $\mathbf{u}\approx\mathbf{v}$. We say that a semigroup $S$ \emph{satisfies} $\mathbf{u}\approx\mathbf{v}$, or $\mathbf{u}\approx\mathbf{v}$ \emph{holds in} $S$ if $\varphi(\mathbf{u})=\varphi(\mathbf{v})$ for every substitution $\varphi\colon\mathcal{X}\to {S}$. An identity $\mathbf{u}\approx\mathbf{v}$ is \emph{nontrivial} if $\mathbf{u}$ and $\mathbf{v}$ are distinct words.

A unary operation $^*$ on a semigroup $S$ is an \emph{involution} if, for any $a,b\in {S}$,
\[
(ab)^*=b^*a^*\ \mbox{ and }\ (a^*)^*=a.
\]
In this case, $(S,{}^*)$ is called an \emph{involutory semigroup} or a \emph{semigroup with involution}; the semigroup $S$ is then referred to as the \emph{semigroup reduct} of $(S,{}^*)$. In the context of involutory semigroups, it is necessary to replace the free semigroup $\mathcal{X}^+$ by the \emph{free involutory semigroup} which can be realized as follows.
\begin{itemize}
    \item Let $\mathcal{X}^*:=\{x^*\mid x\in\mathcal{X}\}$ be a disjoint copy of $\mathcal{X}$ and form the union $\mathcal{X}\cup \mathcal{X}^*$.
    \item Extend the bijection $x\mapsto x^*$ to a bijection $\mathcal{X}\cup \mathcal{X}^*\to \mathcal{X}\cup \mathcal{X}^*$ by letting $(x^*)^*:=x$ for every $x\in \mathcal{X}$.
    \item Form the free semigroup $(\mathcal{X}\cup \mathcal{X}^*)^+$ on $\mathcal{X}\cup \mathcal{X}^*$.
    \item Let $\mathcal{I}(\mathcal{X}):=\left((\mathcal{X}\cup \mathcal{X}^*)^+,^*\right)$ where the unary operation $^*$ is defined by $y_1\cdots y_n\mapsto y_n^*\cdots y_1^*$ for all $y_i\in\mathcal{X}\cup \mathcal{X}^*$.
\end{itemize}
$\mathcal{I}(\mathcal{X})$ is the \emph{free involutory semigroup} on $\mathcal{X}$, and every map $\varphi\colon\mathcal{X}\to S$, where $(S,{}^*)$ is any involutory semigroup, extends uniquely to a morphism of involutory semigroups. Again, the notions of identity (of involutory semigroups) and of satisfaction of an identity are defined accordingly.

The following concepts and results are entirely parallel for semigroups and involutory semigroups. Hence, we formulate them only for the `plain' semigroup case. Given a semigroup $S$, we denote by $\mathsf{Id}(S)$ the set of all semigroup identities $\mathbf{u}\approx\mathbf{v}$ satisfied by $S$.

Given any collection $\Sigma$ of semigroup identities, we say that an identity   $\mathbf{u}\approx\mathbf{v}$ is a \emph{consequence of} $\Sigma$, or that $\Sigma$ \emph{implies} $\mathbf{u}\approx\mathbf{v}$, if every semigroup satisfying all identities in $\Sigma$ also satisfies $\mathbf{u}\approx \mathbf{v}$. Birkhoff's completeness theorem of equational logic (see \cite[Theorem 14.17]{BuSa81}) shows that this notion (which we have given a semantic definition) can be captured by a very transparent set of inference rules. These rules in fact formalize the most natural things one does with identities: substitution of a word for a variable, application of operations to identities (such as, say, multiplying both sides of an identity  on the right by the same word) and using symmetry and transitivity of equality.

An \emph{identity basis} for a semigroup $S$ is any set $\Sigma\subseteq\mathsf{Id}(S)$ such that every identity of $\mathsf{Id}(S)$ is a consequence of $\Sigma$. A semigroup $S$ is said to be \emph{finitely based} if it possesses a finite identity basis; otherwise $S$ is called \emph{nonfinitely based}.

With the main definitions of being finitely based and nonfinitely based in hand, we could already begin exploring our main theme, the Finite Basis Problem for semigroups of triangular matrices. However, in certain cases one can establish results that are essentially stronger than the mere presence or absence of a finite identity basis. These stronger results involve gradations of finite or infinite basedness, formulated in terms of semigroup classes called \emph{varieties}.

The class of all semigroups satisfying all identities from a given set $\Sigma$ of semigroup identities is called the \emph{variety defined by $\Sigma$} and is denoted by  $\var\Sigma$. It is easy to see that the satisfaction of an identity is inherited by forming direct products, taking subsemigroups and morphic images so that each variety is closed under these operators. In fact, varieties can be characterized by this closure property (the HSP-theorem, see \cite[Theorem 11.9]{BuSa81}).

Semigroup varieties are partially ordered under class containment (and even form a lattice that is the subject of extensive investigation; see the surveys \cite{SVV09,GLV22}). Given a semigroup $S$, the variety defined by $\mathsf{Id}(S)$ is called the \emph{variety generated by $S$} (it is the smallest variety that contains $S$); we denote this variety by $\var S$. Similarly, for a class $\mathbf{C}$ of semigroups, $\var\mathbf{C}$ stands for the \emph{variety generated by $\mathbf{C}$}, that is, the smallest variety that contains  $\mathbf{C}$.

 A variety $\mathbf{V}$ is said to be \emph{finitely based} if $\mathbf{V}=\var\Sigma$ for some finite set $\Sigma$ of identities; otherwise it is called \emph{nonfinitely based}. Observe that a semigroup and the variety it generates are either both finitely based or both nonfinitely based.

A variety $\mathbf{V}$ is \emph{hereditarily finitely based} if every semigroup contained in $\mathbf{V}$ is finitely based. We say that a semigroup $S$ is \emph{hereditarily finitely based} if the variety $\var S$ has this property. This is the strengthening of finite basedness that will be encountered in this paper.

The definitions of the two strengthenings of infinite basedness that we meet below involve two important properties of varieties. A variety $\mathbf{V}$ is said to be \emph{finitely generated} if $\mathbf{V}=\var S$ for some finite semigroup $S$, and \emph{locally finite} if all finitely generated semigroups in $\mathbf{V}$ are finite.

Now, a finitely generated variety $\mathbf{V}$ is said to be \emph{strongly nonfinitely based} if $\mathbf{V}$ is not contained in any finitely based, finitely generated variety. Similarly, a locally finite variety $\mathbf{V}$ is \emph{inherently nonfinitely based} if $\mathbf{V}$ is not contained in any finitely based, locally finite variety. Each of these two properties is stronger (in fact, much stronger) than that of being nonfinitely based as it is quite `contagious'. Indeed, if $\mathbf{V}$ is strongly or inherently nonfinitely based, then so are all finitely generated (respectively, locally finite) varieties containing $\mathbf{V}$. A semigroup $S$ is \emph{strongly} or \emph{inherently nonfinitely based} if the variety $\var S$ has this property.

On several occasions, we will use the following classical fact due to Garret Birkhoff:
\begin{prop}\cite[Theorem II.10.16]{BuSa81}\label{prop:birkhoff}
Every finitely generated variety is locally finite.
\end{prop}

By Proposition~\ref{prop:birkhoff}, if a finite semigroup is inherently nonfinitely based, then it is also strongly nonfinitely based. The question of whether the converse is true was posed in \cite{Vol00,Vol01} and remained open for more than 20 years. Only recently have examples of strongly nonfinitely based  but not inherently nonfinitely based finite semigroups been found~\cite{GSV25}; one of these examples is the monoid of all $4\times4$ triangular matrices over the 2-element field~\cite[Theorem 11]{GSV25}.

The final concept we need is that of equational equivalence. Two semigroups are \emph{equationally equivalent} if they satisfy the same identities, or equivalently, they generate the same variety. Clearly, equationally equivalent semigroups behave the same with respect to the Finite Basis Problem: if one of them is finitely based, so is the other, and so on.

We conclude with a minor but, we believe, necessary clarification. All the semigroups whose Finite Basis Problem we study in this paper are in fact monoids, and we consistently refer to them as such. However, since we do not include the nullary operation corresponding to the neutral element in the signature, we consider them in the semigroup signature and hence study their \textbf{semigroup} identities rather than their monoid identities. This choice does not affect the answer to the Finite Basis Problem: indeed, it can easily be verified that a monoid has a finite basis of monoid identities if and only if its semigroup identities admit a finite basis. Some of the monoids under consideration are groups; again, we study their \textbf{semigroup} identities rather than their group identities. For groups, however, the answer to the Finite Basis Problem may depend essentially on the convention we adopt.

\section{Laws involving multiplication}
\label{sec:plain}
\subsection{Summary of known results and our contribution}
\label{subsec:summary}
To place our results in a proper context, we have composed Table~\ref{tab:plainstatus}, which provides a concise overview of the known facts about the Finite Basis Problem for several matrix monoids. It also incorporates some of the findings of the present paper (highlighted in blue), namely those that deal with matrix monoids defined over an arbitrary semiring.

The rows of Table~\ref{tab:plainstatus} are labeled by the underlying semirings $\mathbb{S}$ while the columns correspond to the matrix monoids $T_n(\mathbb{S})$, $U^0T_n(\mathbb{S})$, and $UT_n(\mathbb{S})$; for the sake of completeness, we have also included a column for the full matrix monoid $M_n(\mathbb{S})$. We have merged cells in a row labeled by a semiring whenever the corresponding matrix monoids over this semiring coincide.

In Table \ref{tab:plainstatus}, we always refer to the strongest result currently known (for example, if a monoid was shown to be NFB in~[X] and later proved to be SNFB in~[Y], we cite only~[Y]).

The reader will notice empty cells as well as some cells containing incomplete information when the results gathered in a cell do not cover all possible monoids corresponding to it. This means that, to the best of our knowledge, the Finite Basis Problem is open for monoids about which the table provides no information. An explicit list of open problems is given in Section~\ref{subsec:plainopen}.

\begin{landscape}
\begin{table}
\caption{Current status of the Finite Basis Problem for the monoids $M_n(\mathbb{S})$, $T_n(\mathbb{S})$, $U^0T_n(\mathbb{S})$, and $UT_n(\mathbb{S})$. Acronyms: FB = finitely based; HFB = hereditarily finitely based; NFB = nonfinitely based; SNFB = strongly nonfinitely based; INFB = inherently nonfinitely based. Footnote marks 1)--6) point to comments on the next page.}
    \label{tab:plainstatus}
 {\centering
    \begin{tabular}{|c|c|c|c|c|}
   \hline
   \rule[-8pt]{0pt}{20pt} Semiring $\mathbb{S}$ & $M_n(\mathbb{S})$ & $T_n(\mathbb{S})$ & $U^0T_n(\mathbb{S})$ & $UT_n(\mathbb{S})$\\
   \hline
    \rule[-2pt]{0pt}{16pt}  Boolean  semiring $\mathbb{B}$  & INFB\cite{Sap87a}$^{1)}$ & \multicolumn{2}{c|}{INFB for $n\ge4$ \cite{VG04} and $n=3$ \cite{LL11}} & SNFB for $n\ge5$ \cite{Vol04,GSV25}$^{2)}$\\
    \rule[-8pt]{0pt}{16pt}  &  & \multicolumn{2}{c|}{FB for $n=2$ \cite{LL11}} & FB for $n\le4$ \cite{Vol04}\\
   \hline
    \rule[-8pt]{0pt}{20pt} 2-element field $\mathbb{F}_2$ & INFB \cite{Sap87a} & \multicolumn{2}{c|}{SNFB for $n\ge4$ \cite{GSV25}; HFB for $n=2$ \cite{ZLL13}} & HFB$^{3)}$\\
    \hline
    \rule[-2pt]{0pt}{16pt} Finite field &  INFB \cite{Sap87a} & INFB  & SNFB for $n\ge4$ ({\blue Theorem~\ref{thm:strongU0T})} &  HFB$^{3)}$\\
    \rule[-8pt]{0pt}{16pt} with $\ge3$ elements &  & for $n\ge4$ \cite{VG03} & \rule[-4pt]{0pt}{16pt} FB for $n=2$ \cite{CHL16} & \\
    \hline
    \rule[-2pt]{0pt}{16pt}  Infinite field of &  FB$^{4)}$  & FB$^{4)}$ &  FB for $n=2$ \cite{CHL16} & HFB{}$^{3)}$\\
    \rule[-8pt]{0pt}{16pt} finite characteristic & & & & \\
    \hline
    \rule[-2pt]{0pt}{16pt} $\mathbb{Z}$ or field of &  FB$^{4)}$  & FB$^{4)}$ & NFB for $n=3$ \cite{Vol15} & NFB for $n\ge3$\\
    \rule[-8pt]{0pt}{16pt} characteristic 0 & &  &  FB for $n=2$ \cite{CHL16} & FB for $n=2$ \cite{Sap87a}$^{5)}$\\
    \hline
    \rule[-2pt]{0pt}{16pt} Tropical semiring $\mathbb{T}$  & & NFB for $n=2$ \cite{CHLS16} & INFB for $n\ge3$ & SNFB for $n\ge5$ \cite{JF19,GSV25}\\
    \rule[-8pt]{0pt}{16pt}   &  & and $n=3$ \cite{HZL21} & FB for $n=2$ ({\blue Corollary~\ref{cor:inherentU0T}}) & FB for $n\le4$ \cite{JF19}$^{6)}$\\
    \hline
    \rule{0pt}{16pt} Bounded & INFB & INFB for $n\ge3$ & INFB for $n\ge3$ & SNFB for $n\ge5$ \cite{JF19,GSV25} \\
    \rule[-4pt]{0pt}{14pt} distributive  lattice & ({\blue Corollary~\ref{cor:lattice}})  &  FB for $n=2$ & FB for $n=2$ ({\blue Corollary~\ref{cor:inherentU0T}}) & FB for $n\le4$ \cite{JF19}$^{6)}$\\
    \rule[-8pt]{0pt}{18pt} & &({\blue Corollary~\ref{cor:trianglelattice}})&&\\
    \hline
    \rule{0pt}{16pt} Semiring with & & &  INFB for $n\ge3$, & SNFB for $n\ge5$ \cite{JF19,GSV25} \\
    \rule[-8pt]{0pt}{18pt} idempotent addition & & & FB for $n=2$ ({\blue Corollary~\ref{cor:inherentU0T}}) & FB for $n\le4$ \cite{JF19}$^{6)}$\\
    \hline
    \end{tabular}\par}
\end{table}
\end{landscape}

\begin{minipage}{0.96\linewidth}
\small
{}$^{1)}$ That the monoid $M_n(\mathbb{B})$ is inherently nonfinitely based is not explicitly stated in \cite{Sap87a}, but the Brandt monoid $B_2^1$ embeds into $M_2(\mathbb{B})$ whence $\var B_2^1\subseteq\var M_2(\mathbb{B})\subseteq\var M_n(\mathbb{B})$. Since the monoid $B_2^1$ is inherently nonfinitely based \cite[Corollary 6.1]{Sap87a}, so is the monoid $M_n(\mathbb{B})$.

\smallskip

{}$^{2)}$ In \cite{Vol04}, it is shown that the monoid $UT_n(\mathbb{B})$ is equationally equivalent to the \emph{Catalan monoid} $C_n$, that is, the monoid of all order preserving and extensive selfmaps on the chain $1<2<\dots<n$, and for $n\ge 5$, the latter monoid is strongly nonfinitely based by~\cite[Theorem~8]{GSV25} combined with~\cite[Proposition~2.7]{SV23}.

\smallskip

{}$^{3)}$ For every field $F$ of finite characteristic $p$, the monoid $UT_n(F)$ is a nilpotent group of exponent $p^{\lceil\log_p n\rceil}$. For any group $(G,\cdot,{}^{-1})$ of exponent $e>1$, every basis $\Gamma$ of its group identities can be converted into a basis for the semigroup identities of the semigroup $G$ by replacing, for every variable $x$, each occurrence of $x^{-1}$ with $x^{e-1}$ in all identities in $\Gamma$ and adding the identities $x^ey\approx y\approx yx^e$. Clearly, the resulting identity basis for $G$ is finite whenever $\Gamma$ is finite. The group identities of every nilpotent group have a finite basis by Roger Lyndon's theorem~\cite[Theorem 34.14]{Neu67}. Since the variety $\var UT_n(F)$ consists of  nilpotent groups of exponent dividing $p^{\lceil\log_p n \rceil}$, it is hereditarily finitely based.

\smallskip

{}$^{4)}$ That all identities of the monoid $M_n(R)$, where $R$ is $\mathbb{Z}$ or an infinite field of arbitrary characteristic, follow from the associative law is stated in \cite[Lemma~2]{GM78}. The proof of \cite[Lemma~2]{GM78} operates with triangular matrices and therefore applies to the monoid $T_n(R)$ as well. Explicitly, the fact that all identities of $T_2(R)$ follow from the associative law is stated in~\cite[Theorem 4.2]{Kam22}.

\smallskip

$^{5)}$ A nilpotent group has a finite basis for its semigroup identities if and only if it is either abelian or of finite exponent \cite[Proposition 5]{Sap87a}. The monoid $UT_n(R)$, where $R$ is $\mathbb{Z}$ or a field of  characteristic 0, is a nilpotent group of infinite exponent and is non-abelian for $n\ge3$.

\smallskip

{}$^{6)}$ By \cite[Corollary~3.4]{JF19}, for every nontrivial semiring $\mathbb{S}$ with idempotent addition (in particular, for the tropical semiring or a bounded distributive lattice), the monoid $UT_n(\mathbb{S})$ is equationally equivalent to the Catalan monoid $C_n$, which is strongly nonfinitely based for $n\ge 5$ and finitely based for $n\le4$; see Comment {}$^{2)}$ above. We note that in \cite{JF19}, semirings are generally assumed to have commutative multiplication. While this assumption is essential for some results in \cite{JF19}, a closer inspection of the arguments leading to Corollary~3.4 shows that it is unnecessary for the conclusion to hold.
\end{minipage}

\bigskip

\subsection{Proofs}
\subsubsection{Matrices over finite semirings}
\label{subsubsec:finite fields}
We begin with a result that applies to matrices over arbitrary finite semirings.\footnote{The authors are grateful to Marianne Johnson for drawing their attention to the fact that the argument applied to matrices over finite fields in an earlier version of the present paper works equally well in this general situation.} For the 2-element field $\mathbb{F}_2$, this result was established in~\cite[Theorem 11]{GSV25}. The proof in the general case exploits the same idea as in that special case. However, here we depart from a strongly nonfinitely based monoid that is much smaller than the one used in \cite{GSV25} and therefore appears to be of independent interest.
\begin{thm}\label{thm:strongU0T}
Let $\mathbb{S}$ be a finite semiring. For each $n\ge 4$, the monoid $U^0T_n(\mathbb{S})$ is strongly nonfinitely based.
\end{thm}

Consider the following eight matrices from $U^0T_4(\mathbb{S})$:

\begin{equation}
\label{eq:G4}
\begin{tabular}{cccc}
$\left(\begin{matrix} 1&0&0&0\\0&1&0&0\\0&0&0&0\\0&0&0&1 \end{matrix}\right)$,
&
$\left(\begin{matrix} 1&0&0&0\\0&0&1&0\\0&0&0&0\\0&0&0&1 \end{matrix}\right)$,
&
$\left(\begin{matrix} 1&0&0&0\\0&0&0&0\\0&0&0&0\\0&0&0&1 \end{matrix}\right)$,
&
$\left(\begin{matrix} 0&1&0&0\\0&0&0&0\\0&0&0&0\\0&0&0&0 \end{matrix}\right)$,
\\
\rule{0cm}{.3cm}$e$&$a$&$a^2$&$b$\\[5pt]
$\left(\begin{matrix} 0&0&1&0\\0&0&0&0\\0&0&0&0\\0&0&0&0 \end{matrix}\right)$,
&
$\left(\begin{matrix} 0&0&0&0\\0&0&0&0\\0&0&0&1\\0&0&0&0 \end{matrix}\right)$,
&
$\left(\begin{matrix} 0&0&0&0\\0&0&0&1\\0&0&0&0\\0&0&0&0 \end{matrix}\right)$,
&
$\left(\begin{matrix} 0&0&0&1\\0&0&0&0\\0&0&0&0\\0&0&0&0 \end{matrix}\right)$.
\\
\rule{0cm}{.3cm}$ba$&$c$&$ac$&$bac$
\end{tabular}
\end{equation}
It is easy to verify that, together with the  $4\times4$ identity matrix $E$ and the $4\times4$ zero  matrix $O$, the eight matrices in~\eqref{eq:G4} form a submonoid of $U^0T_4(\mathbb{S})$; we denote this monoid by $G_4$.

The definition of strong nonfinite basedness ensures that a finite monoid is strongly nonfinitely based whenever it possesses a strongly nonfinitely based subsemigroup. Since the monoid $U^0T_4(\mathbb{S})$ naturally embeds into $U^0T_n(\mathbb{S})$ for each $n\ge 4$, it suffices, in order to obtain Theorem~\ref{thm:strongU0T}, to prove the following:

\begin{prop}\label{prop:G4}
The monoid $G_4$ is strongly nonfinitely based.
\end{prop}

The proof of Proposition~\ref{prop:G4} relies on a sufficient condition for strong nonfinite basedness established in~\cite{GSV25}. To state this condition, we need a few notions concerning semigroup words.

The set of all variables occurring in a word $\mathbf w$ is called the \emph{content} of $\mathbf w$ and is denoted by $\con(\mathbf w)$. The content of the empty word is defined as the empty set.

A variable $x\in\con(\mathbf w)$ is said to be \emph{linear} if $x$ occurs exactly once in $\mathbf w$ and \emph{repeated} otherwise. A word $\mathbf w$ is said to be \emph{sparse} if every two occurrences of a repeated variable in $\mathbf w$ sandwich some linear variable. A word $\mathbf u$ is called an \emph{isoterm for} a semigroup $S$ if the only word $\mathbf v$ such that $S$ satisfies the identity $\mathbf u\approx\mathbf v$ is the word $\mathbf u$ itself.

\begin{prop}[{\cite[Remark 9]{GSV25}}]
\label{prop:sparse}
If every sparse word is an isoterm for a finite monoid, then the monoid is strongly nonfinitely based.
\end{prop}

\begin{proof}[Proof of Proposition~\ref{prop:G4}]
We will show that every sparse word $\mathbf{u}$ is an isoterm for the monoid $G_4$; the claim will then follow from Proposition~\ref{prop:sparse}.

Thus, suppose that $G_4$ satisfies $\mathbf{u}\approx\mathbf{v}$ for some word $\mathbf{v}$. We have to prove that $\mathbf{v}=\mathbf{u}$.

We proceed step by step, establishing increasingly many similarities between $\mathbf{u}$ and $\mathbf{v}$ until eventually showing that these words coincide. In doing so, we argue by contraposition: assuming that the claim at a given step is false, we provide a substitution $\varphi\colon\mathcal{X}\to G_4$ such that $\varphi(\mathbf{u})\ne\varphi(\mathbf{v})$, thereby contradicting the assumption that $G_4$ satisfies $\mathbf{u}\approx\mathbf{v}$.

To present substitutions compactly, we assign a name to each matrix in~\eqref{eq:G4} (displayed beneath the corresponding matrix). The names are chosen to be consistent with matrix multiplication: the matrix denoted by $a^2$ is indeed the square of the one denoted by $a$, the matrix denoted by $bc$ is indeed the product of the matrices denoted by $b$ and $c$, and so on. In this notation, the following equalities are easily seen to hold in the monoid $G_4$:
\[
ea=a,\ ae=a^2,\ eb=ab=be=b,\ ce=ca=c,\ b^2=c^2=bc=cb=ba^2=a^2c=ec=O.
\]

For compactness, we adopt the convention that any variable whose image under a substitution is not explicitly specified is mapped by that substitution to the identity matrix $E$.

\begin{step}
$\con(\mathbf u)=\con(\mathbf v)$.
\end{step}

\begin{proof}
If a variable $x$ occurs in only one of the words $\mathbf{u}$ and $\mathbf{v}$, then under the substitution $\varphi$ mapping $x$ to the zero matrix $O$, the image of the word in which $x$ occurs is $O$, while the image of the other word is $E$. Hence,  $\varphi(\mathbf{u})\ne\varphi(\mathbf{v})$.
\end{proof}

\begin{step}
If a variable is linear in $\mathbf{u}$, then it is linear in $\mathbf{v}$, and vice versa.
\end{step}

\begin{proof}
If a variable $x$ occurs only once in one of the words $\mathbf{u}$ and $\mathbf{v}$ and more than once in the other, then under the substitution $\varphi$ mapping $x$ to the matrix $b$, the image of the word in which $x$ occurs only once is $b$ while the image of the other is $O$ since $b^2=O$. Hence, $\varphi(\mathbf{u})\ne\varphi(\mathbf{v})$.
\end{proof}

\begin{step}
The linear variables appear in $\mathbf{u}$ and $\mathbf{v}$ in the same order.
\end{step}

\begin{proof}
If the claim is false, there are linear variables $t_1$ and $t_2$ such that $t_1$ precedes $t_2$ in $\mathbf{u}$, but not in $\mathbf{v}$. Let   the substitution $\varphi$ be defined by $\varphi(t_1):=ba$ and $\varphi(t_2):=c$. Then $\varphi(\mathbf{u})=bac$, but $\varphi(\mathbf{v})=O$ since $cb=0$. Hence, $\varphi(\mathbf{u})\ne\varphi(\mathbf{v})$.
\end{proof}

If all variables in $\mathbf{u}$ are linear, the equality $\mathbf{v}=\mathbf{u}$ follows from Steps 1--3. Otherwise, we decompose $\mathbf{u}$ and $\mathbf{v}$ as follows:
\begin{equation}\label{eq:dec}
\begin{split}
\mathbf{u} = \mathbf{a}_0t_1 \mathbf{a}_1t_2 \cdots t_{k-1}\mathbf{a}_{k-1} t_k \mathbf{a}_k,\\
\mathbf{v} = \mathbf{b}_0t_1 \mathbf{b}_1t_2 \cdots t_{k-1}\mathbf{b}_{k-1} t_k \mathbf{b}_k,
\end{split}
\end{equation}
where $t_1, \dots, t_k$ are the linear variables of $\mathbf{u}$ and $\mathbf{v}$ and each of the words $\mathbf{a}_0, \mathbf{a}_1, \dots, \mathbf{a}_{k-1}, \mathbf{a}_k$ and $\mathbf{b}_0, \mathbf{b}_1, \dots, \mathbf{b}_{k-1}, \mathbf{b}_k$ either is empty or involves only variables that are repeated in $\mathbf{u}$ and $\mathbf{v}$. It remains to show that $\mathbf{a}_i = \mathbf{b}_i$ for all $i=0,1,\dots,k-1,k$. In the next steps, we consider an arbitrary index $i$.

\begin{step}
$\con(\mathbf{a}_i) = \con(\mathbf{b}_i)$.
\end{step}

\begin{proof}
Assume that there exists a variable $z$ that occurs in $\mathbf{a}_i$ but not in $\mathbf{b}_i$. Let the substitution $\varphi$ be defined by $\varphi(z):=a$ and
\[
\begin{cases}
\varphi(t_1):=ac&\text{if }\ i=0,\\
\varphi(t_i):=ba,\  \varphi(t_{i+1}):=c&\text{if }\ 0<i<k,\\
\varphi(t_k):=ba&\text{if }\ i=k.
\end{cases}
\]
Then $\varphi(\mathbf{u})=O$ because $ba^2=a^2c=O$. On the other hand, using the equalities $ab=b$ and $ca=c$, we obtain
\[
\varphi(\mathbf{v})=\begin{cases}
ac&\text{if }\ i=0,\\
bac&\text{if }\ 0<i<k,\\
ba&\text{if }\ i=k.
\end{cases}
\]
Hence, $\varphi(\mathbf{u})\ne\varphi(\mathbf{v})$.

The same argument shows that if a variable occurs in $\mathbf{b}_i$, then it must also occur in $\mathbf{a}_i$.
\end{proof}

\begin{step}
Every variable occurs in each of the words $\mathbf{a}_i$ and $\mathbf{b}_i$ at most once.
\end{step}

\begin{proof}
For $\mathbf{a}_i$, this follows from the condition that the word $\mathbf{u}$ is sparse, combined with the fact that all the variables that may occur in $\mathbf{a}_i$ are repeated in $\mathbf{u}$.

Now suppose that there exists a variable $z$ that occurs in $\mathbf{b}_i$ more than once. By Step 4 and the observation of the preceding paragraph, $z$ also occurs in $\mathbf{a}_i$ but exactly once. Let the substitution $\varphi$ be defined by $\varphi(z):=a$ and
\[
\begin{cases}
\varphi(t_1):=c&\text{if }\ i=0,\\
\varphi(t_i):=b,\  \varphi(t_{i+1}):=c&\text{if }\ 0<i<k,\\
\varphi(t_k):=b&\text{if }\ i=k.
\end{cases}
\]
Then $\varphi(\mathbf{v})=O$ because $ba^2=a^2c=O$. On the other hand, using the equalities $ab=b$ and $ca=c$, we obtain
\[
\varphi(\mathbf{u})=\begin{cases}
ac&\text{if }\ i=0,\\
bac&\text{if }\ 0<i<k,\\
ba&\text{if }\ i=k.
\end{cases}
\]
Hence, $\varphi(\mathbf{u})\ne\varphi(\mathbf{v})$.
\end{proof}

\begin{step}
$\mathbf{a}_i = \mathbf{b}_i$.
\end{step}

\begin{proof}
Taking into account Steps 4 and 5, it remains to show that the variables forming the words $\mathbf{a}_i$ and $\mathbf{b}_i$ occur in these words in the same order.

Take any variables $z_1$ and $z_2$ such that $z_1$ precedes $z_2$ in $\mathbf{a}_i$, but not in $\mathbf{b}_i$.  Let the substitution $\varphi$ be defined by $\varphi(z_1):=e$, $\varphi(z_2):=a$ and
\[
\begin{cases}
\varphi(t_1):=c&\text{if }\ i=0,\\
\varphi(t_i):=b,\  \varphi(t_{i+1}):=c&\text{if }\ 0<i<k,\\
\varphi(t_k):=b&\text{if }\ i=k.
\end{cases}
\]
Then $\varphi(\mathbf{v})=O$ because $ae=a^2$, and $ba^2=a^2c=O$. On the other hand, using the equalities $ea=a$, $ab=eb=b$ and $ca=ce=c$, we obtain
\[
\varphi(\mathbf{u})=\begin{cases}
ac&\text{if }\ i=0,\\
bac&\text{if }\ 0<i<k,\\
ba&\text{if }\ i=k.
\end{cases}
\]
Hence,  $\varphi(\mathbf{u})\ne\varphi(\mathbf{v})$.
\end{proof}
From Step 6 and the decompositions \eqref{eq:dec} we obtain the equality $\mathbf{v}=\mathbf{u}$, thus completing the proof of Proposition~\ref{prop:G4}.
\end{proof}

\begin{remks}
1. The 10-element strongly nonfinitely based monoid $G_4$ belongs to a family of small    strongly nonfinitely based monoids found by Sergey Gusev with the help of Mace4 model finder~\cite{Mace4}; see~\cite{GV*} for details. Our proof of strong nonfinite basedness of $G_4$ is different from the proof in~\cite{GV*}.

\smallskip

2. As mentioned, the fact that the monoid $U^0T_n(\mathbb{F}_2)=T_n(\mathbb{F}_2)$ is strongly nonfinitely based for each $n\ge 4$ was established in \cite[Theorem 11] {GSV25}. The proof in~\cite{GSV25} relies on the 42-element monoid $IC_4$ consisting of all order preserving and extensive\footnote{A partial selfmap $\alpha$ on $\{1,2,3,4\}$ is \emph{order preserving} if $i\le j$ implies $\alpha(i)\le\alpha(j)$ for all $i,j$ in the domain of $\alpha$, and \emph{extensive} if $i\le \alpha(i)$ for every $i$ in the domain of $\alpha$.} partial one-to-one selfmaps on the chain $1<2<3<4$; the monoid $IC_4$ is strongly nonfinitely based by \cite[Theorem 8]{GSV25}.

The monoid $G_4$ admits a natural embedding into $IC_4$; Figure~\ref{fig:bijection} shows the selfmaps corresponding to the nonzero matrices of $G_4$ (using the notation introduced in \eqref{eq:G4}), and the zero matrix $O$ corresponds to the nowhere defined map.
\begin{figure}[h]
\centering
\begin{tikzpicture}
[scale=0.9]
\foreach \x in {3.5,5,9,10.5,14.5,16} \foreach \y in {9.5,10.5,11.5,14,15,16,18.5,19.5,20.5,12.5,17,21.5} \filldraw (\x,\y) circle (2pt);

\foreach \x in {3.25,5.25,8.75,10.75,14.25,16.25} \foreach \y in {9.5,14,18.5} \node at (\x,\y) {\tiny 1};
\foreach \x in {3.25,5.25,8.75,10.75,14.25,16.25} \foreach \y in {10.5,15,19.5} \node at (\x,\y) {\tiny 2};
\foreach \x in {3.25,5.25,8.75,10.75,14.25,16.25} \foreach \y in {11.5,16,20.5} \node at (\x,\y) {\tiny 3};
\foreach \x in {3.25,5.25,8.75,10.75,14.25,16.25} \foreach \y in {12.5,17,21.5} \node at (\x,\y) {\tiny 4};

\node[] at (2.3,20) {$E\longmapsto$};
\draw (3.5,18.5) edge[-latex] (5,18.5);
\draw (3.5,19.5) edge[-latex] (5,19.5);
\draw (3.5,20.5) edge[-latex] (5,20.5);
\draw (3.5,21.5) edge[-latex] (5,21.5);

\node[] at (7.8,20) {$e\longmapsto$};
\draw (9,18.5) edge[-latex] (10.5,18.5);
\draw (9,19.5) edge[-latex] (10.5,19.5);
\draw (9,21.5) edge[-latex] (10.5,21.5);

\node[] at (13.3,20) {$a\longmapsto$};
\draw (14.5,18.5) edge[-latex] (16,18.5);
\draw (14.5,19.5) edge[-latex] (16,20.5);
\draw (14.5,21.5) edge[-latex] (16,21.5);

\node[] at (2.3,15.5) {$a^2\longmapsto$};
\draw (3.5,14) edge[-latex] (5,14);
\draw (3.5,17) edge[-latex] (5,17);

\node[] at (7.8,15.5) {$b\longmapsto$};
\draw (9,14) edge[-latex] (10.5,15);

\node[] at (13.3,15.5) {$c\longmapsto$};
\draw (14.5,16) edge[-latex] (16,17);

\node[] at (2.3,11) {$ac\longmapsto$};
\draw (3.5,10.5) edge[-latex] (5,12.5);

\node[] at (7.8,11) {$ba\longmapsto$};
\draw (9,9.5) edge[-latex] (10.5,11.5);

\node[] at (13.3,11) {$bac\longmapsto$};
\draw (14.5,9.5) edge[-latex] (16,12.5);
\end{tikzpicture}
\caption{The embedding of $G_4$ into $IC_4$}\label{fig:bijection}
\end{figure}
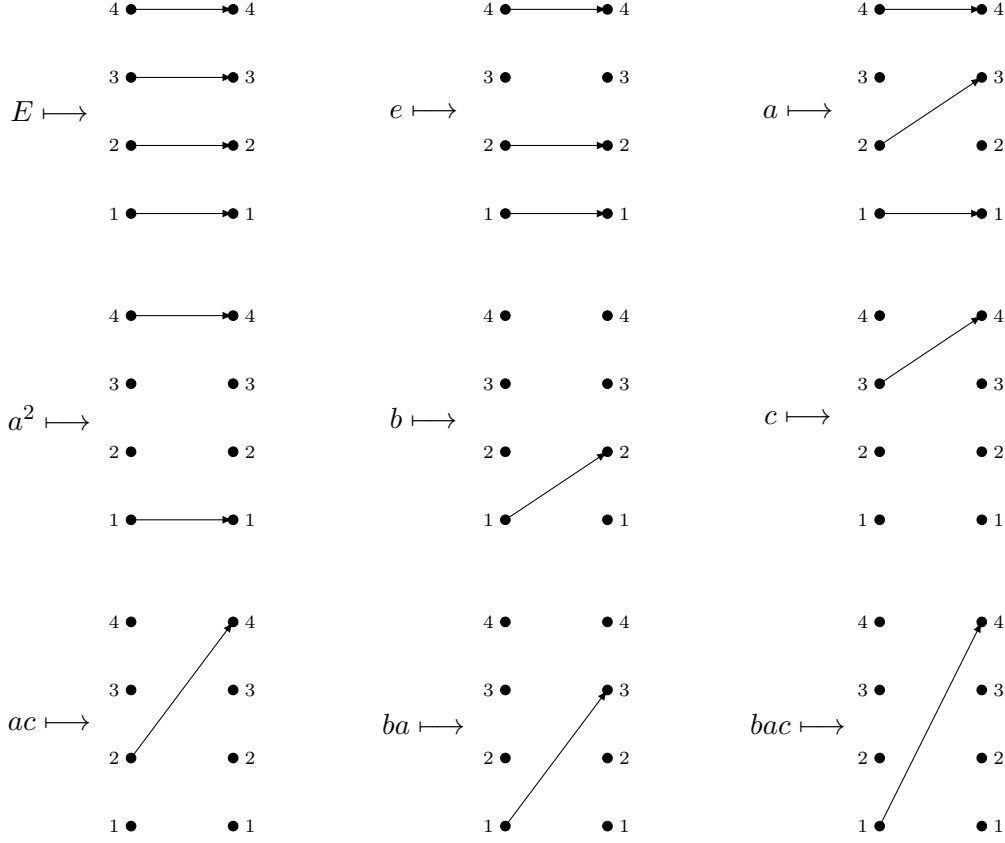

The embedding $G_4\hookrightarrow IC_4$ implies that $\var G_4\subseteq\var IC_4$; moreover, the inclusion is strict, since it is easy to see that the identity $x^2\approx x^3$ holds in $G_4$ but fails in $IC_4$. Thus, the monoid $G_4$ is not only smaller in size than $IC_4$, but also generates a smaller variety. Therefore, one may expect that Proposition~\ref{prop:G4} will find applications beyond the scope of \cite[Theorem 8]{GSV25}.
\end{remks}

In view of Theorem~\ref{thm:strongU0T}, it is natural to ask what can be said about the Finite Basis Problem for the remaining monoids in the header row of Table~\ref{tab:plainstatus}, namely, $M_n(\mathbb{S})$, $T_n(\mathbb{S})$, and $UT_n(\mathbb{S})$, when $\mathbb{S}$ is a finite semiring. The monoid $M_n(\mathbb{S})$ is inherently nonfinitely based for every $n$ and every finite semiring $\mathbb{S}$. Indeed, as mentioned, the Brandt monoid $B_2^1$ embeds into the monoid of all $2\times 2$ matrices over every semiring $\mathbb{S}$ with $0\ne 1$ such that $0+0=0\cdot0=1\cdot0=0\cdot1=0$ and $0+1=1+0=1\cdot1=1$. Since the monoid $B_2^1$ is inherently nonfinitely based \cite[Corollary 6.1]{Sap87a} and $M_2(\mathbb{S})$ naturally embeds into $M_n(\mathbb{S})$ for every $n$, the claim follows.

Since $U^0T_n(\mathbb{S})$ is a submonoid of the monoid $T_n(\mathbb{S})$, the latter is strongly nonfinitely based for $n\ge 4$ by Theorem~\ref{thm:strongU0T}. In fact, under mild conditions, one can extract stronger results from known information about the monoids $T_n(\mathbb{B})$ and $T_n(F)$, where $F$ is a finite field. Roughly speaking, if the ground structure is a `proper' semiring, that is, one far from being a ring, then the argument from \cite{LL11} applies, whereas if it is close to a ring of odd characteristic, then so does the argument from \cite{VG03}.

\begin{prop}\label{prop:Tn(S)summary}
Let $\mathbb{S}$ be a finite semiring and let $R$ be the subsemigroup generated by\/ $1$ in the additive semigroup of $\mathbb{S}$. The monoid $T_n(\mathbb{S})$ is inherently nonfinitely based if either
\begin{itemize}
    \item[{(a)}] $n\ge 3$ and $0\notin R$, or
    \item[{(b)}] $n\ge 4$ with $0\in R$, and $|R|$ divisible by an odd prime.
\end{itemize}
\end{prop}

\begin{proof}
By definition, the set $R$ is closed under addition, but it is easy to see that it is also closed under multiplication.

(a) If $0\notin R$, then $R^0:=R\cup\{0\}$ forms a subsemiring in $\mathbb{S}$, and the onto map $R^0\twoheadrightarrow\mathbb{B}$ defined by $0\mapsto 0$ and $r\mapsto 1$ for $r\in R$ is a semiring morphism.  Applying this morphism entrywise to matrices in $T_n(R^0)$ induces an onto monoid morphism $T_n(R^0)\twoheadrightarrow T_n(\mathbb{B})$. The monoid $T_n(\mathbb{B})$ is inherently nonfinitely based for every $n\ge 3$; this is \cite[Theorem 4.1]{LL11} for $n=3$ and \cite[Theorem 2.1]{VG04} for $n\ge 4$\footnote{Theorem~2.1 in \cite{VG04} erroneously claimed that the monoid $T_3(\mathbb{B})$ is not inherently nonfinitely based; this error was fixed in~\cite{LL11}.}. Thus, for $n\ge 3$, the variety $\var T_n(\mathbb{S})$ contains an inherently nonfinitely based semigroup, and hence the monoid $T_n(\mathbb{S})$ is itself inherently nonfinitely based.

(b) If $0\in R$, then $R$ is easily seen to be a subring of $\mathbb{S}$. If $p$ is an odd prime dividing $|R|$, the quotient ring $R/pR$ is the field with $p$ elements, and $T_n(R/pR)$ is a quotient of the monoid $T_n(R)$. By the main result of \cite{VG03}, the monoid of all triangular $n\times n$ matrices over any field with at least 3 elements is inherently nonfinitely based whenever $n\ge 4$. Thus, for $n\ge 4$, the variety $\var T_n(\mathbb{S})$ contains the inherently nonfinitely based monoid $T_n(R/pR)$, and hence the monoid $T_n(\mathbb{S})$ is itself inherently nonfinitely based.
\end{proof}

Similar observations can be made in relation to the monoid $UT_n(\mathbb{S})$, where $\mathbb{S}$ is a finite semiring. If the subsemigroup generated by $1$ in the additive semigroup of $\mathbb{S}$ does not contain $0$, then the argument from the proof of Proposition~\ref{prop:Tn(S)summary}(a) shows that the variety $\var UT_n(\mathbb{S})$ contains the monoid $T_n(\mathbb{B})$. The latter monoid is strongly nonfinitely based for $n\ge 5$;  see Comment $^{2)}$ after Table~\ref{tab:plainstatus}. Hence, the monoid $UT_n(\mathbb{S})$ with $n\ge 5$ is also strongly nonfinitely based in this case.

We can also observe that if $\mathbb{S}$ is a ring, then the monoid $UT_n(\mathbb{S})$ is a  nilpotent group, and if, in addition, $\mathbb{S}$ is finite, the group will be finite as well. By the reasoning from Comment $^{3)}$ after Table~\ref{tab:plainstatus}, $UT_n(\mathbb{S})$ is hereditarily finitely based in this situation.

Finally, we note that for every (not necessarily finite) semiring $\mathbb{S}$, the monoid $UT_2(\mathbb{S})$ is commutative, and every commutative semigroup is hereditarily finitely based~\cite[Theorem~9]{Per69}.

\subsubsection{Matrices over finite fields}
We turn now to matrices over fields. It follows from \cite[Corollary~2]{VG03} that for any finite field $F$, the variety $\var U^0T_n(F)$ does not contain the Brandt monoid $B_2^1$, and the proof of \cite[Lemma~2]{VG03} shows that all subgroups of the monoid $U^0T_n(F)$ are $p$-groups where $p$ is the characteristic of $F$, and hence are nilpotent. By~\cite[Theorem~2]{Sap87b}, these two properties imply that no monoid of the form $U^0T_n(F)$ is inherently nonfinitely based.

In some sense, however, the monoids $U^0T_n(F)$ with $n\ge4$ are at a very short `edit' distance from being inherently nonfinitely based. Indeed, if one only slightly modifies the condition on the diagonal entries by allowing just one of them to take the values $\pm1$ rather than $0,1$, then one obtains a monoid of the same size as $U^0T_n(F)$ and of a similar structure, but inherently nonfinitely based whenever the characteristic of the field is odd (so that $-1\ne 1$) and $n\ge4$.

To establish this rather unexpected fact, we employ the equational characterization of  inherently nonfinitely based monoids from \cite{Vol00}.

In a finite semigroup $S$, the $|S|!$-th power of every element is an idempotent. Following the standard convention in finite semigroup theory, we write $\omega$ instead of $|S|!$, with the understanding that whenever an expression involving variables and the symbol $\omega$ is evaluated in a finite semigroup $S$, each occurrence of $\omega$ is interpreted by $|S|!$.

Define a sequence of expressions by:
\[
[x,y]_1:=x^{\omega-1}y^{\omega-1}xy,\ \ [x,y]_{k+1}:=[[x,y]_k,y]_1.
\]

\begin{prop}\cite[Proposition 4.4]{Vol00}\footnote{In identity (3) in  \cite{Vol00}, which corresponds to identity \eqref{eq:engel} here, the terms $(\mathbf{e}y\mathbf{e})^{\omega-1} \mathbf{e}y^{\omega+1}\mathbf{e}$ and $\mathbf{e}z\mathbf{e}$ on the left-hand side were erroneously interchanged. This error does not affect the proof of \cite[Proposition 4.4]{Vol00}.}
\label{prop:wfb}
A finite monoid $M$ is inherently nonfinitely based if and only if it violates one of the identities:
\begin{gather}
\bigl((xy)^{\omega}(yx)^{\omega}(xy)^{\omega}\bigr)^{\omega}\approx(xy)^{\omega},\label{eq:ds}\\
[(\mathbf{e}y\mathbf{e})^{\omega-1} \mathbf{e}y^{\omega+1}\mathbf{e},\mathbf{e}z\mathbf{e}]_\omega\approx\mathbf{e},\ \text{where}\ \mathbf{e}:=(xyzt)^{\omega}.\label{eq:engel}
\end{gather}
\end{prop}

For any field $F$ of odd characteristic, denote by $U^{\pm}_{11}T_n(F)$ the submonoid of $T_n(F)$ consisting of matrices whose (1,1) entry is $\pm1$, while all other diagonal entries are 0s or 1s; that is,
    \begin{equation}\label{eq:UpmT}
 U^{\pm}_{11}T_n(F):=\left\{\bigl(\alpha_{ij}\bigr)_{n\times n}\in T_n(F)\;\middle|\; \alpha_{11}=\pm1,\ \alpha_{22}, \dots,\alpha_{nn}\in\{0,1\}\right\}.
    \end{equation}

\begin{thm}\label{thm:pm1}
For any finite field $F$ of odd characteristic and any $n\ge4$, the monoid $U^{\pm}_{11}T_n(F)$ is inherently nonfinitely based.
\end{thm}

\begin{proof}
We will show that the monoid $U^{\pm}_{11}T_4(F)$ violates identity \eqref{eq:engel}. Then Proposition~\ref{prop:wfb} implies that $U^{\pm}_{11}T_4(F)$ is inherently nonfinitely based, and since $U^{\pm}_{11}T_4(F)$ naturally embeds into $U^{\pm}_{11}T_n(F)$ for any $n>4$, the general claim will follow.

Consider the following  $4\times4$ matrices from $U^{\pm}_{11}T_4(F)$:
\[
f:=\left(\begin{matrix} 1&1&0&0\\0&0&0&0\\0&0&0&1\\0&0&0&1 \end{matrix}\right),\ \
a:=\left(\begin{matrix} 1&0&0&0\\0&0&1&0\\0&0&0&0\\0&0&0&1 \end{matrix}\right),\ \
h:=\left(\begin{matrix} -1&-1&0&0\\\phantom{-}0&\phantom{-}0&0&0\\\phantom{-}0&\phantom{-}0&0&1\\\phantom{-}0&\phantom{-}0&0&1 \end{matrix}\right).
\]
It is easy to check that $f=f^2=h^2=(fahf)^2$ and $h=hf=fh=h^3$.

Define a substitution $\varphi\colon\{x,y,z,t\}\to U^{\pm}_{11}T_4(F)$ by letting
\[
\varphi(x)=\varphi(t):=f, \ \ \varphi(y):=a, \ \ \varphi(z):=h.
\]
The value of the right-hand side of \eqref{eq:engel} under this substitution is $\varphi(\mathbf{e})=(fahf)^\omega=f$ since $\omega$ is an even number.

To compute the value of the left-hand side, first notice that $\varphi(\mathbf{e}z\mathbf{e})=fhf=h$. Further, it is easy to see that all powers of the matrix $a$ with exponents greater than one are equal to
\[
a^2=\left(\begin{matrix} 1&0&0&0\\0&0&0&0\\0&0&0&0\\0&0&0&1 \end{matrix}\right),
\]
whence $\varphi(\mathbf{e}y^{\omega+1}\mathbf{e})=fa^2f=f$. Finally, it is easy to compute that
\[
\varphi(\mathbf{e}y\mathbf{e})=faf=\left(\begin{matrix} 1&1&0&1\\0&0&0&0\\0&0&0&1\\0&0&0&1 \end{matrix}\right).
\]

Now consider for each $\gamma\in F$, the matrix
\[
f_\gamma:=\left(\begin{matrix} 1&1&0&\gamma\\0&0&0&0\\0&0&0&1\\0&0&0&1 \end{matrix}\right).
\]
Notice that $faf=f_1$ and $f=f_0$. It is easy to see that $f_\gamma f_\delta=f_{\gamma+\delta}$ for all $\gamma,\delta\in F$. From this, a straightforward induction yields $f_\gamma^k=f_{k\gamma}$ for every positive integer $k$. Hence,
\begin{equation}\label{eq:inversion}
f_\gamma^{\omega-1}=f_{(\omega-1)\gamma}=f_{\omega\gamma-\gamma}=f_{-\gamma},
\end{equation}
since $\omega\gamma=0$, as $\omega$ is a multiple of the characteristic of the field $F$. The equality $h^3=h$ ensures that every odd power of the matrix $h$ equals $h$. In particular, $h^{\omega-1}=h$ since the number $\omega-1$ is odd. Thus, the expression $[f_\gamma,h]_1$ expands as
\[
[f_\gamma,h]_1=f_\gamma^{\omega-1}h^{\omega-1}f_\gamma h=f_{-\gamma}hf_{\gamma}h,
\]
and computing the last product yields $f_{-2\gamma}$. By induction, one computes that
\begin{equation}\label{eq:commutator}
    [f_\gamma,h]_k=f_{(-2)^k\gamma}
\end{equation}
for every positive integer $k$.

We are now ready to compute the value of the left-hand side of~\eqref{eq:engel} under the substitution $\varphi$:
\begin{align*}
    \varphi\bigl([(\mathbf{e}y\mathbf{e})^{\omega-1} \mathbf{e}y^{\omega+1}\mathbf{e},\mathbf{e}z\mathbf{e}]_\omega\bigr)
   &=[f_1^{\omega-1}f,h]_\omega&&\text{since }\ \varphi(\mathbf{e}y\mathbf{e})=f_1,\ \varphi(\mathbf{e}y^{\omega+1}\mathbf{e})=f,\ \varphi(\mathbf{e}z\mathbf{e})=h\\
   &=[f_1^{\omega-1},h]_\omega&&\text{since $f_1f=f_1f_0=f_1$}\\
   &=[f_{-1},h]_\omega&&\text{by \eqref{eq:inversion}}\\
   &=f_{-(-2)^{\omega}}&&\text{by \eqref{eq:commutator}}.
\end{align*}
Since  the characteristic of $F$ is odd, $-(-2)^{\omega}\ne0$, whence the value $f_{-(-2)^{\omega}}$ of the left-hand side of~\eqref{eq:engel} differs from $f_0=f$, which was shown above to be the value of the right-hand side.
\end{proof}

\subsubsection{Matrices over fields of characteristic 0}
\label{subsubsec:infinite fields}

For the monoids $T_n(F)$ and $UT_n(F)$ with $F$ an infinite field or the ring $\mathbb{Z}$ of integers, the Finite Basis Problem falls within the scope of certain general results; see Comments $^{3)}$, $^{4)}$, and $^{5)}$ after Table~\ref{tab:plainstatus}. For the intermediate monoid $U^0T_n(F)$, the problem remains open for $n\ge 4$. For $F$ a field of characteristic~0 or the ring $\mathbb{Z}$, the monoid $U^0T_3(F)$ was shown to be nonfinitely based in~\cite[Theorem~2]{Vol15}\footnote{There is an inaccuracy in the proof of Proposition~4 in~\cite{Vol15}, namely, in Case 3 of the proof. It is claimed there that the subsemigroup $S_{010}:=\left\{\left(\begin{smallmatrix}
0 &  \alpha_{12} & \alpha_{13}\\
0 & 1            & \alpha_{23}\\
0 & 0            & 0
\end{smallmatrix}\right)\;\middle|\; \alpha_{12},\alpha_{13},\alpha_{23}\in\mathbb{R}\right\}$ of $U^0T_3(\mathbb{R})$ satisfies the identity $xyx\approx x$, which is incorrect. However, the semigroup $S_{010}$ does satisfy the weaker identity $xyzxy\approx xy$, and this is fully sufficient for the proof of Proposition~4.}. In~\cite{CHL16}, the monoid $U^0T_2(F)$ was shown to be finitely based for $F$ any field or the ring $\mathbb{Z}$, and an explicit finite basis was provided.

Here we demonstrate that the finite basis property of $U^0T_2(F)$ is, so to speak, fragile: if $F$ is a field of characteristic~0 or the ring $\mathbb{Z}$, the monoid $U^{\pm}_{11}T_2(F)$ turns out to be nonfinitely based. The matrices of the latter monoid (defined in~\eqref{eq:UpmT} above) differ from those in $U^0T_2(F)$ only in their (1,1) entry, which is allowed to take the values $\pm1$ rather than $0,1$.

To prove this observation, we utilize a tool developed by Karl Auinger, Yuzhu Chen, and the present authors in~\cite{ACHLV15}. In order to present it, we need to recall a few concepts.

The first concept is that of the Mal'cev product \cite{Mal67}. The \emph{Mal'cev product} of two classes of \sgps\ $\mathbf{A}$ and $\mathbf{B}$ is the class of all \sgps\ $S$ that admit a congruence $\theta$ such that the quotient \sgp\ $S/\theta$ lies in $\mathbf{B}$ while every $\theta$-class that is a sub\sgp\ in $S$ lies in $\mathbf{A}$.

The next concept we need is that of a Zimin word. Let $x_1,x_2,\dots,x_k,\dots$ be a sequence of distinct variables. The sequence $\{Z_k\}_{k=1,2,\dots}$ of \emph{Zimin words} is defined inductively by
\[
Z_1:=x_1,\quad Z_{k+1}:=Z_kx_{k+1}Z_k.
\]

Finally, recall that a \sgp\ is said to be \emph{periodic} if each of its one-generated subsemigroups is finite and \emph{locally finite} if each of its finitely generated subsemigroups is finite.

Now we state the result of \cite{ACHLV15} in a form that is convenient for the use in the present paper.

\begin{prop}\cite[Theorem~3.4 + Remark~3.5]{ACHLV15}
\label{prop:kauffman}
A semigroup $S$ is nonfinitely based whenever it is contained in the variety generated by the Mal'cev product of a semigroup variety all of whose periodic members are locally finite with a locally finite semigroup variety, and each Zimin word is an isoterm for $S$.
\end{prop}

\begin{prop}\label{prop:2x2pm1}
For any field $F$ of characteristic $0$, the monoid $U^{\pm}_{11}T_2(F)$ is nonfinitely based.
\end{prop}

\begin{proof}
We will verify that $U^{\pm}_{11}T_2(F)$ satisfies the conditions of Proposition~\ref{prop:kauffman}.

Let $\theta$ denote the kernel of the morphism $\diag\colon U^{\pm}_{11}T_2(F)\to D$, where
\[
D:=\left\{\left(\begin{smallmatrix}
1 & 0\\
0 & 1
\end{smallmatrix}\right), \
\left(\begin{smallmatrix}
1 & 0\\
0 & 0
\end{smallmatrix}\right), \
\left(\begin{smallmatrix}
-1 &  0\\
\phantom{-}0 & 1
\end{smallmatrix}\right), \
\left(\begin{smallmatrix}
-1 &  0\\
\phantom{-}0 & 0
\end{smallmatrix}\right)
\right\}
\]
is the submonoid formed by the diagonal matrices from $U^{\pm}_{11}T_2(F)$. Only two $\theta$-classes are subsemigroups in $U^{\pm}_{11}T_2(F)$, namely,
\[
S_{11}:=\left\{\left(\begin{smallmatrix}
1 & \alpha\\
0 & 1
\end{smallmatrix}\right)\;\middle|\; \alpha\in F\right\}\ \text{ and }\
S_{10}:=\left\{\left(\begin{smallmatrix}
1 & \alpha\\
0 & 0
\end{smallmatrix}\right)\;\middle|\; \alpha\in F\right\}.
\]
A direct computation shows that $S_{11}$ satisfies the identity $xy\approx yx$, while $S_{10}$ satisfies the identity $xy\approx y$. Hence both $S_{11}$ and $S_{10}$ satisfy the identity $xyz\approx yxz$. It is well known (and easy to verify) that every periodic semigroup satisfying the latter identity is locally finite. Denote by $\mathbf{U}$ the variety defined by the identity $xyz\approx yxz$ and by $\mathbf{V}$ the variety generated by the semigroup $D$. Then these varieties satisfy the assumptions of Proposition~\ref{prop:kauffman} (here $\mathbf{V}$ is locally finite since it is generated by a finite semigroup; see Proposition~\ref{prop:birkhoff}, and the monoid $U^{\pm}_{11}T_2(F)$ belongs to the Mal'cev product of $\mathbf{U}$ with $\mathbf{V}$.

It remains to verify that each Zimin word is an isoterm for $U^{\pm}_{11}T_2(F)$. Since the characteristic of the field $F$ is $0$, the map $m\mapsto\left(\begin{smallmatrix}
1 & m\\
0 & 1
\end{smallmatrix}\right)$ is an embedding of the additive semigroup $\mathbb{N}$ of positive integers into $U^{\pm}_{11}T_2(F)$. It is well known that every identity $\mathbf{u}\approx\mathbf{v}$ satisfied by $\mathbb{N}$ is \emph{balanced}, that is, every variable occurs equally many times in $\mathbf{u}$ and in $\mathbf{v}$. Therefore if an identity of the form $Z_k\approx\mathbf{w}$ holds in $U^{\pm}_{11}T_2(F)$, it must be balanced. By \cite[Lemma 5.2.1]{Sap87a}, every group satisfying a nontrivial balanced identity of the form $Z_k\approx\mathbf{w}$ is nilpotent. However, the subgroup of $U^{\pm}_{11}T_2(F)$ generated by the matrices $\left(\begin{smallmatrix}
-1 &  0\\
\phantom{-}0 & 1
\end{smallmatrix}\right)$ and $\left(\begin{smallmatrix}
1 & 1\\
0 & 1
\end{smallmatrix}\right)$ is not nilpotent because it has trivial center. Hence every nontrivial identity of the form $Z_k\approx\mathbf{w}$ fails in $U^{\pm}_{11}T_2(F)$.
\end{proof}

Clearly, the argument in the proof of Proposition~\ref{prop:2x2pm1} also applies to the monoid $U^{\pm}_{11}T_2(\mathbb{Z})$.

For any field $F$ of characteristic different from $2$, let $U^{\pm}T_n(F)$ denote the submonoid of $T_n(F)$ consisting of matrices whose diagonal entries are 0s, 1s or --1s; that is,
\begin{equation}\label{eq:UpmTfull}
 U^{\pm}T_n(F):=\left\{\bigl(\alpha_{ij}\bigr)_{n\times n}\in T_n(F)\;\middle|\; \alpha_{ii}\in\{0,\pm1\}\ \text{for}\ i=1,\dots,n\right\}.
\end{equation}

\begin{cor}\label{cor:UpmT2}
For any field $F$ of characteristic $0$, the monoid $U^{\pm}T_2(F)$ is nonfinitely based.
\end{cor}

\begin{proof}
 The proof of Proposition~\ref{prop:2x2pm1} can be easily adapted to work for $U^{\pm}T_2(F)$. First, we extend the submonoid $D$ to the set of all diagonal matrices in $U^{\pm}T_2(F)$. The kernel $\theta$ of the morphism $\left(\begin{smallmatrix}
\alpha_{11} & \alpha_{12}\\0 & \alpha_{22}
\end{smallmatrix}\right)\mapsto\left(\begin{smallmatrix}
\alpha_{11} & 0\\0 & \alpha_{22}
\end{smallmatrix}\right)$ has, in addition to $S_{11}$ and $S_{10}$, two further classes that are subsemigroups, namely,
\[
S_{01}:=\left\{\left(\begin{smallmatrix}
0 & \alpha\\
0 & 1
\end{smallmatrix}\right)\;\middle|\; \alpha\in F\right\}\ \text{ and }\
S_{00}:=\left\{\left(\begin{smallmatrix}
0 & \alpha\\
0 & 0
\end{smallmatrix}\right)\;\middle|\; \alpha\in F\right\}.
\]
It is readily verified that $S_{01}$ satisfies the identity $xy\approx x$, while $S_{00}$ satisfies the identity $xy\approx zt$. Hence all $\theta$-classes that are subsemigroups satisfy the identity $xyzt\approx xzyt$, and every periodic semigroup satisfying this identity is known to be locally finite. Now, if we denote by $\mathbf{U}$ the variety defined by the identity $xyzt\approx xzyt$ and by $\mathbf{V}$ the variety generated by the $9$-element semigroup of all diagonal matrices in $U^{\pm}T_2(F)$, the proof of Proposition~\ref{prop:2x2pm1} applies without further modification.
\end{proof}

It was noticed in \cite[Section 4]{Vol15} that the monoid $U^{\pm}T_3(\mathbb{R})$ is nonfinitely based. In fact, the following general result holds.

\begin{thm}\label{thm:UpmTn}
For any field $F$ of characteristic $0$, the monoid $U^{\pm}T_n(F)$ is nonfinitely based.
\end{thm}

\begin{proof}
Since the case of the monoid $U^{\pm}T_2(F)$ has been dealt in Corollary~\ref{cor:UpmT2}, we assume throughout the proof that $n\ge 3$.

We will use a condition for nonfinite basedness due to the third-named author~\cite{Vol89}. Its formulation involved the 5-element semigroup $A_2$ formed by the  following matrices
\[
\begin{pmatrix} 0 & 0\\ 0 & 0\end{pmatrix},\
\begin{pmatrix} 1 & 0\\ 0 & 0\end{pmatrix},\
\begin{pmatrix} 0 & 1\\ 0 & 0\end{pmatrix},\
\begin{pmatrix} 1 & 0\\ 1 & 0\end{pmatrix},\
\begin{pmatrix} 0 & 1\\ 0 & 1\end{pmatrix}.
\]
\begin{prop}\cite[Theorem on p. 188]{Vol89}\label{prop:old}
Let $S$ be a semigroup such that the semigroup $A_2$ belongs to the variety $\var S$. Assume that there exist a subsemigroup $T$ of $S$ and a group $G$ of exponent $d$ in the variety $\var S$ such that  $s^d\in T$ for all $s\in S$, and the group $G$ does not belong to the variety $\var T$. Then the semigroup $S$ is nonfinitely based.
\end{prop}

We will verify that the conditions of Proposition~\ref{prop:old} hold for the monoid $U^{\pm}T_n(F)$ with $U^0T_n(F)$ in the role of the subsemigroup $T$ and the symmetric group $\mathfrak{S}_3$ of exponent 6 in the role of the group $G$.

\setcounter{step}{0}
\begin{step}
The semigroup $A_2$ belongs to the variety $\var U^{\pm}T_n(F)$.
\end{step}

\begin{proof}
Since the monoid $U^{\pm}T_3(F)$ naturally embeds into $U^{\pm}T_n(F)$ for each $n\ge 3$, it suffices to show that $A_2$ belongs to the variety $\var U^{\pm}T_3(F)$.

Since the field $F$ has characteristic 0, it contains the ring $\mathbb{Z}$. Consider the sub\sgp\ $M$ of $\mathrm{UT}_3(\mathbb{Z})$ generated by the matrices
\[
p:=\begin{pmatrix}
1 & 0 & 0\\
0 & 0 & 0\\
0 & 0 & 1
\end{pmatrix}\quad
\text{and}\quad q:=\begin{pmatrix}
1 & 1 & 0\\
0 & 0 & 1\\
0 & 0 & 1
\end{pmatrix}.
\]
Clearly, for each matrix $(\mu_{ij})\in M$, one has $\mu_{ij}\ge0$ and $\mu_{11}=\mu_{33}=1$, whence the set $M^+$ of all matrices $(\mu_{ij})\in M$ such that $\mu_{13}>0$ forms an ideal of $M$. A straightforward
calculation shows that, besides $p$, and $q$, the only matrices in $M\setminus M^+$ are $pq=\left(\begin{smallmatrix}
1 & 1 & 0 \\
0 & 0 & 0\\
0 & 0 & 1
\end{smallmatrix}\right)$ and
$qp=\left(\begin{smallmatrix}
1 & 0 & 0 \\
0 & 0 & 1\\
0 & 0 & 1
\end{smallmatrix}\right)$.
Consider the following bijection between $M\setminus M^+$ and the set of non-zero matrices in $A_2$:
\[
p\mapsto\left(\begin{smallmatrix} 1 & 0 \\ 1 & 0
\end{smallmatrix}\right),\quad
q\mapsto\left(\begin{smallmatrix} 0 & 1 \\ 0 & 0
\end{smallmatrix}\right),\quad
pq\mapsto\left(\begin{smallmatrix} 0 & 1 \\ 0 & 1
\end{smallmatrix}\right),\quad
qp\mapsto\left(\begin{smallmatrix} 1 & 0 \\ 0 & 0
\end{smallmatrix}\right).
\]
Extending this map to $M$ by sending all elements from $M^+$ to $\left(\begin{smallmatrix} 0 & 0 \\ 0 & 0
\end{smallmatrix}\right)$ yields a morphism from $M$ onto $A_2$. Thus, as a morphic image of a sub\sgp\ of ${U^{\pm}T}_3(F)$, the semigroup $A_2$ belongs to the variety $\var U^{\pm}T_3(F)$.
\end{proof}

\begin{step}
The symmetric group $\mathfrak{S}_3$ belongs to the variety $\var U^{\pm}T_n(F)$.
\end{step}

\begin{proof}
In the proof of Proposition~\ref{prop:2x2pm1}, we have already used the matrices $\left(\begin{smallmatrix}
-1 &  0\\
\phantom{-}0 & 1
\end{smallmatrix}\right)$ and $\left(\begin{smallmatrix}
1 & 1\\
0 & 1
\end{smallmatrix}\right)$. Their images under the morphism $U^{\pm}T_2(\mathbb{Z})\twoheadrightarrow U^{\pm}T_2(\mathbb{Z}/3\mathbb{Z})$ that extends the natural morphism $\mathbb{Z}\twoheadrightarrow \mathbb{Z}/3\mathbb{Z}$ generate in $U^{\pm}T_2(\mathbb{Z}/3\mathbb{Z})$ a subgroup isomorphic to $\mathfrak{S}_3$. Thus, $\mathfrak{S}_3$ belongs even to the variety $\var U^{\pm}T_2(F)$, and hence, to the variety $\var U^{\pm}T_n(F)$ for each $n$.
\end{proof}

\begin{step}
$a^6\in U^0T_n(F)$ for all matrices $a\in U^{\pm}T_n(F)$.
\end{step}

\begin{proof}
This readily follows from the definitions of $U^{\pm}T_n(F)$ and $U^0T_n(F)$. In fact, even the square of every triangular matrix with diagonal entries in $\{0,\pm1\}$ is a triangular matrix with diagonal entries in $\{0,1\}$.
\end{proof}

\begin{step}
The symmetric group $\mathfrak{S}_3$ does not belong to the variety $\var U^0T_n(F)$.
\end{step}

To accomplish this, the most involved step, we describe a procedure for constructing a semigroup identity that holds in the monoid $U^0T_n(F)$ for an arbitrary field $F$ and forces any group satisfying it to be nilpotent. The claim will then follow from the known fact that the group $\mathfrak{S}_3$ is not nilpotent.

We will need a few facts about matrix semigroups over arbitrary fields.

\begin{lem}\label{lem:nilpotent}
Every subgroup of the monoid $U^0T_n(F)$ is nilpotent of class at most $n-1$.
\end{lem}

\begin{proof}
The proof that the group of units of $U^0T_n(F)$, namely, the unitriangular group $UT_n(F)$, is nilpotent of class $n-1$ can be found in many textbooks, and the case of an arbitrary subgroup can be treated in a similar way. We present the argument for completeness.

So, let $H$ be a subgroup of the monoid $U^0T_n(F)$ and let $e$ be the identity element of $H$. Then for an arbitrary matrix $g\in H$, we have $g=eg$ and $e=gh$ where $h$ is the inverse of $g$ in $H$. Therefore $\diag(g)=\diag(e)\diag(g)$ and $\diag(e)=\diag(g)\diag(h)$. Since diagonal entries of $e$, $g$, and $h$ are all equal to 0 or 1, these equalities imply that $\diag(e)=\diag(g)$. Hence $H-e:=\{g-e\mid g\in H\}$ consists of matrices with zero main diagonal. The set $R$ of all matrices from  $eT_n(F)e$ with zero main diagonal is a ring under matrix addition and multiplication and $R^n=\{O\}$, where $O$ is the $n\times n$ zero matrix. We have $H=e+(H-e)\subseteq e+R$, and $e+R$ is a group under matrix multiplication that is nilpotent of class $n-1$ by \cite[Section 2.2.2]{CMZ17}.
\end{proof}

An \emph{epigroup} is a semigroup in which a power of every element belongs to a subgroup.

\begin{lem}\label{lem:eprigroup}
The monoid $U^0T_n(F)$ is an epigroup.
\end{lem}

\begin{proof}
This can be verified directly (using the Fitting decomposition), but one may also invoke the following general result:
\begin{lem}\label{lem:putcha}\cite[Theorem~1.4]{Okn98}
Let $F$ be an arbitrary field. A subsemigroup $S$ of the monoid $M_n(F)$ is an epigroup\footnote{In the monograph \cite{Okn98}, epigroups are  called $\pi$-\emph{regular semigroups}.} whenever there is a family of polynomials $P_k(x_{ij})$ in the  variables $x_{ij}$, $i,j=1,\dots,n$, such that a matrix $\bigl(\alpha_{ij}\bigr)_{n\times n}$ belongs to $S$ if and only if $P_k(\alpha_{ij})=0$ for every $k$.
\end{lem}
Lemma~\ref{lem:putcha} applies to the monoid $U^0T_n(F)$  by taking the family consisting of the polynomials $x_{ij}$ for all $1\le j<i\le n$, together with $x^2_{ii}-x_{ii}$ for each $i=1,\dots,n$.
\end{proof}

A semigroup is said to have \emph{$\mathfrak{J}$-height less than} $K$ if every chain
\[
I_1\subseteq I_2\subseteq\dots\subseteq I_K
\]
of its principal ideals contains two equal ideals.

\begin{lem}\label{lem:okninski}\cite[Corollary 3.9(3)]{Okn98}
Every subepigroup of the monoid $M_n(F)$ has $\mathfrak{J}$-height less than $N:=2^{2n}\prod_{i=1}^{n}\binom{n}{i}^2$.
\end{lem}

We can now start constructing a semigroup identity with the required properties.

Let $x_1,x_2,\dots,x_k,\dots$ be a sequence of distinct variables. Following Mal’cev \cite{Mal53}, we inductively define two sequences of words $\{\mathbf{u}_k\}_{k=1,2,\dots}$ and $\{\mathbf{u}'_k\}_{k=1,2,\dots}$ by
\[
\mathbf{u}_1:=x_1,\quad \mathbf{u}'_1:=x_2,\quad \mathbf{u}_{k+1}:=\mathbf{u}_kx_{k+2}\mathbf{u}'_k,\quad  \mathbf{u}'_{k+1}:=\mathbf{u}'_kx_{k+2}\mathbf{u}_k.
\]
\begin{lem}\label{lem:malcev}\cite[Theorem 1]{Mal53}
A group satisfies the identity $\mathbf{u}_n\approx\mathbf{u}'_n$ if and only if it is nilpotent of class at most $n-1$.
\end{lem}

By construction, the words $\mathbf{u}_n$ and $\mathbf{u}'_n$ involve the variables $x_1,\dots,x_{n+1}$. Define the following $n+1$ words over $\mathcal{Y}:=\{x_1,\dots,x_{n+1}\}$:
\begin{align}
\label{eq:words}
\mathbf{w}_1&:=x_1^2x_2\cdots x_{n+1}x_1,\notag\\
\mathbf{w}_2&:=x_1x_2^2\cdots x_{n+1}x_1,\notag\\
\ldots&\hbox to 3.3cm{\dotfill}\\
\mathbf{w}_{n}&:=x_1x_2\cdots x_{n}^2x_{n+1}x_1,\notag\\
\mathbf{w}_{n+1}&:=x_1x_2\cdots x_{n+1}x_1.\notag
\end{align}
The construction comes from~\cite[Section~2]{AVG09}. The last line of \eqref{eq:words} stands out because the word $\mathbf{w}_{n+1}$ involves no squared variable, unlike all the preceding words. The reason for this distinction will become clear shortly.

Denote by $\varphi$ the endomorphism of $\mathcal{Y}^+$ that extends the map $x_i\mapsto\mathbf{w}_i$, $i=1,\dots,n+1$. For each $k=1,2,\dotsc$, let $\mathbf{w}_{i,k}=\varphi^k(x_i)$. We claim that the identity
\begin{equation}\label{eq:nilpotence}
\mathbf{v}:=\mathbf{u}_n(\mathbf{w}_{1,2N},\dots,\mathbf{w}_{n+1,2N})\approx \mathbf{u}'_n(\mathbf{w}_{1,2N},\dots,\mathbf{w}_{n+1,2N})=:\mathbf{v}',
\end{equation}
where $N$ is the parameter from Lemma~\ref{lem:okninski}, has the requested properties: it holds in $U^0T_n(F)$ and forces any group satisfying it to be nilpotent.

 We start with the following fact which is a slight modification of \cite[Lemma 2]{KV24}.
\begin{lem}\label{lem:almeida}
If an epigroup $S$ has $\mathfrak{J}$-height less than $N$, then for every substitution $\mathcal{Y}\to{S}$, there is a subgroup $H$ of $S$ such that the values of all words $\mathbf{w}_{1,2N},\dots,\mathbf{w}_{n+1,2N}$ under this substitution belong to $H$.
\end{lem}

There are two differences between Lemma~\ref{lem:almeida} and Lemma~2 of~\cite{KV24}.

The first is that, in the latter lemma, the semigroup $S$ is required to be stable rather than an epigroup. Stability is a property of semigroups defined in terms of Green's relations, but we do not need to define it here---it is known that every epigroup is stable~\cite{HM79}, and therefore every result proved for stable semigroups automatically applies to epigroups.

The second difference is that, in the statement of Lemma~2 of~\cite{KV24}, the parameter $N$ upper-bounds the number of principal ideals rather than the length of their chains. However, the proof of the lemma uses only the fact that chains consisting of distinct principal ideals cannot have length exceeding $N$.

Thus, we may regard Lemma~\ref{lem:almeida} as having already been established in~\cite{KV24} and therefore omit its proof.

The following reasoning also reuses some tricks from \cite[Section~3]{KV24}, although for a quite distinct purpose. However, because of several important differences, we provide a full self-contained proof rather than a list of modifications that would have to be made to the arguments in \cite[Section~3]{KV24}.

Lemmas~\ref{lem:eprigroup}, \ref{lem:okninski}, and \ref{lem:almeida} ensure that under an arbitrary substitution $\zeta\colon\mathcal{Y}\to U^0T_n(F)$, the values of the words $\mathbf{w}_{1,2N},\dots,\mathbf{w}_{n+1,2N}$ lie in a subgroup of $U^0T_n(F)$. By Lemma~\ref{lem:nilpotent}, every subgroup of $U^0T_n(F)$ is nilpotent of class at most $n-1$, and the identity $\mathbf{u}_n\approx\mathbf{u}'_n$ holds in every such group by Lemma~\ref{lem:malcev}. Therefore, substituting for $x_1,\dots,x_{n+1}$ the values of the words $\mathbf{w}_{1,2N},\dots,\mathbf{w}_{n+1,2N}$ yields
\[
\mathbf{u}_n(\zeta(\mathbf{w}_{1,2N}),\dots,\zeta(\mathbf{w}_{n+1,2N}))=\mathbf{u}'_n(\zeta(\mathbf{w}_{1,2N}),\dots,\zeta(\mathbf{w}_{n+1,2N})).
\]
However,
\begin{gather*}
\mathbf{u}_n(\zeta(\mathbf{w}_{1,2N}),\dots,\zeta(\mathbf{w}_{n+1,2N}))=\zeta(\mathbf{u}_n(\mathbf{w}_{1,2N},\dots,\mathbf{w}_{n+1,2N}))=\zeta(\mathbf{v}),\\
\mathbf{u}'_n(\zeta(\mathbf{w}_{1,2N}),\dots,\zeta(\mathbf{w}_{n+1,2N}))=\zeta(\mathbf{u}'_n(\mathbf{w}_{1,2N},\dots,\mathbf{w}_{n+1,2N}))=\zeta(\mathbf{v}'),
\end{gather*}
and hence $\mathbf{v}$ and $\mathbf{v}'$ take the same value under $\zeta$. Since the substitution was arbitrary, the identity $\mathbf{v}\approx\mathbf{v}'$ holds in $U^0T_n(F)$.

Now take an arbitrary group $G$ satisfying the identity $\mathbf{v}\approx\mathbf{v}'$. We aim to show that then $G$ satisfies the identity $\mathbf{u}_n\approx\mathbf{u}'_n$ as well. This amounts to verifying that $\mathbf{u}_n(h_1,\dots,h_{n+1})=\mathbf{u}'_n(h_1,\dots,h_{n+1})$ for an arbitrary $m$-tuple of elements $h_1,\dots,h_{n+1}\in\ G$.

The free semigroup $\mathcal{Y}^+$ is a subsemigroup in the free group $\mathcal{FG}(\mathcal{Y})$ over $\mathcal{Y}$. The endomorphism $\varphi\colon x_i\mapsto w_i$ of $\mathcal{Y}^+$ extends to an endomorphism of $\mathcal{FG}(\mathcal{Y})$, still denoted by $\varphi$. The words $\mathbf{w}_1,\dots\mathbf{w}_{n+1}$ defined by \eqref{eq:words} generate $\mathcal{FG}(\mathcal{Y})$ since in $\mathcal{FG}(\mathcal{Y})$, one can express $x_1,\dots,x_{n+1}$ via $\mathbf{w}_1,\dots,\mathbf{w}_{n+1}$ as follows:
\begin{align*}
x_1&=\mathbf{w}_1\mathbf{w}_{n+1}^{-1},\\
x_2&=x_1^{-1}\mathbf{w}_2\mathbf{w}_{n+1}^{-1}x_1,\\
x_3&=(x_1x_2)^{-1}\mathbf{w}_3\mathbf{w}_{n+1}^{-1}x_1x_2,\\
\ldots&\hbox to 3.5cm{\dotfill}\\
x_{n}&=(x_1x_2\cdots x_{n-1})^{-1}\mathbf{w}_{n}\mathbf{w}_{n+1}^{-1}x_1x_2\cdots x_{n-1},\\
x_{n+1}&=(x_1x_2\cdots x_{n})^{-1}\mathbf{w}_{n+1}x_1^{-1}.
\end{align*}
(This is where the distinct expression for $\mathbf{w}_{n+1}$ comes into play.) Hence $\varphi$ treated as an endomorphism of $\mathcal{FG}(\mathcal{Y})$ is onto, and so is any power of $\varphi$. It is well known that every onto endomorphism of a finitely generated free group is an automorphism; see, e.g., \cite[Proposition~I.3.5]{LS77}. Denote by $\varphi^{-2N}$ the inverse of the automorphism $\varphi^{2N}$ of $\mathcal{FG}(\mathcal{Y})$ and let $g_i=\varphi^{-2N}(x_i)$, $i=1,\dots,n+1$. Then
\begin{multline}
\label{eq:auto}
\mathbf{w}_{i,2N}(g_1,\dots,g_{n+1})=\mathbf{w}_{i,2N}(\varphi^{-2N}(x_1),\dots,\varphi^{-2N}(x_{n+1}))\\
{}=\varphi^{-2N}(\mathbf{w}_{i,2N}(x_1,\dots,x_{n+1}))=\varphi^{-2N}(\varphi^{2N}(x_i))=x_i
\end{multline}
for all $i=1,\dots,n+1$. Since the equalities~\eqref{eq:auto} hold in the free $(n+1)$-generated group, they remain valid under any interpretation of the variables $x_1,\dots,x_{n+1}$ by arbitrary $n+1$ elements of an arbitrary group. Now we define a substitution $\zeta\colon\mathcal{Y}\to G$ letting
\[
\zeta(x_i):=g_i(h_1,\dots,h_{n+1}),\ \ i=1,\dots,n+1.
\]
Then in view of~\eqref{eq:auto} we have
\begin{multline*}
\zeta(\mathbf{w}_{i,2N}(x_1,\dots,x_{n+1}))= \mathbf{w}_{i,2N}(\zeta(x_1),\dots,\zeta(x_{n+1}))\\
{}= \mathbf{w}_{i,2N}\bigl(g_1(h_1,\dots,h_{n+1}),\dots,g_{n+1}(h_1,\dots,h_{n+1})\bigr)=h_i
\end{multline*}
for all $i=1,\dots,n+1$. Hence we have
\begin{align*}
\mathbf{u}_n(h_1,\dots,h_{n+1})&=\mathbf{u}_n(\zeta(\mathbf{w}_{1,2N}(x_1,\dots,x_{n+1})),\dots,\zeta(\mathbf{w}_{n+1,2N}(x_1,\dots,x_{n+1})))\\
&=\zeta(\mathbf{u}_n(\mathbf{w}_{1,2N}(x_1,\dots,x_{n+1}),\dots,\mathbf{w}_{n+1,2N}(x_1,\dots,x_{n+1})))\\
&=\zeta(\mathbf{v}(x_1,\dots,x_{n+1})),
\end{align*}
and, similarly, $\mathbf{u}'_n(h_1,\dots,h_{n+1})=\zeta(\mathbf{v}'(x_1,\dots,x_{n+1}))$. Since the identity $\mathbf{v}\approx\mathbf{v}'$ holds in $G$, the values of the words $\mathbf{v}$ and $\mathbf{v}'$ under $\zeta$ are equal, whence $\mathbf{u}_n(h_1,\dots,h_{n+1})=\mathbf{u}'_n(h_1,\dots,h_{n+1})$, as required. By Lemma~\ref{lem:malcev} the group $G$ is nilpotent. This completes the proof of Step 4 and of Theorem~\ref{thm:UpmTn}.
\end{proof}

\begin{remk}
If $F$ is a field of characteristic 0 containing, for some $m>2$, the roots of the equation $x^m=1$, then the proofs of Corollary~\ref{cor:UpmT2} and Theorem~\ref{thm:UpmTn} apply, with minor adjustments, to the monoid of all triangular $n\times n$ matrices over $F$, whose diagonal entries satisfy $x^{m+1}=x$. E.g, for each $m>2$, the monoid of all triangular complex $n\times n$ matrices with diagonal entries in the set $\{0,1,\xi,\dots,\xi^{m-1}\}$, where $\xi$ is a primitive $m$-th root of unity, is nonfinitely based.
\end{remk}

What happens if, instead of relaxing the condition imposed on the diagonal entries, we strengthen it by restricting the entries to $\pm1$ (so excluding 0)? The resulting set
\[
G^{\pm}_n(F):=\left\{\bigl(\alpha_{ij}\bigr)_{n\times n}\in T_n(F)\;\middle|\; \alpha_{ii}\in\{\pm1\}\ \text{for}\ i=1,\dots,n\right\}
\]
is nothing but  the group of units of the monoid $U^{\pm}T_n(F)$. It is an extension of the  subgroup $UT_n(F)$, which is nilpotent of class $n-1$, by the $2^n$-element group $D^{\pm}_n$ of the diagonal matrices with diagonal entries $\pm1$. In~\cite[Section~5]{Sap87a}, Mark Sapir stated a conjecture whose validity implies that if a group $G$ that is an extension of a nilpotent group of class $c$ by a finite group is finitely based as a semigroup, then $G$ is either abelian or of finite exponent. If $F$ is a field of characteristic 0, the group $G^{\pm}_n(F)$ is neither abelian nor of finite exponent. Thus, if Sapir's conjecture holds, then $G^{\pm}_n(F)$ is nonfinitely based as a semigroup for every $n$. So far, the conjecture has been established only for $c=1$, that is, for extensions of abelian groups by finite ones~\cite[Proposition~6]{Sap87a}. This implies that the group $G^{\pm}_2(F)$ is nonfinitely based as a semigroup.

Using the techniques from \cite{Vol15}, we can move one step further.

\begin{prop}\label{prop:group3pm}
For any field $F$ of characteristic $0$, the group $G^{\pm}_3(F)$ is nonfinitely based as a semigroup.
\end{prop}

\begin{proof}
We apply Proposition~\ref{prop:kauffman}. To this end, we nust verify two conditions: first, that the group $G^{\pm}_3(F)$ is contained in the variety generated by the Mal'cev product of a semigroup variety all of whose periodic members are locally finite and a locally finite semigroup variety; and second, that each Zimin word is an isoterm for $G^{\pm}_3(F)$.

Since $G^{\pm}_3(F)$ is an extension of the subgroup $UT_3(F)$ by the $8$-element group $D^{\pm}_3$, it belongs to the Mal'cev product of $\var UT_3(F)$ with $\var D^{\pm}_3$. The latter variety is locally finite, as it is generated by a finite semigroup; see Proposition~\ref{prop:birkhoff}. The group $UT_3(F)$ is nilpotent of class~2 and so, by Lemma~\ref{lem:malcev}, it satisfies the identity
\[
\mathbf{u}_3=x_1x_3x_2x_4x_2x_3x_1\approx x_2x_3x_1x_4x_1x_3x_2=\mathbf{u}'_3.
\]
Substituting $x_1x_2x_1$ for $x_2$ in this identity, we obtain that $UT_3(F)$ also satisfies  the identity
\[
x_1x_3x_1x_2x_1x_4x_1x_2x_1x_3x_1\approx x_1x_2x_1x_3x_1x_4x_1x_3x_1x_2x_1.
\]
Multiplying this identity through by $x_1x_2$ on the left and by $x_2x_1$ on the right, we obtain the identity
\begin{equation}\label{eq:Z4}
x_1x_2x_1x_3x_1x_2x_1x_4x_1x_2x_1x_3x_1x_2x_1\approx (x_1x_2)^2x_1x_3x_1x_4x_1x_3x_1(x_2x_1)^2.
\end{equation}
The left-hand side of the identity \eqref{eq:Z4} is the Zimin word $Z_4$, while the right-hand side is different from $Z_4$. A direct inspection shows that the identity \eqref{eq:Z4} is balanced. We are now in a position to apply the fact established in \cite[Lemma~3.3]{Sap87a} to the variety $\var UT_3(F)$: if a semigroup variety satisfies a nontrivial balanced identity of the form $Z_k\approx\mathbf{w}$, then all periodic semigroups of this variety  are locally finite.

It remains to verify that each Zimin word is an isoterm for $G^{\pm}_3(F)$. The argument here essentially repeats that from the proof of Proposition~\ref{prop:2x2pm1}. Since the characteristic of $F$ is 0, the additive semigroup of positive integers embeds into $G^{\pm}_3(F)$ via the map  $m\mapsto\left(\begin{smallmatrix}
1 & m & 0\\
0 & 1 & 0\\
0 & 0 &  1
\end{smallmatrix}\right)$. Therefore every identity holding in $G^{\pm}_3(F)$ must be balanced. By \cite[Lemma 5.2.1]{Sap87a}, every group satisfying a nontrivial balanced identity of the form $Z_k\approx\mathbf{w}$ is nilpotent. However, the group $G^{\pm}_3(F)$ is not nilpotent.
\end{proof}

\begin{remks}
1. The argument in the proof of Proposition~\ref{prop:group3pm} fails for the group  $G^{\pm}_4(F)$. Even though the subgroup $UT_4(F)$ satisfies the identity $\mathbf{u}_4\approx\mathbf{u}'_4$ by Lemma~\ref{lem:malcev}, this identity, in contrast to the identity $\mathbf{u}_3\approx\mathbf{u}'_3$, is known to hold in certain periodic but not locally finite semigroups~\cite[Theorem~2]{Zim80}.

2. The group identities of the group $G^{\pm}_n(F)$ are finitely based for every $n$, as all these groups fall under the scope of the general result by Alexei Krasilnikov on the laws of extensions of nilpotent groups by abelian ones~\cite[Theorem 1]{Kra91}.
\end{remks}

\subsubsection{Matrices over infinite fields of prime characteristic}

What can be said about the Finite Basis Problem for the monoid $U^{\pm}T_n(F)$ when $F$ is an infinite field of prime characteristic? The argument in the proof of Theorem~\ref{thm:UpmTn} fails at its first step, since in this case the semigroup $A_2$ can be shown not to belong to the variety $\var U^{\pm}T_n(F)$. Nevertheless, we can prove the following.

\begin{thm}\label{thm:UpmTn-prime}
For any infinite field $F$ of characteristic $p>2$ and any $n\ge4$, the monoid $U^{\pm}T_n(F)$ is inherently nonfinitely based.
\end{thm}

\begin{proof}
 In view of Theorem~\ref{thm:pm1}, it suffices to prove that the variety $\var U^{\pm}T_n(F)$ is locally finite. Indeed, the matrices of $U^{\pm}T_n(F)$ whose entries all lie in the prime subfield of the field $F$ form a submonoid isomorphic to the monoid $U^{\pm}T_n(\mathbb{Z}/p\mathbb{Z})$. The submonoid $U^{\pm}_{11}T_n(\mathbb{Z}/p\mathbb{Z})$ of $U^{\pm}T_n(\mathbb{Z}/p\mathbb{Z})$ is inherently nonfinitely based for each $n\ge4$ by Theorem~\ref{thm:pm1}. Therefore, $\var U^{\pm}T_n(F)$ contains an inherently nonfinitely based semigroup, and, clearly, every locally finite variety containing such a semigroup is itself inherently nonfinitely based.

We use a sufficient condition for the local finiteness of a semigroup variety, which is an easy consequence of a well-known result of Tom Brown. We include the proof for the sake of completeness and because a convenient reference is lacking.

\begin{lem}\label{lem:brown}
 The variety generated by the Mal'cev product of two locally finite semigroup varieties is locally finite.
\end{lem}

\begin{proof}
Brown~\cite{Br68}, see also \cite[Theorem~1]{Br71}, proved that if $\varphi\colon S\twoheadrightarrow T$ is a morphism of the semigroup $S$ onto the locally finite semigroup $T$ such that $\varphi^{-1}(e)$ is a locally finite subsemigroup of $S$ for each idempotent $e$ of $T$, then $S$ is locally finite. This readily implies that the Mal'cev product $\mathbf{M}$ of two locally finite semigroup varieties consists entirely of locally finite semigroups. The Mal'cev product of two semigroup varieties is known to be closed under taking subsemigroups and forming direct products~\cite[Theorems~1 and~2]{Mal67}. Consequently, by the HSP-theorem \cite[Theorem 11.9]{BuSa81}, the variety $\var\mathbf{M}$ consists of morphic images of semigroups from $\mathbf{M}$. Since the morphic image of a locally finite semigroup is again locally finite, it follows that the variety $\var\mathbf{M}$ is locally finite.
\end{proof}

Consider the kernel $\theta$ of the morphism $\diag\colon U^{\pm}T_n(F)\to D$, where $D$ is the submonoid formed by the diagonal matrices from $U^{\pm}T_n(F)$. The idempotents of $D$ are the $2^n$ diagonal matrices $e_P$ corresponding to the subsets $P\subseteq\{1,2,\dots,n\}$; here $e_P:=(\varepsilon_{ij})_{n\times n}$ with
\[
\varepsilon_{ij}=\begin{cases}
1 &\text{if $i=j\in P$},\\
0 &\text{if $i\ne j$ or $i=j\notin P$}.
\end{cases}
\]
Correspondingly, the $\theta$-classes that are subsemigroups of $U^{\pm}T_n(F)$ are of the form
\[
\diag^{-1}(e_P)=\left\{\bigl(\alpha_{ij}\bigr)_{n\times n}\in UT^0_n(F)\;\middle|\; \alpha_{ii}=1 \text{ if and only if } i\in P\right \}.
\]
We denote the subsemigroup $\diag^{-1}(e_P)$ by $S_P$ in the rest of the proof.

Clearly, the monoid $U^{\pm}T_n(F)$ lies in the Mal'cev product of the variety generated by all semigroups $S_P$ with the variety $\var D$. Therefore, $\var U^{\pm}T_n(F)$ is contained in the variety generated by this Mal'cev product. The variety $\var D$ is locally finite by Proposition~\ref{prop:birkhoff} since its generating semigroup $D$ is finite. Lemma~\ref{lem:brown} thus reduces the problem of establishing the local finiteness of $\var U^{\pm}T_n(F)$ to the same task for the variety $\var\left\{S_P\mid P\subseteq\{1,2,\dots,n\}\right\}$.

By Lemma~\ref{lem:putcha}, every semigroup $S_P$ is an epigroup: the corresponding family of polynomials consists of the polynomials $x_{ij}$ for all $1\le j<i\le n$, together with $x_{ii}-1$ for each $i\in P$ and $x_{jj}$ for each $j\notin P$.

For idempotents $e,f$ of a semigroup, define
\begin{equation}
\label{eq:natorder}
e \le f \quad \Longleftrightarrow \quad ef = fe = e;
\end{equation}
this relation is known to be a partial order on the set of all idempotents. In every semigroup $S_P$ this partial order is simply the equality relation. Indeed,  suppose $e,f$ are idempotents in $S_P$ with $e\le f$. Then $\diag(e)=e_P=\diag(f)$ whence $f-e$ is a triangular matrix with zero main diagonal. Therefore $(f-e)^n=O$, where $O$ is the $n\times n$ zero matrix. Since $f^2=f$ and $e^2=ef = fe = e$, expanding the binomial $(f-e)^n$ yields
\[
(f - e)^n = \sum_{k=0}^n \binom{n}{k} f^{\,n-k} (-e)^k = f + \sum_{k=1}^n \binom{n}{k} (-1)^k e=f-e.
\]
Hence, $f-e=O$, and so $f=e$.

By \cite[Proposition~3]{Sh94}, distinct idempotents of an epigroup $S$ are all incomparable with respect to the order~\eqref{eq:natorder} if and only if $S$ has a least ideal, and this ideal is equal to the union of all subgroups of $S$. (Such epigroups are termed \emph{Archimedean} in \cite{Sh94}.) So, let $K_P$ be the least ideal of $S_P$. We now prove that the product of any $t:=2^n\prod_{i=1}^{n}\binom{n}{i}$ matrices from $S_P$ falls into $K_P$.

Indeed, take any $a_1,\dots,a_t\in S_P$. For each $\ell=1,\dots,t$, let $R_\ell$ be the right ideal of $S_P$ generated by the product $a_1\cdots a_\ell$. Then
\begin{equation}\label{eq:right chain}
R_1\supseteq R_2\supseteq\dots\supseteq R_t.
\end{equation}
If some ideal in the chain \eqref{eq:right chain} is $\{O\}$ (which is possible only when $P=\varnothing$ and so $K_P=\{O\}$), then certainly $a_1\cdots a_t=O\in K_P$. Otherwise, we invoke \cite[Corollary 3.9(3)]{Okn98}, according to which the length of every chain of nonzero principal right ideals in every subepigroup of the monoid $M_n(F)$ is less than $t$. Then $R_\ell=R_{\ell+1}$ for some $1\le\ell<t$, which implies that
\[
a_1\cdots a_\ell=a_1\cdots a_\ell a_{\ell+1}b
\]
for some $b$ that is either a matrix from $S_P$ or the $n\times n$ identity matrix. Repeatedly substituting $a_1\cdots a_\ell a_{\ell+1}b$ for the factor $a_1\cdots a_\ell$ on the right-hand side, we obtain
\[
a_1\cdots a_\ell=a_1\cdots a_\ell(a_{\ell+1}b)^k
\]
for all $k=1,2,\dotsc$. For some $k$, the matrix $(a_{\ell+1}b)^k$ belongs to a subgroup of the epigroup $S_P$, and all subgroups of $S_P$ are contained in $K_P$. Hence, $(a_{\ell+1}b)^k$ lies in $K_P$. Since $K_P$ is an ideal, it follows that $a_1\cdots a_\ell\in K_P$, and therefore $a_1\cdots a_t\in K_P$, as claimed.

It is easy to see that every subgroup $H$ of $S_P$ has exponent dividing $p^{\lceil\log_p n\rceil}$; we include the argument for completeness. If $e$ is the identity element of $H$, then $\diag(e)=e_P=\diag(h)$ for an arbitrary $h\in H$, and hence $a:=h-e$ is a triangular matrix with zero main diagonal. Since $he=eh$, one has $ea=ae$, whence for every $r\ge 1$,
\[
h^{p^r}=(e+a)^{p^r}=\sum_{k=0}^{p^r} \binom{p^r}{k} e^{\,p^r-k}a^k=e+a^{p^r},
\]
because $F$ has characteristic $p$ and each binomial coefficient $\binom{p^r}{k}$ with $1\le k<p^r$ is divisible by~$p$. Taking $r=\lceil\log_p n\rceil$, we have $p^r\geq n$, and hence $a^{p^r}=O$. Therefore $h^{p^{\lceil\log_p n\rceil}}=e$. Thus every element of $H$ has order dividing $p^{\lceil\log_p n\rceil}$, and so the exponent of $H$ divides $p^{\lceil\log_p n\rceil}$.

We conclude that every subgroup of $S_P$ satisfies the identity $x\approx x^{p^{\lceil\log_p n\rceil}+1}$. As shown, $a_1\cdots a_t\in K_P$ for any matrices $a_1,\dots,a_t\in S_P$. Since $K_P$ is a union of subgroups, the product $a_1\cdots a_t$ in fact belongs to a subgroup of $S_P$.
Therefore, the identity
\begin{equation}\label{eq:finite index}
 x_1\cdots x_t\approx (x_1\cdots x_t)^{p^{\lceil\log_p n\rceil}+1}
\end{equation}
holds in each semigroup $S_P$.

In Step 4 of the proof of Theorem~\ref{thm:UpmTn}, we have constructed the identity \eqref{eq:nilpotence} which holds in the monoid $U^0T_n(F)$ and has the property that every group satisfying it is nilpotent of class at most $n-1$. Since every $S_P$ is a subsemigroup of $U^0T_n(F)$, the identity \eqref{eq:nilpotence} also holds in each semigroup $S_P$.

Thus, the variety  $\var\left\{S_P\mid P\subseteq\{1,2,\dots,n\}\right\}$ is contained in the variety $\mathbf{U}$ defined by the identities \eqref{eq:finite index} and \eqref{eq:nilpotence}. To prove that the former variety is locally finite, it therefore suffices to show that $\mathbf{U}$ is locally finite. To that end, we employ a powerful result of Sapir, which is a part of \cite[Theorem~P]{Sap87a}.

We recall two notions involved in Sapir's result. A variety is said to be of \emph{finite axiomatic rank} if, for some fixed $k$, it can be defined by a set of identities involving at most $k$ variables. (For instance, every finitely based variety is of finite axiomatic rank.) A semigroup $S$ is called a \emph{nilsemigroup} if $S$ has a zero and, for each element $s\in S$, some power of $s$ equals zero.

\begin{prop}
\label{prop:sapir}
A variety $\mathbf{V}$ of finite axiomatic rank consisting of periodic semigroups is locally finite if \textup(and, obviously, only if\textup) all groups and all nilsemigroups in $\mathbf{V}$  are locally finite.
\end{prop}

By construction, the variety $\mathbf{U}$ is finitely based. Clearly, every semigroup satisfying the identity \eqref{eq:finite index} is periodic, so $\mathbf{U}$ consists of periodic semigroups. If a nilsemigroup $S$ satisfies \eqref{eq:finite index}, then the product of any $t$ elements of $S$ equals zero. Therefore, for every $m$, an $m$-generated subsemigroup of $S$ contains at most $1+m+m^2+\dots m^{t-1}$ elements, and hence $S$ is locally finite. By the choice of the identity \eqref{eq:nilpotence}, every group in $\mathbf{U}$ is nilpotent, and it is well known that periodic nilpotent groups are locally finite. Thus, Proposition~\ref{prop:sapir} applies to the variety $\mathbf{U}$, yielding its local finiteness and thus completing the proof of Theorem~\ref{thm:UpmTn-prime}.
\end{proof}

\subsubsection{Matrices over additively idempotent semirings}
\label{subsubsec:ai-semirings}

Here we consider matrix monoids over an arbitrary \emph{additively idempotent} semiring, that is, a semiring $\mathbb{S}=(S,+,\cdot)$ satisfying, in addition to (S1)–-(S5), the condition
\[
a+a=a\  \text{ for all }\ a\in S.
\]
Additively idempotent semirings naturally arise in various areas of mathematics (e.g., idempotent analysis, tropical geometry, and optimization) and computer science, where rings and other `more classical' algebras fail to provide adequate tools; see the contributions in the conference volumes \cite{Guna98,LiMa:2005} and the monograph \cite{Bis04} for diverse examples. Recall that the tropical semiring $\mathbb{T}$ and bounded distributive lattices are special instances of additively idempotent semirings.

\begin{thm}\label{thm:ai-semirings}
For every additively idempotent semiring\ $\mathbb{S}$, the monoid $U^0T_n(\mathbb{S})$ is equationally equivalent to the monoid $T_n(\mathbb{B})$ of all $n\times n$ triangular Boolean matrices.
\end{thm}

The proof of Theorem~\ref{thm:ai-semirings} relies on techniques developed by Marianne Johnson and Peter Fenner in~\cite{JF19}, as well as on a combinatorial characterization of the identities holding in the monoid $T_n(\mathbb{B})$ obtained by the third-named author~\cite{Vol25}. To state this characterization, we need some further notions related to semigroup words.

The \emph{length} $|\mathbf{w}|$ of a word $\mathbf{w}=w_1\cdots w_k$, where $w_1, \dots, w_k$ are variables, is $k$; the length of the empty word is 0. We say that $\mathbf{w}$ is a \emph{scattered subword} of a word $\mathbf{u}$ if there exist words $\mathbf{u}_1, \mathbf{u}_2, \ldots, \mathbf{u}_k, \mathbf{u}_{k+1}$ (some of which may be empty) such that
\begin{equation}\label{l-scattered subword}
\mathbf{u}=\mathbf{u}_1w_1 \mathbf{u}_2\cdots \mathbf{u}_kw_k \mathbf{u}_{k+1}.
\end{equation}
We refer to any decomposition of the form (\ref{l-scattered subword}) as an \emph{occurrence} of $\mathbf{w}$ in $\mathbf{u}$. We say that $\mathbf{w}$ occurs in a word $\mathbf{u}$ \emph{with gaps} $G_1, G_2, \ldots, G_{k+1}$ if there is an occurrence (\ref{l-scattered subword}) of $\mathbf{w}$ as a scattered subword in $\mathbf{u}$ such that $G_i=\mathsf{con}(\mathbf{u}_i)$ for $i=1, \ldots, k+1$.

\begin{example}
The three occurrences of the word $x$ as a scattered subword in $x^2yx$ are:
\[
x^2yx=\begin{cases}
\underline{x}\cdot xyx,&\text{with gaps $\varnothing,\{x,y\}$;}\\
x\cdot\underline{x}\cdot yx,&\text{with gaps $\{x\},\{x,y\}$};\\
x^2y\cdot\underline{x},&\text{with gaps $\{x,y\},\varnothing$.}
\end{cases}
\]
The two occurrences of the word $xy$ as a scattered subword in $x^2yx$ are:
\[
x^2yx=\begin{cases}
\underline{x}\cdot x\cdot\underline{y}\cdot x,&\text{with gaps $\varnothing,\{x\},\{x\}$;}\\
x\cdot\underline{x}\underline{y}\cdot x,&\text{with gaps $\{x\},\varnothing,\{x\}$.}
\end{cases}
\]
\end{example}

We adopt the convention that the empty word occurs as a scattered subword in every word $\mathbf{u}$ in a unique way with gap $G_1=\mathsf{con}(\mathbf{u})$.

The identities holding in the monoid $T_n(\mathbb{B})$ are characterized as follows:

\begin{prop}\cite[Corollary 2.3]{Vol25}\label{prop:identitiesTnB}
The monoid $T_n(\mathbb{B})$ satisfies an identity if and only if every word of length $k<n$ that occurs in one side of the identity with gaps $G_1,G_2, \ldots, G_{k+1}$ occurs in the other side of the identity with gaps $G'_1, G'_2, \ldots, G'_{k+1}$ such that $G'_\ell\subseteq G_\ell$ for $\ell=1, 2, \ldots, k+1$.
\end{prop}

We will need an explicit expansion formula for the value of any word under an arbitrary substitution in the monoid $U^0T_n(\mathbb{S})$. For this, we introduce another batch of notions and notation.

Let $[n]:=\{1, 2, \ldots, n\}$. An $m$-tuple $\lambda=(\lambda_1, \lambda_2, \ldots, \lambda_m)$ over $[n]$ is called a \emph{walk of length $m$ from $\lambda_1$ to  $\lambda_m$} if $\lambda_1\leq \lambda_2\leq \ldots\leq\lambda_m$. It is clear that every walk has the form
\[
	\lambda=(i_1, \dots, i_1, i_2, \dots, i_2, \dots, i_k, \dots, i_k),
\]
where $i_1, i_2, \dots, i_k\in [n]$ are such that $i_1 < i_2 < \cdots < i_k$, and thus $k\le n$.

A walk $(\rho_1, \rho_2, \ldots, \rho_k)$ is \emph{strictly increasing} if $\rho_1< \rho_2< \dots <\rho_k$. For $i,j\in [n]$ with $i<j$, denote by  $[n]^{k}_{i,j}$ the set of all strictly increasing walks of length $k+1$ from $i$ to $j$ in $[n]$; note that the superscript in this notation is one less than the length.

For a word $\mathbf{u}$, denote the set of all its scattered subwords of length at most $k$ by $\mathsf{sca}_k(\mathbf{u})$. If $\mathbf{w}=w_1\cdots w_{|\mathbf{w}|}$ is a scattered subword of $\mathbf{u}$, denote by $S_{\mathbf{u}}^{\mathbf{w}}$ the set of all gap sequences of the occurrences of $\mathbf{w}$ in $\mathbf{u}$, that is,
\[
S_{\mathbf{u}}^{\mathbf{w}}:=\left\{G=(G_1, G_2,\dots, G_{|\mathbf{w}|+1})\;\middle|\;\mathbf{w} \ \mbox{occurs in $\mathbf{u}$ with gaps}\ G_1, \ldots, G_{|\mathbf{w}|+1}\right\}.
\]

For a matrix $A\in M_n(\mathbb{S})$ and $i,j\in [n]$, we denote the $(i,j)$ entry of $A$ by $A_{ij}$.

The next lemma is inspired by \cite[Lemma 2.1]{JF19}.

\begin{lem}\label{lem:ijexpansion}
Let\/ $\mathbb{S}$ be a semiring. For any substitution $\varphi\colon \mathcal{X}\to U^0T_n(\mathbb{S})$, any word $\mathbf{u}\in\mathcal{X}^+$, and any $i, j\in [n]$ with $i\le j$,
\begin{equation}\label{eq:expansion}
	\varphi(\mathbf{u})_{ij}=\sum_{\mathbf{w}\in \mathsf{sca}_{j-i}(\mathbf{u})}\sum_{G\in S_\mathbf{u}^\mathbf{w}}\sum_{(\rho_1,\dots,\rho_{|\mathbf{w}|+1})\in [n]^{|\mathbf{w}|}_{i, j}} \prod_{k=1}^{|\mathbf{w}|} \varphi (w_k)_{\rho_{k}\rho_{k+1}} \prod_{s=1}^{|\mathbf{w}|+1}\prod_{x\in G_s} \varphi (x)_{\rho_s \rho_s},
\end{equation}
where $w_k$ in the factor $\prod\limits_{k=1}^{|\mathbf{w}|} \varphi (w_k)_{\rho_{k}\rho_{k+1}}$  stands for the $k$-th from the left variable of $\mathbf{w}$ if the scattered subword $\mathbf{w}\in\mathsf{sca}_{j-i}(\mathbf{u})$ is nonempty; otherwise, the factor is understood as $1$.
\end{lem}

\begin{proof}
Let $\mathbf{u}=u_1u_2\cdots u_m$ with $u_1, \dots, u_m\in \mathcal{X}$. By the definition of matrix multiplication,
\[
\varphi(\mathbf{u})_{ij}=\sum\varphi(u_1)_{\lambda_1\lambda_2} \varphi(u_2)_{\lambda_2\lambda_3}\cdots \varphi(u_m)_{\lambda_m\lambda_{m+1}},
\]
where the sum runs over all $(m+1)$-tuples $(\lambda_1, \ldots, \lambda_{m+1})$ with entries in $[n]$ such that $\lambda_1=i$ and $\lambda_{m+1}=j$. Since each matrix $\varphi(u_s)$ is triangular, we have $\varphi(u_s)_{k\ell}=0$ whenever $k > \ell$. Consequently, only tuples satisfying
\[
\lambda_1 \leq \lambda_2 \leq \cdots \leq \lambda_{m+1}
\]
can contribute to the sum. We may therefore restrict to such nondecreasing $(m+1)$-tuples and write
\begin{equation}\label{eq:walkexpansion}
 \varphi(\mathbf{u})_{ij}=\sum\limits_{(\lambda_1, \ldots, \lambda_{m+1})}\varphi(u_1)_{\lambda_1\lambda_2} \varphi(u_2)_{\lambda_2\lambda_3}\cdots \varphi(u_m)_{\lambda_m\lambda_{m+1}},
\end{equation}
where the sum is now taken over all walks $(\lambda_1, \ldots, \lambda_{m+1})$ with $\lambda_1=i$ and $\lambda_{m+1}=j$.

If $i=j$, \eqref{eq:walkexpansion} reduces to
\[
 \varphi(\mathbf{u})_{ii}=\varphi(u_1)_{ii} \varphi(u_2)_{ii}\cdots \varphi(u_m)_{ii}.
\]
Taking into account that the diagonal entries of matrices in $U^0T_n(\mathbb{S})$ are 0s and 1s, we can rewrite the product $\varphi(u_1)_{ii} \varphi(u_2)_{ii}\cdots \varphi(u_m)_{ii}$ as $\prod\limits_{x\in\mathsf{con}(\mathbf{u})} \varphi (x)_{ii}$. On the other hand, under our convention that the empty word occurs as a scattered subword in $\mathbf{u}$ in a unique way with gap $G_1=\mathsf{con}(\mathbf{u})$, the right-hand side of \eqref{eq:expansion} reduces to exactly the same expression. Hence, the equality \eqref{eq:expansion} holds for $i=j$.

Now let  $i<j$. For each summand
\begin{equation}\label{eq:summand}
    \varphi(u_1)_{\lambda_1\lambda_2} \varphi(u_2)_{\lambda_2\lambda_3}\cdots \varphi(u_m)_{\lambda_m \lambda_{m+1}}
\end{equation}
on the right-hand side of \eqref{eq:walkexpansion}, represent the walk $(\lambda_1, \ldots, \lambda_{m+1})$ as
\[
(\lambda_1, \dots, \lambda_{m+1})=(i_1, \dots, i_1, i_2, \dots, i_2, \dots, i_{k+1}, \dots, i_{k+1})
\]
with $i_1 < i_2 < \cdots < i_{k+1}$, where the first occurrences of $i_2, i_3, \ldots, i_{k+1}$ are at positions $s_1+1$, $s_2+1$, \dots, $s_k+1$, respectively. Let $w_1:=u_{s_1}$, $w_2:=u_{s_2}$, \dots, $w_k:=u_{s_{k}}$, and $\mathbf{w}:=w_1w_2\cdots w_k$, and set
\[
\mathbf{u}_1:=u_1\cdots u_{s_1-1},\ \mathbf{u}_2:=u_{s_1+1}\cdots u_{s_2-1}, \dots, \mathbf{u}_{k+1}:=u_{s_k+1}\cdots u_m.
\]
(Here and throughout, an expression of the form $x_p x_{p+1}\cdots x_q$ is understood to equal the empty word whenever $p>q$.) Then $\mathbf{w}$ is a scattered subword of $\mathbf{u}$ of length $k\leq j-i$, and
\[
\mathbf{u}=\mathbf{u}_1w_1 \mathbf{u}_2\cdots \mathbf{u}_kw_k \mathbf{u}_{k+1},
\]
is an occurrence of $\mathbf{w}$ in $\mathbf{u}$ with gaps $G_1:=\mathsf{con}(\mathbf{u}_1), \ldots, G_{k+1}:=\mathsf{con}(\mathbf{u}_{k+1})$.

Using this representation of the word $\mathbf{u}$, we can rewrite the summand \eqref{eq:summand} as
\[
	\varphi(u_1)_{\lambda_1\lambda_2} \varphi(u_2)_{\lambda_2\lambda_3}\cdots \varphi(u_m)_{\lambda_{m} \lambda_{m+1}}= \varphi(\mathbf{u}_1)_{i_1 i_1}\varphi(w_1)_{i_1i_2}\cdots\varphi(w_k)_{i_{k}i_{k+1}}\varphi(\mathbf{u}_{k+1})_{i_{k+1} i_{k+1}}.
\]
Since the diagonal entries of the matrices in $U^0T_n(\mathbb{S})$ are 0s and 1s, we can rewrite $\varphi(\mathbf{u}_s)_{i_s i_s}$ as $\prod\limits_{x\in G_s} \varphi (x)_{i_s i_s}$ for each $s=1,2,\dots,k+1$. Since both 0 and 1 commute with every element of $\mathbb{S}$, we can collect all such products at the end, thus obtaining
\[
\varphi(u_1)_{\lambda_1\lambda_2} \varphi(u_2)_{\lambda_2\lambda_3}\cdots \varphi(u_m)_{\lambda_{m} \lambda_{m+1}}= \prod_{t=1}^{k}\varphi(w_t)_{i_t i_{t+1}} \prod_{s=1}^{k+1}\prod_{x\in G_s} \varphi (x)_{i_s i_s}.
\]
Therefore, each summand \eqref{eq:summand} is equal to a summand on the right-hand side of \eqref{eq:expansion}.
	
It should be clear that, conversely, each summand on the right-hand side of \eqref{eq:expansion} can be rewritten into the form \eqref{eq:summand}, and so it is  equal to a summand on the right-hand side of~\eqref{eq:walkexpansion}. Hence the equality \eqref{eq:expansion} holds also for $i<j$.
\end{proof}

The final ingredient for the proof of Theorem~\ref{thm:ai-semirings} is a partial order on the semiring $\mathbb{S}$. Define $a\le b$ if and only if $a+b=b$; it is well known (and easy to see) that the relation $\le$ is indeed a partial order on any additively idempotent semiring which is compatible with addition and multiplication in the sense that $a\le b$ implies $a+c\le b+c$, $ac\le bc$, and $ca\le cb$ for all $a,b,c$.

\begin{proof}[Proof of Theorem~\ref{thm:ai-semirings}]
The matrices from $U^0T_n(\mathbb{S})$ all of whose entries are 0s and 1s form a  submonoid isomorphic to $T_n(\mathbb{B})$. Hence, every identity holding in $U^0T_n(\mathbb{S})$ holds in  $T_n(\mathbb{B})$. Therefore, to establish the equational equivalence of these monoids, it suffices to show that $U^0T_n(\mathbb{S})$ satisfies every identity $\mathbf{u}\approx \mathbf{v}$ holding in  $T_n(\mathbb{B})$.
	
Let $\varphi\colon \mathcal{X} \to U^0T_n(\mathbb{S})$ be an arbitrary substitution. Since $\mathbf{u}$ and $\mathbf{v}$ share scattered subwords of length at most $n-1$, in particular, they have the same scattered subwords of length~$1$, that is, the same variables. Thus, $\mathsf{con}(\mathbf{u}) = \mathsf{con}(\mathbf{v})$. Consequently, $\varphi(\mathbf{u})_{ii} = \varphi(\mathbf{v})_{ii}$ for all $i \in [n]$. Therefore, to prove that $\varphi(\mathbf{u}) = \varphi(\mathbf{v})$, we only need to show that $\varphi(\mathbf{u})_{ij} = \varphi(\mathbf{v})_{ij}$ for all $i,j \in [n]$ with $i < j$. Due to the symmetry, it suffices to verify that $\varphi(\mathbf{u})_{ij}\le\varphi(\mathbf{v})_{ij}$ whenever $1\le i<j\le n$.
	
Let $\mathbf{w}=w_1w_2\cdots w_k$ be a scattered subword of both $\mathbf{u}$ and $\mathbf{v}$, where $w_1, w_2, \ldots, w_k\in \mathcal{X}$. Without loss of generality, assume $k\leq j-i$. Consider any occurrence of $\mathbf{w}$ in $\mathbf{u}$ with gaps $G_1,G_2, \dots, G_{k+1}$. Since the identity $\mathbf{u}\approx\mathbf{v}$ holds in $T_n(\mathbb{B})$, it follows from Proposition \ref{prop:identitiesTnB} that $\mathbf{w}$ occurs in $\mathbf{v}$ with gaps $G'_1,G'_2, \dots, G'_{k+1}$ such that $G'_s\subseteq G_s$ for each $s=1, 2, \ldots, k+1$.

Since the diagonal entries of matrices in $U^0T_n(\mathbb{S})$ are $0$s or $1$s, so are the products $\prod\limits_{x\in G_s}\ \varphi(x)_{\rho_s \rho_s}$ and $\prod\limits_{x\in G'_s}\ \varphi(x)_{\rho_s \rho_s}$. In the order of $\mathbb{S}$, $0<1$, and therefore, the inclusion $G'_s\subseteq G_s$ ensures that
\[
\prod_{x\in G_s}\ \varphi(x)_{\rho_s \rho_s}=\prod_{x\in  G_s\setminus G'_s}\ \varphi(x)_{\rho_s \rho_s}\cdot\prod_{x\in  G'_s}\ \varphi(x)_{\rho_s \rho_s} \le \prod_{x\in G'_s}\ \varphi(x)_{\rho_s \rho_s}.
\]
The order $\le$ is compatible with multiplication, whence for any $(\rho_1, \rho_2, \dots, \rho_{k+1})\in [n]^{k}_{i,j}$,
\[
\prod_{t=1}^{k}\ \varphi(w_t)_{\rho_t \rho_{t+1}}\prod_{s=1}^{k+1}\prod_{x\in G_s}\ \varphi(x)_{\rho_s \rho_s} \le
\prod_{t=1}^{k}\ \varphi(w_t)_{\rho_t \rho_{t+1}}\prod_{s=1}^{k+1}\prod_{x\in {G}'_s}\ \varphi(x)_{\rho_s \rho_s}.
\]
Hence every summand of expansion \eqref{eq:expansion} of $\varphi(\mathbf{u})_{ij}$ is less than or equal to some summand in the analogous expansion of $\varphi(\mathbf{v})_{ij}$. Since the order $\le$ is compatible with addition and adding extra summands increases the sum, it follows that
\[
\begin{split}
	\varphi(\mathbf{u})_{ij} & \stackrel{\eqref{eq:expansion}}{=}\sum_{\mathbf{w}\in \mathsf{sca}_{j-i}(\mathbf{u})}\sum_{G\in S_\mathbf{u}^\mathbf{w}}\sum_{(\rho_1,\dots,\rho_{|\mathbf{w}|+1})\in [n]^{|\mathbf{w}|}_{i, j}}\prod_{t=1}^{k}\ \varphi(w_t)_{\rho_t \rho_{t+1}}\prod_{s=1}^{k+1}\prod_{x\in G_s}\ \varphi(x)_{\rho_s \rho_s} \\
	& \le\sum_{\mathbf{w}\in \mathsf{sca}_{j-i}(\mathbf{v})}\sum_{G'\in S_\mathbf{v}^\mathbf{w}}\sum_{(\rho_1,\dots,\rho_{|\mathbf{w}|+1})\in [n]^{|\mathbf{w}|}_{i, j}}\prod_{t=1}^{k}\ \varphi(w_t)_{\rho_t \rho_{t+1}}\prod_{s=1}^{k+1}\prod_{x\in G'_s}\ \varphi(x)_{\rho_s \rho_s}
	\stackrel{\eqref{eq:expansion}}{=} \varphi(\mathbf{v})_{ij},
\end{split}
\]
as required.
\end{proof}

Theorem~\ref{thm:ai-semirings} parallels the result of Johnson and Fenner \cite[Corollary 3.4]{JF19} who considered monoids of unitriangular matrices and proved the equational equivalence of $UT_n(\mathbb{S})$ and $UT_n(\mathbb{B})$.

Due to the importance of the tropical semiring $\mathbb{T}$, the following specialization of Theorem~\ref{thm:ai-semirings} seems worth stating:
\begin{cor}\label{cor:SequivalentB}
The monoid $U^0T_n(\mathbb{T})$ is equationally equivalent to the monoid $T_n(\mathbb{B})$ of all $n\times n$ triangular Boolean matrices. 	
\end{cor}

We mention in passing that Corollary~\ref{cor:SequivalentB} allows one to establish  non-obvious structural properties of the monoid $U^0T_n(\mathbb{T})$ without any matrix computations. For instance, combined with the local finiteness of the finitely generated variety $\var T_n(\mathbb{B})$, Corollary~\ref{cor:SequivalentB}  readily implies that $U^0T_n(\mathbb{T})$ is locally finite. (This fact is, of course, known as a special case of a general result of St\'ephane Gaubert~\cite{Gau96}, who proved that periodic matrix semigroups over $\mathbb{T}$ are locally finite.)

Back to our main theme, the Finite Basis Problem, Theorem~\ref{thm:ai-semirings} leads to a complete solution for matrix monoids from the family $U^0T_n(\mathbb{S})$.

\begin{cor}\label{cor:inherentU0T}
For every additively idempotent semiring $\mathbb{S}$, the monoid $U^0T_n(\mathbb{S})$ with $n\ge 3$ is inherently nonfinitely based. The monoid $U^0T_2(\mathbb{S})$ is finitely based.
\end{cor}

\begin{proof}
By Theorem~\ref{thm:ai-semirings}, the answers to  the Finite Basis Problem for the monoids $U^0T_n(\mathbb{S})$ and $T_n(\mathbb{B})$ coincide. It is known that the latter monoid is finitely based for $n=2$~\cite[Theorem 3.5]{LL11} and inherently nonfinitely based for $n=3$~\cite[Theorem 4.1]{LL11} and $n\ge 4$~\cite[Theorem 2.1]{VG04}
\end{proof}

\subsubsection{Matrices over semirings with small idempotents}
\label{subsubsec:si-semirings}

Arguments in Subsection~\ref{subsubsec:ai-semirings} can be extended to certain monoids of triangular matrices with more relaxed restrictions on the diagonal elements. The price to pay is that the underlying additively idempotent semirings must then satisfy additional conditions, which nevertheless hold in many important instances.

Suppose that a semiring $\mathbb{S}$ satisfies (S1)--(S5), and in addition, its multiplicative idempotents commute, that is, $\mathbb{S}$ satisfies the implication
\begin{equation}\label{eq:commutingidem}
    e=e^2\ \&\ f=f^2\to ef=fe.
\end{equation}
Then the product of any two multiplicative idempotents is a multiplicative idempotent whence the set $ET_n(\mathbb{S})$ of all $n\times n$ triangular matrices with idempotent diagonal entries,
\begin{equation}\label{eq:ET}
    ET_n(\mathbb{S}):=\left\{\bigl(\alpha_{ij}\bigr)_{n\times n}\in T_n(\mathbb{S})\;\middle|\;  \alpha^2_{ii}=\alpha_{ii}\ \text{for}\ i=1,\dots,n\right\},
\end{equation}
forms a submonoid of the monoid $T_n(\mathbb{S})$. The monoid $ET_n(\mathbb{S})$ is the principal object of this subsection.

Clearly, $ET_n(\mathbb{S})=U^0T_n(\mathbb{S})$ whenever 0 and 1 are the only multiplicative idempotents in $\mathbb{S}$, as is the case for fields, $\mathbb{Z}$, and $\mathbb{T}$. Otherwise, the monoid $ET_n(\mathbb{S})$ strictly contains $U^0T_n(\mathbb{S})$ and can be much larger. In particular, if every element of $\mathbb{S}$ is multiplicatively idempotent (for instance, if $\mathbb{S}$ is a bounded distributive lattice), then $ET_n(\mathbb{S})=T_n(\mathbb{S})$.

The following is an analogue of Lemma~\ref{lem:ijexpansion} for the monoid $ET_n(\mathbb{S})$.
\begin{lem}\label{lem:ijexpansion_forES}
Let\/ $\mathbb{S}$ be a semiring whose multiplicative idempotents commute. For any substitution $\varphi\colon \mathcal{X}\to ET_n(\mathbb{S})$, any word $\mathbf{u}\in\mathcal{X}^+$, and any $i, j\in [n]$ with $i\le j$,
\begin{equation}\label{eq:expansion_forES}
	\varphi(\mathbf{u})_{ij}=\sum_{\mathbf{w}\in \mathsf{sca}_{j-i}(\mathbf{u})}\sum_{G\in S_\mathbf{u}^\mathbf{w}}\sum_{(\rho_1,\dots,\rho_{|\mathbf{w}|+1})\in [n]^{|\mathbf{w}|}_{i, j}}  \prod_{x\in G_1} \varphi (x)_{\rho_1 \rho_1}\cdot \prod_{s=1}^{|\mathbf{w}|}\varphi (w_s)_{\rho_{s}\rho_{s+1}}\left(\prod_{x\in G_{s+1}} \varphi (x)_{\rho_{s+1} \rho_{s+1}}\right),
\end{equation}
where $w_s$ in the factor $\prod\limits_{s=1}^{|\mathbf{w}|}\varphi (w_s)_{\rho_{s}\rho_{s+1}}\left(\prod\limits_{x\in G_{s+1}} \varphi (x)_{\rho_{s+1} \rho_{s+1}}\right)$ stands for the $s$-th from the left variable of $\mathbf{w}$ if the scattered subword $\mathbf{w}\in\mathsf{sca}_{j-i}(\mathbf{u})$ is nonempty; otherwise, the factor is understood as $1$.
\end{lem}

\begin{proof}
We proceed similarly to the proof of Lemma~\ref{lem:ijexpansion}: write  $\mathbf{u}=u_1u_2\cdots u_m$, where $u_1, \dots, u_m$ are variables, and represent $\varphi(\mathbf{u})_{ij}$ as the sum in \eqref{eq:walkexpansion} taken over all walks $(\lambda_1, \ldots, \lambda_{m+1})$ with $\lambda_1=i$ and $\lambda_{m+1}=j$.

If $i=j$, \eqref{eq:walkexpansion} reduces to
\[
 \varphi(\mathbf{u})_{ii}=\varphi(u_1)_{ii} \varphi(u_2)_{ii}\cdots \varphi(u_m)_{ii}.
\]
Since the diagonal entries of matrices in $ET_n(\mathbb{S})$ are idempotents, we may use \eqref{eq:commutingidem} to rearrange the product $\varphi(u_1)_{ii} \varphi(u_2)_{ii}\cdots \varphi(u_m)_{ii}$ into blocks of equal factors and then, by idempotency, replace each block with a single factor. This reduces the product to $\prod\limits_{x\in\mathsf{con}(\mathbf{u})} \varphi (x)_{ii}$. The right-hand side of \eqref{eq:expansion_forES} reduces to exactly the same expression, by our convention that the empty word occurs as a scattered subword in $\mathbf{u}$ in a unique way, with gap $G_1=\mathsf{con}(\mathbf{u})$. Hence, the equality \eqref{eq:expansion_forES} holds for $i=j$.

Now let  $i<j$. As in the proof of Lemma~\ref{lem:ijexpansion}, consider the   summand \eqref{eq:summand} in \eqref{eq:walkexpansion} corresponding to the walk
\[
(\lambda_1, \ldots, \lambda_{m+1})=(i_1, \ldots, i_1, i_2, \ldots, i_2, \ldots, i_{k+1}, \ldots, i_{k+1})
\]
with $i_1 < i_2 < \cdots < i_{k+1}$ and the first occurrences of $i_2, i_3, \dots, i_{k+1}$ at positions $s_1+1$, $s_2+1$, \dots, $s_k+1$, respectively. We assign to \eqref{eq:summand} an occurrence
\[
\mathbf{u}=\mathbf{u}_1w_1 \mathbf{u}_2\cdots \mathbf{u}_kw_k \mathbf{u}_{k+1},
\]
of a scattered subword $\mathbf{w}:=w_1w_2\cdots w_k$ of length $k\leq j-i$ such that  $w_1:=u_{s_1}$, $w_2:=u_{s_2}$, \dots, $w_k:=u_{s_k}$, and
\[
\mathbf{u}_1:=u_1\cdots u_{s_1-1},\ \mathbf{u}_2:=u_{s_1+1}\cdots u_{s_2-1}, \dots, \mathbf{u}_{k+1}:=u_{s_k+1}\cdots u_m.
\]
Let $G_1:=\mathsf{con}(\mathbf{u}_1), \dots, G_{k+1}:=\mathsf{con}(\mathbf{u}_{k+1})$ be the gaps of this occurrence. Rewrite the summand \eqref{eq:summand}  as
\[
	\varphi(u_1)_{\lambda_1\lambda_2} \varphi(u_2)_{\lambda_2\lambda_3}\cdots \varphi(u_m)_{\lambda_{m} \lambda_{m+1}}= \varphi(\mathbf{u}_1)_{i_1 i_1}\varphi(w_1)_{i_1i_2}\cdots\varphi(w_k)_{i_{k}i_{k+1}}\varphi(\mathbf{u}_{k+1})_{i_{k+1} i_{k+1}}.
\]
Taking into account that the diagonal entries of matrices in $ET_n(\mathbb{S})$ are commuting idempotents, we conclude that $\varphi(\mathbf{u}_s)_{i_s i_s}=\prod\limits_{x\in G_s} \varphi (x)_{i_s i_s}$ for each $s=1,2,\dots,k+1$. Hence,
\[
\varphi(u_1)_{\lambda_1\lambda_2} \varphi(u_2)_{\lambda_2\lambda_3}\cdots \varphi(u_m)_{\lambda_{m} \lambda_{m+1}}=\prod_{x\in G_1} \varphi (x)_{i_1 i_1}\cdot \prod_{s=1}^{k}\varphi (w_s)_{i_{s}i_{s+1}}\left(\prod_{x\in G_{s+1}} \varphi (x)_{i_{s+1} i_{s+1}}\right).
\]
Thus, each summand \eqref{eq:summand} is equal to a summand on the right-hand side of \eqref{eq:expansion_forES}.
	
Conversely, each summand on the right-hand side of \eqref{eq:expansion_forES} can be rewritten into the form \eqref{eq:summand}, and so it is equal to a summand on the right-hand side of~\eqref{eq:walkexpansion}. Hence the equality \eqref{eq:expansion_forES} holds also for $i<j$.
\end{proof}

To obtain an analogue of Theorem~\ref{thm:ai-semirings} for matrix monoids of the form $ET_n(\mathbb{S})$, we need a further restriction on the underlying semiring. Recall that the relation $\le$ defined by $a\le b$ if and only if $a+b=b$ is a compatible partial order on additively idempotent semirings. We say that an additively idempotent semiring has \emph{small idempotents} if $e\le 1$ for every multiplicative idempotent $e$. Alternatively, semirings with small idempotents can be defined as additively idempotent semirings satisfying the implication
\[
e=e^2\to e+1=1.
\]

Examples of semirings with small idempotents include bounded distributive lattices, and more generally, \emph{inclines} with 0 and 1. Inclines are additively idempotent semirings satisfying the identities $x+xy\approx x\approx x+yx$; these identities are equivalent to the inequalities $xy\le x\ge yx$, which in the presence of 1 are equivalent to $y\le 1$. (Thus, in an incline with 1, all elements are smaller than or equal to 1, not only idempotents.) Matrices over inclines have been considered in the literature; see the monograph~\cite{CKR84}, the survey~\cite{KR04}, and the references therein.

\begin{thm}\label{thm:si-semirings}
For every semiring $\mathbb{S}$ with commuting and small idempotents, the monoid $ET_n(\mathbb{S})$ is equationally equivalent to the monoid $T_n(\mathbb{B})$ of all $n\times n$ triangular Boolean matrices.
\end{thm}

\begin{proof}
The arguments from the proof of Theorem~\ref{thm:ai-semirings} apply \emph{mutatis mutandis}. It suffices to show that $ET_n(\mathbb{S})$ satisfies every identity $\mathbf{u}\approx \mathbf{v}$ holding in  $T_n(\mathbb{B})$, that is,  $\varphi(\mathbf{u}) = \varphi(\mathbf{v})$ for an arbitrary substitution $\varphi\colon \mathcal{X} \to ET_n(\mathbb{S})$.

By Proposition~\ref{prop:identitiesTnB}, the words $\mathbf{u}$ and $\mathbf{v}$ share scattered subwords of length at most $n-1$. In particular, they share variables, that is, $\mathsf{con}(\mathbf{u}) = \mathsf{con}(\mathbf{v})$. Since the diagonal entries of matrices in $ET_n(\mathbb{S})$ are commuting idempotents, we have
\[
\varphi(\mathbf{u})_{ii}=\prod_{x\in\mathsf{con}(\mathbf{u})}=\prod_{x\in\mathsf{con}(\mathbf{v})}=\varphi(\mathbf{v})_{ii}
\]
for all $i\in [n]$. It remains to verify that $\varphi(\mathbf{u})_{ij} = \varphi(\mathbf{v})_{ij}$ for all $i,j \in [n]$ with $i < j$, that is,  $\varphi(\mathbf{u})_{ij}\le\varphi(\mathbf{v})_{ij}$ and $\varphi(\mathbf{v})_{ij}\le\varphi(\mathbf{u})_{ij}$. We check only the first inequality, as the second will follow by symmetry.

Let $\mathbf{w}=w_1w_2\cdots w_k$ with $k\leq j-i$ be a scattered subword of both $\mathbf{u}$ and $\mathbf{v}$, where $w_1, w_2, \ldots, w_k\in \mathcal{X}$.  If  $\mathbf{w}$ occurs in $\mathbf{u}$ with gaps $G_1,G_2, \dots, G_{k+1}$,  by Proposition \ref{prop:identitiesTnB}, $\mathbf{w}$ occurs in $\mathbf{v}$ with gaps $G'_1,G'_2, \dots, G'_{k+1}$ such that $G'_s\subseteq G_s$ for $1\le s \le k+1$. We then have
\begin{align*}
\prod_{x\in G_s}\ \varphi(x)_{\rho_s \rho_s}&=\prod_{x\in  G_s\setminus G'_s} \varphi(x)_{\rho_s \rho_s}\cdot\prod_{x\in  G'_s}\ \varphi(x)_{\rho_s \rho_s}&&\text{since the diagonal entries commute}\\
&\le \prod_{x\in G'_s}\ \varphi(x)_{\rho_s \rho_s}&&\text{since $\prod_{x\in  G_s\setminus G'_s} \varphi(x)_{\rho_s \rho_s}\le 1$}.
\end{align*}
Using the multiplicative compatibility of the order $\le$, for any $(\rho_1, \rho_2, \dots, \rho_{k+1})\in [n]^{k}_{i,j}$, we obtain
\begin{multline*}
\prod_{x\in G_1} \varphi (x)_{\rho_1 \rho_1}\cdot \prod_{s=1}^{k}\varphi (w_s)_{\rho_{s}\rho_{s+1}}\left(\prod_{x\in G_{s+1}} \varphi (x)_{\rho_{s+1} \rho_{s+1}}\right)\\ \le
\prod_{x\in G'_1} \varphi (x)_{\rho_1 \rho_1}\cdot \prod_{s=1}^{k}\varphi (w_s)_{\rho_{s}\rho_{s+1}}\left(\prod_{x\in G'_{s+1}} \varphi (x)_{\rho_{s+1} \rho_{s+1}}\right).
\end{multline*}
Hence every summand of expansion \eqref{eq:expansion_forES} of $\varphi(\mathbf{u})_{ij}$ is less than or equal to some summand in the analogous expansion of $\varphi(\mathbf{v})_{ij}$. Since the order $\le$ is compatible with addition and adding extra summands increases the sum, we obtain $\varphi(\mathbf{u})_{ij}\le\varphi(\mathbf{v})_{ij}$, as required.
\end{proof}

Theorem~\ref{thm:si-semirings} allows one to solve the Finite Basis Problem for monoids of the form $ET_n(\mathbb{S})$, where $\mathbb{S}$ is an arbitrary semiring with commuting and small idempotents.

\begin{cor}\label{cor:inherentET}
For every semiring $\mathbb{S}$ with commuting and small idempotents, the monoid $ET_n(\mathbb{S})$ with $n\ge 3$ is inherently nonfinitely based. The monoid $ET_2(\mathbb{S})$ is finitely based.
\end{cor}

The proof repeats that of Corollary~\ref{cor:inherentU0T}, and we omit it.

Recall that $ET_n(L)=T_n(L)$ for any bounded distributive lattice $L$. Hence, Corollary~\ref{cor:inherentET} specializes to the following solution to the Finite Basis Problem for monoids of triangular matrices over lattices.

\begin{cor}\label{cor:trianglelattice}
For every bounded distributive lattice $L$, the monoid $T_n(L)$ is inherently nonfinitely based for $n\ge3$ and finitely based for $n=2$.
\end{cor}

We will say more about matrix monoids over lattices in the next subsection.

\subsubsection{Matrices over distributive lattices}

Semigroups of matrices over distributive lattices have been considered in the literature from various viewpoints; see, e.g., \cite{Give64,Cha70,Sko86,Tan02}. Their identities, however, do not appear to have been explored so far. In particular, the following easy observation does not seem to have been recorded previously.

\begin{prop}\label{prop:lattices}
For every bounded distributive lattice $L$, the monoid $M_n(L)$ is equationally equivalent to the monoid $M_n(\mathbb{B})$ of all $n\times n$ Boolean matrices.
\end{prop}

\begin{proof}
The matrices with entries 0 and 1 form a submonoid of $M_n(L)$ isomorphic to the monoid $M_n(\mathbb{B})$. Therefore, every identity of $M_n(L)$ holds in $M_n(\mathbb{B})$.

To prove the converse, we embed $M_n(L)$ into a direct power of $M_n(\mathbb{B})$. For this, it suffices to construct, for any two distinct matrices $a,b\in M_n(L)$, a semigroup morphism $M_n(L)\to M_n(\mathbb{B})$ whose images of $a$ and $b$ are distinct; see \cite[Section II.8]{BuSa81}. If $a=\bigl(\alpha_{ij}\bigr)$ and $b=\bigl(\beta_{ij}\bigr)$, then the inequality $a\ne b$ implies that $\alpha_{k\ell}\ne\beta_{k\ell}$ for some $k,\ell\in\{1,\dots,n\}$. Since every distributive lattice embeds into a direct power of $\mathbb{B}$, there exists a lattice morphism $\psi\colon L\to\mathbb{B}$ such that  $\psi(\alpha_{k\ell})\ne\psi(\beta_{k\ell})$. Lift $\psi$ to a map $\varphi\colon M_n(L)\to M_n(\mathbb{B})$ by letting $\varphi\bigl((\gamma_{ij})\bigr):=\bigl(\psi(\gamma_{ij})\bigr)$ for every matrix $\bigl(\gamma_{ij}\bigr)\in M_n(L)$. It is easy to verify that $\varphi$ is a semigroup morphism, and by construction, $\varphi(a)\ne\varphi(b)$. Therefore, every identity of $M_n(\mathbb{B})$ holds in $M_n(L)$.
\end{proof}

The following corollary, which we include for completeness, is now immediate.

\begin{cor}\label{cor:lattice}
For every bounded distributive lattice $L$, the monoid $M_n(L)$ is inherently nonfinitely based.
\end{cor}

\begin{proof}
Recall that the monoid $M_n(\mathbb{B})$ of all $n\times n$ Boolean matrices is inherently nonfinitely based; see Comment $^{1)}$ after Table~\ref{tab:plainstatus}.
\end{proof}

\begin{remks}
1. The argument used in Proposition~\ref{prop:lattices} works equally well for the  submonoids $T_n(L)=ET_n(L)$, $U^0T_n(L)$, and $UT_n(L)$ of the monoid $M_n(L)$. Thus, we have a simple proof that the monoids $T_n(L)=ET_n(L)$ and $U^0T_n(L)$ are equationally equivalent to $T_n(\mathbb{B})=ET_n(\mathbb{B})=U^0T_n(\mathbb{B})$ while the monoid $UT_n(L)$ is equationally equivalent to $UT_n(\mathbb{B})$. Even though these equivalences are special cases of those established in Subsections~\ref{subsubsec:ai-semirings} and \ref{subsubsec:si-semirings} and respectively~\cite{JF19} by syntactic means, we believe that their direct derivation from basic principles is worth recording.

2. Proposition~\ref{prop:lattices} allows one to obtain easy, uniform, and calculation-free proofs for several facts about matrices over distributive lattices that were previously established via matrix computations.

For a typical instance, in~\cite{Give64}, one of the first papers in the area, Yehoshafat Give'on proved that for an arbitrary bounded distributive lattice $L$, the monoid $M_n(L)$ is locally finite~\cite[Theorem 1]{Give64}. By Proposition~\ref{prop:lattices}, $M_n(L)$ belongs to the finitely generated variety $\var M_n(\mathbb{B})$, and by Proposition~\ref{prop:birkhoff}, every finitely generated variety is locally finite. Hence the result.

Yet another structural property of the monoid $M_n(L)$ established in~\cite{Give64} (and later rediscovered in~\cite{Sko86}) is direct finiteness. Recall that a monoid is called \emph{directly finite} (or \emph{Dedekind finite}) if it satisfies the implication
\begin{equation}\label{eq:dedekind}
    ab=1 \to ba=1.
\end{equation}
By the proof of Proposition~\ref{prop:lattices}, $M_n(L)$ embeds into a direct power of the finite (and hence directly finite) monoid $M_n(\mathbb{B})$. Since implications are inherited by direct powers and submonoids, it follows that $M_n(L)$ satisfies \eqref{eq:dedekind}.
\end{remks}

\subsection{Open problems}
\label{subsec:plainopen}
Here we summarize the as yet unsettled instances of the Finite Basis Problem for monoids of triangular matrices over the semirings considered above.

For monoids of triangular matrices over a finite field, only the cases of $2\times 2$ and $3\times 3$ matrices remain open. It is known that for any finite field $F$, the monoid $T_3(F)$ is neither inherently nonfinitely based~\cite{VG03} nor hereditarily finitely based~\cite{ZLL13}. However, nothing we currently know excludes  the options for $T_3(F)$ to be strongly nonfinitely based or, oppositely, finitely based.

\begin{problem}
 Determine the status of the Finite Basis Problem for the monoid of all $3\times 3$ triangular matrices over a finite field.
\end{problem}

The monoid $T_2(\mathbb{F}_2)$ was shown to be finitely based in~\cite{ZLL12}, where an explicit finite basis for its identities was provided. It was subsequently proved to be hereditarily finitely based in~\cite{ZLL13}. Moreover, $T_2(\mathbb{F}_2)$ falls within the scope of the general positive solution obtained in~\cite{CHL16} for the Finite Basis Problem for monoids of the form $U^0T_2(F)$, where $F$ is an arbitrary field, since $T_2(\mathbb{F}_2)=U^0T_2(\mathbb{F}_2)$. For monoids of $2\times 2$ triangular matrices over finite fields with more than two elements, the question of their finite basedness remains open.

\begin{problem}
 Determine the status of the Finite Basis Problem for the monoid of all $2\times 2$ triangular matrices over a finite field with at least $3$ elements.
\end{problem}

In view of the aforementioned result of~\cite{CHL16} and Theorem~\ref{thm:strongU0T}, the case $n=3$ is the only one for which the Finite Basis Problem for monoids of the form $U^0T_n(F)$, where $F$ is a finite field, has not yet been settled.

\begin{problem}
 Determine the status of the Finite Basis Problem for the monoid of all $3\times 3$ triangular matrices with diagonal entries $0$ and $1$ over a finite field.
\end{problem}

 Monoids $U^0T_2(F)$, where $F$ is an infinite field, are finitely based, with explicit finite basis found in~\cite{CHL16}. Monoids $U^0T_3(F)$, where $F$ is a field of characteristic 0, is nonfinitely based~\cite{Vol15}.  That is all we currently know about the Finite Basis Problem for monoids of the form $U^0T_n(F)$, where $F$ is an infinite field.

\begin{problem}
Determine the status of the Finite Basis Problem for the monoid of all $n\times n$ triangular matrices with diagonal entries $0$ and $1$ over an infinite field for $n\ge 4$ and for $n=3$ if the field has prime characteristic.
\end{problem}

The monoid family $U^{\pm}T_n(F)$ is defined over any field $F$ of characteristic different from $2$. Theorem~\ref{thm:UpmTn} shows that the monoid $U^{\pm}T_n(F)$ is nonfinitely based when $F$ has characteristic~$0$, but for the group of units of $U^{\pm}T_n(F)$, the absence of finite basis of semigroup identities is established only for $n=2$ and $n=3$ (Proposition~\ref{prop:group3pm}).

\begin{problem}
Determine the status of the Finite Basis Problem for the group of all $n\times n$ triangular matrices with diagonal entries $\pm1$ a field of characteristic $0$ for $n\ge 4$.
\end{problem}

Theorem~\ref{thm:pm1} and Theorem~\ref{thm:UpmTn-prime}  imply that the monoid $U^{\pm}T_n(F)$ is inherently nonfinitely based if $n\ge 4$ and $F$ is a (finite or infinite) field of odd characteristic.

\begin{problem}
Determine the status of the Finite Basis Problem for the monoid of all $2\times 2$ and $3\times 3$ triangular matrices with diagonal entries $0$ and $\pm1$ over a field of odd characteristic.
\end{problem}

It is known (see \cite{Izh14,Okn15}) that for each $n$, the monoid $T_n(\mathbb{T})$ of $n\times n$ triangular tropical matrices satisfies nontrivial semigroup identities. One expects that $T_n(\mathbb{T})$ is nonfinitely based for all $n$ but so far this has been established only for $n=2$ \cite{CHLS16} and $n=3$ \cite{HZL21}.

\begin{problem}
Determine the status of the Finite Basis Problem for the monoid of all $n\times n$ triangular matrices over the tropical semiring for $n\ge 4$.
\end{problem}

The monoid family $ET_n(\mathbb{S})$ is defined over any semiring whose multiplicative idempotents commute, but our result on the Finite Basis Problem for $ET_n(\mathbb{S})$  (Corollary~\ref{cor:inherentET}) covers only the case where the idempotents in $\mathbb{S}$ are small.

\begin{problem}
Determine the status of the Finite Basis Problem for the monoid of all $n\times n$ triangular matrices with idempotent diagonal entries over additively idempotent semirings  whose multiplicative idempotents commute but are not necessarily small.
\end{problem}

\section{Laws involving multiplication and skew transposition}
\label{sec:unary}
\subsection{Twisted involutory semigroups}
\label{subsec:stability}

Recall that \sgps\ satisfying both $xy\approx yx$ and $x^2\approx x$ are called \emph{semilattices}. An involutory semigroup $(S,{}^*)$ whose reduct $S$ is a semilattice with an absorbing element 0 is said to be a \emph{twisted semilattice} if 0 is the only fixed point of the involution $x\mapsto x^*$. This class of involutory semigroups was first considered by Siemion Fajtlowicz~\cite{Faj72}. It is easy to see that the minimum nontrivial object in this class is the 3-element twisted semilattice $S\ell_3:=(\{e,f,0\}, {}^*)$ in which $e^2=e$, $f^2=f$ and all other products are equal to $0$, while the unary operation swaps $e$ and $f$ and fixes 0, i.e., $e^*:=f$, $f^*:=e$, and $0^*:=0$.

An involutory semigroup is called \emph{twisted} if the variety it generates contains $S\ell_3$. This notion is important in the context of the Finite Basis Problem because all gradations of infinite basedness considered in this paper persist under the addition of an involution whenever the resulting involutory semigroup is twisted.

\begin{lem}[Transfer Lemma]\label{lem:stability}
Let $(S,{}^*)$ be a twisted involutory semigroup.
\begin{itemize}
\item[$1.$] If the semigroup $S$ is nonfinitely based, then $(S,{}^*)$ is nonfinitely based.

\item[$2.$] If the involutory semigroup variety generated by $(S,{}^*)$ is finitely generated and the semigroup $S$ is strongly nonfinitely based, then $(S,{}^*)$ is strongly nonfinitely based.

\item[$3.$] The semigroup $S$ is inherently nonfinitely based if and only if $(S,{}^*)$ is inherently nonfinitely based.
\end{itemize}
\end{lem}

\begin{proof} \textbf{Claim 1} is~\cite[Theorem 4]{Lee17}; see also~\cite[Theorem 10.1]{Lee23a}.

\smallskip

\textbf{Claim 2}. We argue by contraposition and assume that the involutory semigroup $(S,{}^*)$ is not strongly nonfinitely based. Then $(S,{}^*)$ belongs to a finitely based involutory semigroup variety generated by a finite involutory semigroup $(T, {}^*)$. Since every identity of the semigroup  $T$ holds in the semigroup $S$, the latter belongs to the semigroup variety $\var T$. By Claim~1, the semigroup $T$ is finitely based. Hence the semigroup $S$ is not strongly nonfinitely based.

\smallskip

\textbf{Claim 3}. The `only if' part is~\cite[Theorem 16]{Lee17}, but the proof given there omitted some important details. Therefore, we provide a full proof here.

By definition, the fact that $S$ is inherently nonfinitely based means that the semigroup variety $\var S$ is inherently nonfinitely based. In particular, $\var S$ is locally finite. This implies that the involutory semigroup variety $\mathbf{V}$ generated by  $(S,{}^*)$ is also locally finite. Indeed, if $(T,{}^*)$ is an involutory semigroup in $\mathbf{V}$ and is generated as an involutory semigroup by a finite set $A$, then the semigroup $T$ belongs to the variety $\var S$ and is generated, as a plain semigroup, by the finite set $A\cup\{a^*\mid a\in A\}$. Therefore, $T$ is a finite set.

Now, if we assume that $\mathbf{V}$ is not inherently nonfinitely based, then $\mathbf{V}$ must be contained in some locally finite finitely based involutory semigroup variety $\mathbf{U}$. If $(U, {}^*)$ is an involutory semigroup generating $\mathbf{U}$, then the semigroup $U$ is finitely based by Claim~1. Since every semigroup identity of $U$ holds in the semigroup $S$, the latter belongs to the semigroup variety $\var U$. It remains to verify that the variety $\var U$ is locally finite, for this will yield a contradiction to the inherent nonfinite basedness of $S$.

Thus, let $Q\in\var U$ be a semigroup generated by a finite set $B$. By the HSP-theorem \cite[Theorem 11.9]{BuSa81}, there is an onto morphism $\varphi\colon R\twoheadrightarrow Q$, where $R$ is a subsemigroup of a direct power $\overline{U}$ of $U$. We may assume that the subsemigroup $R$ is generated by a finite set $C$ containing, for each $b\in B$, one element of $\varphi^{-1}(b)$. The involutory semigroup $(\overline{U},{}^*)$ belongs to the involutory semigroup variety $\mathbf{U}$ and therefore is locally finite. Hence the finite set $C$ generates a finite involutory subsemigroup of  $(\overline{U}, {}^*)$. Since this involutory subsemigroup contains $R$, it follows that $R$ is finite. Consequently, $Q=\varphi(R)$ is finite as well.

The `if' part holds even without the assumption that $(S,{}^*)$ is twisted, by \cite[Lemma 2.1]{ADPV14}. (In \cite{ADPV14}, this result is stated for finite $S$, but its proof does not use  finiteness.)
\end{proof}

\begin{remks}
1. In \cite[Theorem 3.1]{ADPV14}, it was shown that a \textbf{finite} twisted involutory semigroup is inherently nonfinitely based whenever its semigroup reduct is inherently nonfinitely based. The above proof of Lemma~\ref{lem:stability}{\red.3} demonstrates that the finiteness condition is not required.

\smallskip

2. The `twistedness' condition in  Lemma~\ref{lem:stability} is essential. An involutory semigroup that is not twisted can be finitely based even when its semigroup reduct is nonfinitely based. A classical example (due to John Isbell~\cite{Isb70}) is the one-relator monoid $\langle a,b\mid ab^2a=1\rangle$, which is in fact a group. Viewed as an involutory semigroup with group inversion as the unary operation, the group is finitely based; indeed, the single law $x^2y^2\approx y^2x^2$ forms an identity basis for it. On the other hand, the semigroup reduct of this group is nonfinitely based.

Yet another example is of interest because of its small size (which is in fact the minimum possible) and---in the context of the present paper---because it is a semigroup of triangular matrices. Consider the following six matrices from $U^0T_3(\mathbb{B})$:
\[
\begin{tabular}{cccccc}
$\left(\begin{matrix} 0&0&0\\0&0&0\\0&0&0 \end{matrix}\right)$,
&
$\left(\begin{matrix} 0&1&0\\0&1&0\\0&0&0 \end{matrix}\right)$,
&
$\left(\begin{matrix} 1&0&0\\0&0&1\\0&0&1 \end{matrix}\right)$,
&
$\left(\begin{matrix} 0&0&1\\0&0&1\\0&0&0 \end{matrix}\right)$,
&
$\left(\begin{matrix} 0&1&0\\0&0&0\\0&0&0 \end{matrix}\right)$,
&
$\left(\begin{matrix} 0&0&1\\0&0&0\\0&0&0 \end{matrix}\right)$.
\\
\rule{0cm}{.3cm}$O$&$e$&$f$&$ef$&$fe$&$fef$
\end{tabular}
\]
They form a subsemigroup of $U^0T_3(\mathbb{B})$ denoted by $L_3$. The 6-element semigroup $L_3$ is nonfinitely based~\cite{Lee12,ZL11}, being one of the four nonfinitely based semigroups of minimum possible size \cite{LZ15}. It admits a unique involution, which interchanges the matrices $ef$ and $fe$ and fixes all the other matrices of $L_3$. The corresponding involution semigroup is finitely based~\cite{Lee16}.

\smallskip

3. The reader may have noticed that the statement of Claim~2 has a form different from those of Claim~1 and the `only if' part of Claim~3. In fact, we do not know whether the implication
\begin{center}
\emph{$S$ is strongly nonfinitely based and $(S,{}^*)$ is twisted $\to$ $(S,{}^*)$ is strongly nonfinitely based}
\end{center}
holds. The difficulty is that it is by no means clear whether $(S,{}^*)$ generates a finitely generated variety if $S$ does. Note that, in general, if a semigroup $S$ is equationally equivalent to its finite subsemigroup $T$ and has an involution that leaves $T$ invariant, the involutory semigroup $(S,{}^*)$ need not be equationally equivalent to $(T, {}^*)$. For instance, the monoid $UT_n(\mathbb{T})$ is equationally equivalent to its submonoid $UT_n(\mathbb{B})$ by \cite[Corollary 3.4]{JF19}, whereas the involutory monoids  $(UT_n(\mathbb{T}),{}^S)$ and $(UT_n(\mathbb{B}),{}^S)$ are not  equationally equivalent~\cite[Theorem 5.2]{HZL21}.

\smallskip

4. No analogous `stability' result holds for finite basedness: if $(S,{}^*)$ is a twisted involutory semigroup and the semigroup $S$ is finitely based, $(S,{}^*)$ may nevertheless be nonfinitely based. Many instances of such `unstable' behavior can be found in the literature. From the perspective of the present paper, it is of interest that the smallest possible example is a monoid of triangular matrices considered as an involutory monoid under skew transposition.

The following five matrices from $T_2(\mathbb{B})$:
\begin{equation}\label{eq:A01}
\begin{pmatrix} 0 & 0\\ 0 & 0\end{pmatrix},\
\begin{pmatrix} 0 & 1\\ 0 & 0\end{pmatrix},\
\begin{pmatrix} 1 & 1\\ 0 & 0\end{pmatrix},\
\begin{pmatrix} 0 & 1\\ 0 & 1\end{pmatrix},\
\begin{pmatrix} 1 & 0\\ 0 & 1\end{pmatrix},
\end{equation}
form a submonoid of $T_2(\mathbb{B})$, which is usually denoted by $A_0^1$ in the literature. That this monoid is finitely based was shown by Charles Edmunds~\cite{Edm77}, and Edmond Lee~\cite{Lee08} proved that it is even hereditarily finitely based. Since the set $A_0^1$ is closed under skew transposition, one can consider the involutory monoid $(A_0^1, {}^S)$. The first four matrices in \eqref{eq:A01} form an involutory subsemigroup, and the map
\[
\begin{pmatrix} 0 & 0\\ 0 & 0\end{pmatrix}\mapsto 0,\ \
\begin{pmatrix} 0 & 1\\ 0 & 0\end{pmatrix}\mapsto 0,\ \
\begin{pmatrix} 1 & 1\\ 0 & 0\end{pmatrix}\mapsto e,\ \
\begin{pmatrix} 0 & 1\\ 0 & 1\end{pmatrix}\mapsto f
\]
is readily seen to be an involutory semigroup morphism from this subsemigroup onto $S\ell_3$. Hence, $S\ell_3$ belongs to the variety generated by $(A_0^1, {}^S)$, and the latter monoid is twisted. However, the involutory monoid $(A_0^1, {}^S)$ is nonfinitely based, as shown in~\cite{GZL202}. It has been recently established that all involution semigroups of order less than five are finitely based \cite{GLLZ26}; therefore $(A_0^1, {}^S)$ is of minimum possible size among nonfinitely based involutory semigroups whose semigroup reducts are finitely based.
\end{remks}

We are going to employ the Transfer Lemma in our study of the Finite Basis Problem for monoids of triangular matrices equipped with the skew transposition. To do so, we need to determine which of these monoids is twisted. We recall the convention stated in Section~\ref{subsec:matrices} that the term `semiring' denotes a structure  $\mathbb{S}=(S,+,\cdot)$ satisfying (S1)--(S5).

\begin{lem}\label{lem:twisted1}
For every semiring $\mathbb{S}$, the involutory monoids $(U^0T_n(\mathbb{S}),{}^S)$ and $(T_n(\mathbb{S}),{}^S)$ are twisted.
\end{lem}

\begin{proof}
The $n\times n$ matrix with the (1,1) entry 1 and all other entries 0, its skew transpose, and
the $n\times n$ zero matrix constitute an involutory subsemigroup in the involutory monoids $(U^0T_n(\mathbb{S}),{}^S)$ and $(T_n(\mathbb{S}),{}^S)$, and this involutory subsemigroup is isomorphic to $S\ell_3$.
\end{proof}

\begin{remks}
1. Obviously, the argument in the proof of Lemma~\ref{lem:twisted1} applies equally to   $(M_n(\mathbb{S}),{}^S)$, and hence this involutory monoid is twisted for every semiring $\mathbb{S}$. We mention in passing that, in contrast, the monoid $M_n(\mathbb{S})$  equipped with the usual transposition need not be twisted. For instance, if $F$ is a finite field such that $|F|\equiv 3\pmod{4}$, the involutory monoid $(M_2(F),{}^T)$ is not twisted. Indeed, this follows from Lemma~\ref{lem:stability}{\red.3}, since the monoid $M_2(F)$ is inherently nonfinitely based \cite[Corollary 6.2]{Sap87a}, whereas the involutory monoid $(M_2(F),{}^T)$ is not \cite[Theorem 3.11.2]{ADV12}.

2. If multiplicative idempotents of a semiring $\mathbb{S}$ commute, the involutory monoid $(ET_n(\mathbb{S}),{}^S)$ is also twisted by the same argument. Recall that the monoid $ET_n(\mathbb{S})$ was defined in \eqref{eq:ET}.

3. For any field $F$ of odd characteristic, the involutory monoid $(U^{\pm}T_n(F),{}^S)$ contains $(U^0T_n(F),{}^S)$ as an involutory submonoid, and hence is twisted by Lemma \ref{lem:twisted1}. Recall that the monoid $U^{\pm}T_n(F)$ was defined in \eqref{eq:UpmTfull}.
\end{remks}

The (more involved) case of the involutory monoid $(UT_n(\mathbb{S}),{}^S)$ is handled in the next lemma.

\begin{lem}\label{lem:twisted2}
\emph{1}. The involutory monoid $(UT_n(\mathbb{S}),{}^S)$ is twisted if $n\ge3$ and\/ $\mathbb{S}$ is either an additively idempotent semiring, a field of characteristic\/ $0$ or the ring\/ $\mathbb{Z}$.

\emph{2}. The involutory monoid $(UT_n(\mathbb{S}),{}^S)$ is not twisted if either $n=2$ or\/  $\mathbb{S}$ is a field of finite characteristic.
\end{lem}

\begin{proof}
\textbf{Claim 1} was proved in \cite[Lemma 3.1]{HZL21} for the tropical semiring $\mathbb{T}$. The argument given there works verbatim for an arbitrary additively idempotent semiring.

For the proof of the remaining part of the claim, it suffices to show that for $n\ge3$, the involutory monoid $(UT_n(\mathbb{Z}),{}^S)$ is twisted. Indeed, every field $F$ of characteristic $0$ contains $\mathbb{Z}$ as a subring, and hence $(UT_n(F),{}^S)$ contains $(UT_n(\mathbb{Z}),{}^S)$ as an involutory submonoid.

Consider in $UT_n(\mathbb{Z})$ the following three sets of unitriangular matrices with nonnegative entries:
\begin{align*}
H_{12}&:=\left\{\bigl(\alpha_{ij}\bigr)_{n\times n}\in UT_n(\mathbb{Z})\;\middle|\;  \alpha_{12}=0,\ \alpha_{ij}>0\ \text{if}\ 1\le i<j\le n\ \text{and}\ (i,j)\ne(1,2)  \right\},\\
H_{n-1\,n}&:=\left\{\bigl(\alpha_{ij}\bigr)_{n\times n}\in UT_n(\mathbb{Z})\;\middle|\;  \alpha_{n-1\,n}=0,\ \alpha_{ij}>0\ \text{if}\ 1\le i<j\le n\ \text{and}\ (i,j)\ne(n-1,n)  \right\},\\
Z&:=\left\{\bigl(\alpha_{ij}\bigr)_{n\times n}\in UT_n(\mathbb{Z})\;\middle|\;  \alpha_{ij}>0\ \text{for all}\ 1\le i<j\le n\right\}.
\end{align*}
Since $n\ge3$, $(1,2)\ne (n-1,n)$ whence the three sets are disjoint. Clearly, $A^S\in H_{n-1\,n}$ for every matrix $A\in H_{12}$ and $B^S\in H_{12}$ for every matrix $B\in H_{n-1\,n}$, whereas $C^S\in Z$ for every matrix $C\in Z$. It is easy to compute that $AB,BA,AC,CA,BC,CB\in Z$ for all $A\in H_{12}$,  $B\in H_{n-1\,n}$, and $C\in Z$. Besides that, each of the sets $H_{12},H_{n-1\,n},Z$ is readily seen to be closed under matrix multiplication. Therefore, $(H_{12}\cup H_{n-1\,n}\cup Z, {}^S)$ is an involutory subsemigroup in $(UT_n(\mathbb{Z}),{}^S)$, and the map $H_{12}\cup H_{n-1\,n}\cup Z\to S\ell_3$ that sends matrices in $H_{12}$, $H_{n-1\,n}$, and $Z$ to $e$, $f$, and 0, respectively, is an onto morphism of involutory semigroups. Therefore $S\ell_3$ belongs to the variety generated by $(UT_n(\mathbb{Z}),{}^S)$.

 \smallskip

\textbf{Claim 2}. Since $A^S=A$ for all matrix $A\in UT_2(\mathbb{S})$, the involutory monoid $(UT_2(\mathbb{S}),{}^S)$ satisfies the identity $x^*\approx x$. This identity fails in $S\ell_3$, and therefore $S\ell_3$ does not belong to the variety generated by $(UT_2(\mathbb{S}),{}^S)$.

For every field $F$ of finite characteristic $p$, the monoid $UT_n(F)$ is a group of exponent $p^{\lceil\log_p n \rceil}$ and therefore satisfies the identity $x^{p^{\lceil\log_p n \rceil}}y\approx y$. This identity fails in every nontrivial semilattice, in particular, in $S\ell_3$. Hence $S\ell_3$ does not belong to the variety generated by the involutory monoid $(UT_n(F),{}^S)$.
\end{proof}

\subsection{Summary of known results and our contribution}
As in Section~\ref{subsec:summary}, we provide an overview of the known results and some of our present findings (highlighted in blue) in the form of a table. Table~\ref{tab:involutorystatus} is organized in the same way as Table~\ref{tab:plainstatus}: its rows are labeled by the underlying semirings $\mathbb{S}$ while the columns correspond to the involutory matrix monoids $(M_n(\mathbb{S}), {}^S)$,  $(T_n(\mathbb{S}), {}^S)$, $(U^0T_n(\mathbb{S}), {}^S)$, and $(UT_n(\mathbb{S}), {}^S)$. Cells in a row labeled by a semiring are merged whenever the corresponding matrix monoids over this semiring coincide.

It should be clear that, in view of the Lemma~\ref{lem:twisted1}, each nonfinite-basedness result in the columns of Table~\ref{tab:plainstatus} corresponding to the monoids $M_n(\mathbb{S})$, $T_n(\mathbb{S})$, and $U^0T_n(\mathbb{S})$ implies an analogous result for the involutory matrix monoids $(M_n(\mathbb{S}),{}^S)$, $(T_n(\mathbb{S}),{}^S)$, and $(U^0T_n(\mathbb{S}),{}^S)$, respectively, by the Transfer Lemma. By Lemma~\ref{lem:twisted2}{\red.1}, the same transfer applies to the nonfinite-basedness results in the column corresponding to the monoid $UT_n(\mathbb{S})$ whenever $n\ge 3$ and $\mathbb{S}$ is either an additively idempotent semiring, a field of characteristic $0$, or the ring $\mathbb{Z}$. Such `induced' results are highlighted in pink. Some of them have already appeared in the literature, while some, to the best of our knowledge, have not. We overview the consequences of the Transfer Lemma in Subsection~\ref{subsec:transfer}; see Proposition~\hyperref[prop:multiprop]{3.4.i} and Table~\ref{tab:transfer}.

As in Table~\ref{tab:plainstatus}, some cells in Table~\ref{tab:involutorystatus} are empty or incomplete in the sense that the results gathered in a cell do not cover all involutory monoids corresponding to it. This indicates that no solution to the Finite Basis Problem for the monoids not mentioned therein is currently known to us. An explicit list of open problems is given in Section~\ref{subsec:involopen}.

\begin{landscape}
\newcommand{\circled}[1]{%
  \tikz[baseline=(char.base)]{
    \node[shape=circle,draw,inner sep=1pt] (char) {#1};}}
\begin{table}[p]
\caption{Current status of the Finite Basis Problem for the involutory monoids $(M_n(\mathbb{S}), {}^S)$,  $(T_n(\mathbb{S}), {}^S)$, $(U^0T_n(\mathbb{S}), {}^S)$, and $(UT_n(\mathbb{S}), {}^S)$ Acronyms: FB = finitely based; HFB = hereditarily finitely based; NFB = nonfinitely based; SNFB = strongly nonfinitely based; INFB = inherently nonfinitely based. Footnote marks 1)--4) point to comments on the next page; marks of the form \protect\circled{i} refer to Proposition~\hyperref[prop:multiprop]{3.4.i}, i = 1,\dots,14.}
    \label{tab:involutorystatus}
 {\centering
    \begin{tabular}{|c|c|c|c|c|}
   \hline
   \rule[-6pt]{0pt}{18pt} Semiring $\mathbb{S}$ & $(M_n(\mathbb{S}), {}^S)$ & $(T_n(\mathbb{S}), {}^S)$ & $(UT^0_n(\mathbb{S}), {}^S)$ & $(UT_n(\mathbb{S}), {}^S)$, $n\ge3^{1)}$\\
   \hline
    \rule[-2pt]{0pt}{16pt} Boolean semiring $\mathbb{B}$  &  \tikz[baseline]{%
\node[fill=pink!30, rounded corners]{INFB}}$^{\circled{\scriptsize 1}}$ & \multicolumn{2}{c|}{\tikz[baseline]{%
\node[fill=pink!30, rounded corners]{INFB for $n\ge3$}}$^{\circled{\scriptsize 2}}$}  & \tikz[baseline]{%
\node[fill=pink!30, rounded corners]{SNFB for $n\ge5$}}$^{\circled{\scriptsize 3}}$\\
    \rule[-6pt]{0pt}{18pt}   &  & \multicolumn{2}{c|}{NFB for $n=2$  \cite{GZL20}} & NFB for $n=3,4$ \cite{ZLW20}\\
    \hline
    \rule[-8pt]{0pt}{20pt} 2-element field $\mathbb{F}_2$ & \tikz[baseline]{%
\node[fill=pink!30, rounded corners]{INFB}}$^{\circled{\scriptsize 4}}$ & \multicolumn{2}{c|}{\tikz[baseline]{%
\node[fill=pink!30, rounded corners]{SNFB for $n\ge4$}}$^{\circled{\scriptsize 5}}$\tikz[baseline]{%
\node[rounded corners]{NFB for $n=2$ \cite{ZL20}}}} & HFB$^{3)}$\\
    \hline
    \rule[-2pt]{0pt}{16pt} Finite field &  \tikz[baseline]{%
\node[fill=pink!30, rounded corners]{INFB}}$^{\circled{\scriptsize 4}}$  & \tikz[baseline]{%
\node[fill=pink!30, rounded corners]{INFB for $n\ge4$}}$^{\circled{\scriptsize 6}}$  &\tikz[baseline]{%
\node[fill=pink!30, rounded corners]{SNFB for $n\ge4$}}$^{\circled{\scriptsize 7}}$ & HFB$^{3)}$ \\
    \rule[-8pt]{0pt}{12pt} with $\ge3$ elements &  &  NFB for $n=2$ \cite{ZL20} &  NFB for $n=2$ \cite{ZL20} & \\
    \hline
    \rule[-2pt]{0pt}{16pt}  Infinite field of &  FB$^{2)}$  &  FB for $n\ge 3$ &  NFB for $n=2$ \cite{ZL20} &  HFB$^{3)}$ \\
    \rule[-4pt]{0pt}{14pt} finite characteristic & &  ({\blue Proposition~\ref{prop:3x3infinitefield}}) & & \\
    \rule[-8pt]{0pt}{16pt} && NFB for $n=2$  \cite{ZL20} &&\\
    \hline
    \rule[-2pt]{0pt}{16pt} $\mathbb{Z}$ or field of &  FB$^{2)}$  &  FB for $n\ge 3$  & \tikz[baseline]{%
\node[fill=pink!30, rounded corners]{NFB for $n=3$}}$^{\circled{\scriptsize 8}}$  & \tikz[baseline]{%
\node[fill=pink!30, rounded corners]{NFB for $n\ge 3$}}$^{\circled{\scriptsize 9}}$ \\[-4pt]
    \raisebox{12pt}{characteristic 0} & &  \raisebox{10pt}{({\blue Proposition~\ref{prop:3x3infinitefield}})}  &  \raisebox{6pt}{NFB for $n=2$ \cite{ZJL17}} & \\[-8pt]
    \rule[-8pt]{0pt}{16pt} && NFB for $n=2$ \cite{ZJL17} &  &\\
    \hline
    \rule[-2pt]{0pt}{16pt} Tropical semiring $\mathbb{T}$ & &\tikz[baseline]{%
\node[fill=pink!30, rounded corners]{NFB for $n=2,3$}}$^{\circled{\scriptsize 10}}$ & \tikz[baseline]{%
\node[fill=pink!30, rounded corners]{INFB for $n\ge3$}}$^{\circled{\scriptsize 11}}$ & \tikz[baseline]{%
\node[fill=pink!30, rounded corners]{NFB for $n\ge5$}}$^{\circled{\scriptsize 12}}$ \\
     &  & & NFB for $n=2$ \cite{HZL21} & NFB for $n=3,4$ \cite{HZL21}\\
    \hline
    \rule{0pt}{16pt} Bounded & \tikz[baseline]{%
\node[fill=pink!30, rounded corners]{INFB}}$^{\circled{\scriptsize 13}}$ & \tikz[baseline]{%
\node[fill=pink!30, rounded corners]{INFB for $n\ge3$}}$^{\circled{\scriptsize 14}}$ &\tikz[baseline]{%
\node[fill=pink!30, rounded corners]{INFB for $n\ge3$}}$^{\circled{\scriptsize 11}}$ & \tikz[baseline]{%
\node[fill=pink!30, rounded corners]{NFB for $n\ge5$}}$^{\circled{\scriptsize 12}}$  \\
    \rule[-4pt]{0pt}{14pt} distributive lattice &  &  NFB for $n=2$ & NFB for $n=2$ \cite{HZL21} & NFB for $n=3,4$ \cite{HZL21}\\
    \rule[-8pt]{0pt}{18pt}  & & ({\blue Corollary \ref{cor:BLT}}) & &\\
    \hline
    \rule{0pt}{16pt} Semiring with & & &  \tikz[baseline]{%
\node[fill=pink!30, rounded corners]{INFB for $n\ge3$}}$^{\circled{\scriptsize 11}}$ & \tikz[baseline]{%
\node[fill=pink!30, rounded corners]{NFB for $n\ge5$}}$^{\circled{\scriptsize 12}}$ \\
    \rule[-8pt]{0pt}{18pt} idempotent addition & & & NFB for $n=2$ \cite{HZL21}$^{4)}$  & NFB for $n=3,4$ \cite{HZL21}$^{4)}$\\
    \hline
    \end{tabular}\par}
\end{table}
\end{landscape}

\begin{minipage}{0.96\linewidth}
\small
{}$^{1)}$ The case $n=2$ is excluded because for any semiring $\mathbb{S}$, the involutory monoid $(UT_2(\mathbb{S}), {}^S)$ is finitely based. Indeed, every $2\times 2$ unitriangular matrix is symmetric with respect to the secondary diagonal, whence $(UT_2(\mathbb{S}), {}^S)$ satisfies the identity $x^*\approx x$. Obviously, modulo this identity, every involutory semigroup identity is equivalent to a `plain' semigroup identity. Furthermore, the monoid $UT_2(\mathbb{S})$ is commutative, and every commutative semigroup is finitely based~\cite[Theorem~9]{Per69}.

\smallskip

{}$^{2)}$ For fields $F$ of characteristic $\ne2$ containing square roots of ${-1}$ and ${2}$, the involutory monoids $(M_n({F}), {}^S)$ and $(M_n({F}), {}^T)$, where ${}^T$  denotes the usual transpose, are isomorphic; see~\cite[p. 45]{ADPV14} for an explicit isomorphism. That the latter involutory monoid is finitely based for every infinite field $F$ follows from~\cite[Theorem 3.7]{ADV12}. Alternatively, for any infinite field $F$ and any $n\ge 3$, the fact that the involutory monoid $(M_n({F}), {}^S)$ is finitely based follows from Proposition~\ref{prop:3x3infinitefield}.

\smallskip

{}$^{3)}$ For every field $F$ of finite characteristic $p$, the monoid $UT_n(F)$ is a nilpotent group of exponent $p^{\lceil\log_p n\rceil}$. For any group $(G,\cdot,{}^{-1})$ of exponent $e>1$, every basis $\Gamma$ of its group identities can be converted into a basis for the semigroup identities of the semigroup $G$ by replacing, for every variable $x$, each occurrence of $x^{-1}$ with $x^{e-1}$ in all identities in $\Gamma$ and adding the identities $x^ey\approx y\approx yx^e$. Clearly, the resulting identity basis of $G$ is finite whenever $\Gamma$ is finite. The group identities of every nilpotent group have a finite basis by Lyndon's theorem~\cite[Theorem 34.14]{Neu67}, and an inspection of the argument in Lyndon's theorem shows that the result holds for involutory identities of nilpotent groups equipped with an additional involution. Since the variety $\var(UT_n(F),{}^S)$ consists of nilpotent groups of exponent dividing $p^{\lceil\log_p n \rceil}$ equipped with an additional involution, this variety is hereditarily finitely based.

\smallskip

{}$^{4)}$ In \cite{HZL21}, additively idempotent semirings are assumed to have commutative multiplication. Inspecting the arguments in \cite{HZL21} shows that this assumption is inessential for the cited results to hold.
\end{minipage}

\bigskip

\subsection{Proofs}
\subsubsection{Consequences of the Transfer Lemma}
\label{subsec:transfer}
To present all the `induced' results highlighted in pink in Table~\ref{tab:involutorystatus} in a compact form, we have organized them in Table~\ref{tab:transfer}; see the next page. The $i$-th row of Table~\ref{tab:transfer} corresponds to the following statement:

\addtocounter{thm}{1}
\begin{multiprop}\label{prop:multiprop}
The involutory monoid listed in the first column has the nonfinite-basedness property specified in the second column.
\end{multiprop}

\begin{proof}
The Transfer Lemma applies, since the involutory monoid in the first column is twisted by either Lemma~\ref{lem:twisted1} or Lemma~\ref{lem:twisted2}{\red.1}, and its semigroup reduct has the same nonfinite-basedness property by the result(s) cited in the third column.
\end{proof}

In the last column of Table~\ref{tab:transfer}, we cite the source in which Proposition~\hyperref[prop:multiprop]{3.4.i} first appeared. In two cases (marked with $\star$), the corresponding result was proved before the Transfer Lemma had been established; therefore, the present argument based on this lemma streamlines the original proof. Empty cells in the last column indicate cases where, to the best of our knowledge, the corresponding result has not been explicitly recorded in the literature. Nevertheless, except for rows 7, 11, 13, and 14, where the underlying plain monoid fact in the third column comes from the present paper, Proposition~\hyperref[prop:multiprop]{3.4.i} is almost certainly known to experts.

\renewcommand{\arraystretch}{1.4}
\begin{table}[p]
\caption{Nonfinite-basedness results for the involutory monoids $(M_n(\mathbb{S}), {}^S)$,  $(T_n(\mathbb{S}), {}^S)$, $(U^0T_n(\mathbb{S}), {}^S)$, and $(UT_n(\mathbb{S}), {}^S)$ that follow from the Transfer Lemma. Acronyms: NFB = nonfinitely based; SNFB = strongly nonfinitely based; INFB = inherently nonfinitely based; AIS = additively idempotent semiring; BDL = bounded distributive lattice. In the last column, $\star$ marks results obtained before the Transfer Lemma was established.}
\label{tab:transfer}
{\centering
\begin{tabular}{r|p{2.5cm}|p{1.7cm}|p{6.4cm}|p{3.1cm}|}
\cline{2-5}
&\multicolumn{1}{|c|}{Monoid} &
\multicolumn{1}{c|}{Result} &
\multicolumn{1}{c|}{Reference for plain monoid result} &
\multicolumn{1}{c|}{Reference} \\
\cline{2-5}
1.&$(M_n(\mathbb{B}), {}^S)$ & INFB & See Comment {}$^{1)}$ after Table~\ref{tab:plainstatus} &\\
\cline{2-5}
2.&$(T_n(\mathbb{B}), {}^S)$ & INFB for~$n\ge 3$ & \cite[Theorem 2.1]{VG04} for $n\ge 4$;

\cite[Theorem 4.1]{LL11} for $n=3$ & \cite[Theorem 3.14]{ADV12}$^\star$\\
\cline{2-5}
3.&$(UT_n(\mathbb{B}), {}^S)$ & SNFB for~$n\ge 5$ & \cite[Theorem~1]{Vol04} and \cite[Theorem~8]{GSV25},

combined with \cite[Proposition~2.7]{SV23} &\\
\cline{2-5}
4.&$(M_n(F), {}^S)$

$F$ = finite field & INFB & \cite[Corollary 6.2]{Sap87a} & \\
\cline{2-5}
5.&$(T_n(\mathbb{F}_2), {}^S)$ & SNFB for~$n\ge 4$ & \cite[Theorem 11]{GSV25} &\\
\cline{2-5}
6.&$(T_n(F), {}^S)$

$F$ = finite field
with $|F|\ge 3$
& INFB for~$n\ge 4$ & \cite[Theorem on p. 474]{VG03} &\\
\cline{2-5}
7.&$(U^0T_n(F), {}^S)$

$F$ = finite field & SNFB for~$n\ge 4$ & Theorem~\ref{thm:strongU0T} &\\
\cline{2-5}
8.&$(U^0T_3(F), {}^S)$

$F=\mathbb{Z}$ or field with $\mathsf{char}\,F{=}0$ & NFB & \cite[Theorem 2]{Vol15} & \cite[Theorem 2]{Vol15}$^\star$\\
\cline{2-5}
9.&$(UT_n(F), {}^S)$
$F=\mathbb{Z}$ or field
with $\mathsf{char}\,F{=}0$ & NFB  for~$n\ge 3$ & \cite[Proposition 5]{Sap87a} &\\
\cline{2-5}
10.&$(T_n(\mathbb{T}), {}^S)$ & NFB for~$n{=}2,3$ & \cite[Corollary 5.4]{CHLS16} for $n=2$;

\cite[Theorem 3.15]{HZL21} for $n=3$
& \cite[Theorem 3.3]{HZL21} \cite[Theorem 3.16]{HZL21}\\
\cline{2-5}
11.&$(U^0T_n(\mathbb{S}), {}^S)$

$\mathbb{S}$ = AIS & INFB for~$n\ge 3$ & Corollary~\ref{cor:inherentU0T} &\\
\cline{2-5}
12.&$(UT_n(\mathbb{S}), {}^S)$

$\mathbb{S}$ = AIS & NFB for~$n\ge 5$ & \cite[Corollary~3.4]{JF19}, combined with

\cite[Theorem~1]{Vol04}& \cite[Theorem 5.3]{HZL21}\\
\cline{2-5}
13.&$(M_n(L), {}^S)$

L = BDL & INFB & Corollary~\ref{cor:lattice} &\\
\cline{2-5}
14.&$(T_n(L), {}^S)$

L = BDL & INFB for~$n\ge 3$ & Corollary~\ref{cor:trianglelattice} &\\
\cline{2-5}
\end{tabular}\par}
\end{table}

In addition, we include a few further applications of the Transfer Lemma.

\begin{prop}\label{prop:pminvol}
The involutory monoid $(U^{\pm}T_n(F), {}^S)$ is nonfinitely based if $F$ is a field of characteristic $0$, and inherently nonfinitely based if $F$ is a field of odd characteristic and $n\ge 4$.
\end{prop}

\begin{proof}
The involutory monoid $(U^{\pm}T_n(F), {}^S)$ is twisted by Remark~3 made after Lemma~\ref{lem:twisted1}.

If $F$ is a field of characteristic $0$, the monoid $U^{\pm}T_n(F)$ is nonfinitely based by Theorem~\ref{thm:UpmTn}, and Claim 1 of the Transfer Lemma applies.

If $F$ is a finite field of odd characteristic, the monoid $U^{\pm}T_n(F)$ contains the monoid $U^{\pm}_{11}T_n(F)$ defined in \eqref{eq:UpmT} as a submonoid. For $n\ge 4$, the latter monoid is inherently nonfinitely based by Theorem~\ref{thm:pm1}, and hence $U^{\pm}T_n(F)$ is also inherently nonfinitely based. If $F$ is an infinite field of odd characteristic and $n\ge4$, the monoid  $U^{\pm}T_n(F)$ is inherently nonfinitely based by Theorem~\ref{thm:UpmTn-prime}. Hence, Claim 3 of the Transfer Lemma applies, showing that also $(U^{\pm}T_n(F), {}^S)$ with $n\ge 4$ is inherently nonfinitely based for any field $F$ of odd characteristic.
\end{proof}

If $F$ is a field of characteristic different from 2 and $n=2m$ is even, the monoid $T_n(F)$ admits another involution, the so-called \emph{symplectic} involution. The \emph{symplectic transpose} $A^Y$ of a matrix $A\in T_n(F)$ is defined by
\[
A^{Y}:=\begin{pmatrix}
    E_m & 0\\
    0 & -E_m
\end{pmatrix}A^{S}\begin{pmatrix}
    E_m & 0\\
    0 & -E_m
\end{pmatrix},
\]
where $E_m$ is the $m\times m$ identity matrix. It is known that, under mild conditions on the field $F$, every involution ${}^*$ on $T_n(F)$ is equivalent to either the skew transposition or the symplectic involution, in the sense that the involutory monoid $(T_n(F),{}^*)$ is isomorphic to either $(T_n(F),{}^S)$ or $(T_n(F),{}^Y)$; see \cite[Proposition~2.5]{DKL06}.\footnote{In \cite{DKL06}, the result is stated for involutions ${}^*$ preserving addition, that is, satisfying $(a+b)^*=a^*+b^*$. However, it is known that every automorphism of the monoid $T_n(F)$ respects addition~\cite[Corollary of Theorem 7]{St69}, and this easily implies that every involution of this monoid also respects addition.}

Each of the submonoids $U^0T_n(F)$ and $U^{\pm}T_n(F)$ is closed under the symplectic involution, and the same argument as in Lemma~\ref{lem:twisted1} shows that each of the involutory monoids $(U^0T_n(F),{}^Y)$, $(U^{\pm}T_n(F),{}^Y)$, and $(T_n(F),{}^Y)$ is twisted. Therefore, the Transfer Lemma yields the following nonfinite-basedness results for these involutory monoids for every even $n\ge 4$:
\begin{itemize}
    \item the analogs of Propositions~\hyperref[prop:multiprop]{3.4.6} and~\hyperref[prop:multiprop]{3.4.7}, with ${}^S$ replaced by ${}^Y$ and $F$ any finite field of odd characteristic;
    \item the analog of Proposition~\ref{prop:pminvol}, again with ${}^S$ replaced by ${}^Y$ and $F$ any field of characteristic different from 2.
\end{itemize}

\begin{prop}\label{prop:etinvol}
For every semiring \/ $\mathbb{S}$ with commuting and small idempotents, the involutory monoid $(ET_n(\mathbb{S}), {}^S)$ is inherently nonfinitely based whenever $n\ge 3$.
\end{prop}

\begin{proof}
The involutory monoid $(ET_n(\mathbb{S}), {}^S)$  is twisted by Remark~2 made after Lemma~\ref{lem:twisted1}. Since the monoid $ET_n(\mathbb{S})$ with $n\ge 3$ is inherently nonfinitely based by Corollary~\ref{cor:inherentET}, the Transfer Lemma applies, showing that also $(ET_n(\mathbb{S}), {}^S)$ with $n\ge 3$ is inherently nonfinitely based.
\end{proof}

\subsubsection{Matrices over fields}
\label{subsec:fieldsinv}

The following result uses machinery analogous to that employed in \cite[Section 3.2]{ADV12} in the study of the identities of the involutory monoid $(M_n({F}), {}^T)$, where ${}^T$ is the usual transpose. We nevertheless include a detailed proof for completeness.

\begin{prop}\label{prop:3x3infinitefield}
For any infinite field $F$, the involutory monoid $(T_n(F), {}^S)$ with $n\ge 3$ is finitely based; moreover, all identities holding in $(T_n(F), {}^S)$ with $n\ge 3$ follow from associativity and the involution laws $(xy)^S\approx y^Sx^S$, $(x^S)^S\approx x$.
\end{prop}

We need the following observation, in which $F[t]$ stands for the ring of polynomials in $t$ with coefficients in a field $F$.

\begin{lem}\label{lem:free}
 For every field $F$, the matrices
\[
A:=\begin{pmatrix}
t&1&0\\
0&1&0\\
0&0&t
\end{pmatrix},
\quad
B:=\begin{pmatrix}
t&0&0\\
0&1&1\\
0&0&t
\end{pmatrix}
\]
generate a free subsemigroup in the monoid $T_3(F[t])$, and the matrices
\[
\begin{pmatrix}
t&1&0&0\\
0&1&0&0\\
0&0&1&0\\
0&0&0&t
\end{pmatrix},
\quad
\begin{pmatrix}
t&0&0&0\\
0&1&0&0\\
0&0&1&1\\
0&0&0&t
\end{pmatrix}
\]
generate a free subsemigroup in the monoid $T_4(F[t])$.
\end{lem}

\begin{proof}
Both claims follow by the same reasoning, which is in fact a standard freeness argument for matrix semigroups. We present it only for the case of $3\times 3$ matrices.

Consider the morphism $\varphi\colon\{x,y\}^+\to T_3(F[t])$ extending the substitution $x\mapsto A$, $y\mapsto B$. To prove that the subsemigroup generated by $A$ and $B$ is free, it suffices to show that $\varphi$ is one-to-one, that is, that each word $\mathbf{w}\in\{x,y\}^+$ is uniquely determined by the matrix $\varphi(\mathbf{w})$.
Let $\mathbf{w}=w_1\cdots w_k$, where $w_1,\dots,w_k\in\{x,y\}$. A straightforward induction on $k$ shows that
\[
\varphi(\mathbf{w})=\begin{pmatrix}
t^{k}& \sum\limits_{\{i\,\mid\,w_i=x\}} t^{i-1} & p_{13}(t)\\
0&1& \sum\limits_{\{j\,\mid\,w_j=y\}} t^{j-1}\\
0&0&t^{k}
\end{pmatrix},
\]
where $p_{13}(t)$ is some polynomial. Consequently, the matrix $\varphi(\mathbf{w})$ uniquely determines the positions of the variables $x$ and $y$ in $\mathbf{w}$.
\end{proof}

\begin{proof}[Proof of Proposition~\ref{prop:3x3infinitefield}] If we call identities that follow from associativity and the involution laws \emph{trivial}, our claim amounts to saying that the involutory monoid $(T_n(F), {}^S)$ with $n\ge 3$ does not satisfy any nontrivial involutory identity. We first show this for $n=3$.

It is known (and easy to verify) that the free involutory semigroup on one generator $z$ contains, as a unary subsemigroup, a free involutory semigroup on countably many generators, namely, $\mathcal{I}(\mathcal{Z})$ where
\[
\mathcal{Z}:=\{zz^*z,\ z(z^*)^2z,\ \dots,\ z(z^*)^nz,\ \dots\}.
\]
Therefore we only need to verify that $(T_3(F), {}^S)$ satisfies no nontrivial involutory identity in the single variable $z$. Such an identity can be written as $\mathbf{u}(z,z^S)\approx\mathbf{v}(z,z^S)$ where $\mathbf{u}$ and $\mathbf{v}$ are distinct words. The first claim of Lemma~\ref{lem:free} implies that $\mathbf{u}(A,B)\ne \mathbf{v}(A,B)$ in the monoid $T_3(F[t])$. Setting
\[
\mathbf{u}(A,B)=\begin{pmatrix} u_{11}&u_{12}&u_{13}\\ 0&u_{22}&u_{23}\\ 0&0&u_{33}
\end{pmatrix}\ \text{ and } \ \mathbf{v}(A,B)=\begin{pmatrix} v_{11} & v_{12}&v_{13}\\ 0&v_{22}&v_{23}\\ 0&0&v_{33}
\end{pmatrix},
\]
we see that, for some indices $1\le i\le j\le 3$, the polynomials $u_{ij}$ and $v_{ij}$ are distinct, whence $u_{ij}-v_{ij}$ is a nonzero polynomial. Now take
any $\lambda\in F$ and set $z(\lambda):=\left(\begin{smallmatrix}\lambda &1&0\\0&1&0\\0&0&\lambda\end{smallmatrix}\right)$. Then $z(\lambda)=A(\lambda)$ and $z(\lambda)^S=B(\lambda)$. Therefore, if the equality
\begin{equation}\label{eq:lambda}
\mathbf{u}(z(\lambda),z(\lambda)^S)=\mathbf{v}(z(\lambda),z(\lambda)^S)
\end{equation}
holds in $(T_3(F), {}^S)$, then $\lambda$ must be a root of the nonzero polynomial $u_{ij}-v_{ij}$. Since a nonzero polynomial over a field can have only finitely many roots, \eqref{eq:lambda} can hold for only finitely many $\lambda\in F$. As $F$ is an infinite field, equality \eqref{eq:lambda} fails for all but finitely many $\lambda$, and so, in particular, the identity $\mathbf{u}(z,z^S)\approx\mathbf{v}(z,z^S)$ fails in  $(T_3(F), {}^S)$.

In the same way, the second claim of Lemma~\ref{lem:free} implies that the involutory monoid $(T_4(F), {}^S)$ does not satisfy any nontrivial involutory identity. The general case of $(T_n(F), {}^S)$ for arbitrary $n\ge 3$ now follows from the observation that $(T_n(F), {}^S)$ contains an involutory submonoid isomorphic to $(T_3(F), {}^S)$ whenever $n$ is odd and to $(T_4(F), {}^S)$ whenever $n\ge 4$ is even. The submonoid isomorphic to $(T_3(F), {}^S)$ consists of all matrices of the form
\[
\begin{pmatrix}
\alpha_{11} & 0 & \cdots & 0 &
\alpha_{1m} &
0 & \cdots & 0 &
\alpha_{1n}
\\
0 & 0 & \cdots & 0 & 0 & 0 & \cdots & 0 & 0
\\
\vdots & \vdots &  & \vdots & \vdots &
\vdots &  & \vdots & \vdots
\\
0 & 0 & \cdots & 0 & 0 & 0 & \cdots & 0 & 0
\\
0 & 0 & \cdots & 0 &
\alpha_{mm} &
0 & \cdots & 0 &
\alpha_{mn}
\\
0 & 0 & \cdots & 0 & 0 & 0 & \cdots & 0 & 0
\\
\vdots & \vdots &  & \vdots & \vdots &
\vdots &  & \vdots & \vdots
\\
0 & 0 & \cdots & 0 & 0 & 0 & \cdots & 0 & 0
\\
0 & 0 & \cdots & 0 & 0 & 0 & \cdots & 0 &
\alpha_{nn}
\end{pmatrix},
\qquad\text{where }\ m=\frac{n+1}{2},
\]
that is, all matrices $\bigl(\alpha_{ij}\bigr)_{n\times n}\in T_n(F)$ such that $\alpha_{ij}=0$ whenever
\[
(i,j)\ne (1,1),(1,\frac{n+1}2),(1,n),(\frac{n+1}2,\frac{n+1}2),(\frac{n+1}2,n),(n,n).
\]
Likewise, the submonoid isomorphic to $(T_4(F), {}^S)$ consists of all matrices
of the form
\[
\begin{pmatrix}
\alpha_{11}
& 0 & \cdots & 0
& \alpha_{1m}
& \alpha_{1,m+1}
& 0 & \cdots & 0
& \alpha_{1n}
\\
0 & 0 & \cdots & 0
& 0 & 0
& 0 & \cdots & 0
& 0
\\
\vdots & \vdots & & \vdots
& \vdots & \vdots
& \vdots & & \vdots
& \vdots
\\
0 & 0 & \cdots & 0
& 0 & 0
& 0 & \cdots & 0
& 0
\\
0 & 0 & \cdots & 0
& \alpha_{mm}
& \alpha_{m,m+1}
& 0 & \cdots & 0
& \alpha_{mn}
\\
0 & 0 & \cdots & 0
& 0
& \alpha_{m+1,m+1}
& 0 & \cdots & 0
& \alpha_{m+1,n}
\\
0 & 0 & \cdots & 0
& 0 & 0
& 0 & \cdots & 0
& 0
\\
\vdots & \vdots & & \vdots
& \vdots & \vdots
& \vdots & & \vdots
& \vdots
\\
0 & 0 & \cdots & 0
& 0 & 0
& 0 & \cdots & 0
& 0
\\
0 & 0 & \cdots & 0
& 0 & 0
& 0 & \cdots & 0
& \alpha_{nn}
\end{pmatrix},
\qquad\text{where }\  m=\frac n2,
\]
that is, all matrices $\bigl(\alpha_{ij}\bigr)_{n\times n}\in T_n(F)$ such that $\alpha_{ij}=0$ whenever
\[(i,j)\ne (1,1),(1,\frac{n}2),(1,\frac{n}2+1),(1,n),(\frac{n}2,\frac{n}2),(\frac{n}2,\frac{n}2+1),(\frac{n}2,n),(\frac{n}2+1,\frac{n}2+1),(\frac{n}2+1,n),(n,n).
\]
\end{proof}

\begin{remks} 1. Clearly, the argument in the proof of Proposition~\ref{prop:3x3infinitefield} equally applies to the involutory monoid $(T_n(\mathbb{Z}), {}^S)$ with $n\ge 3$.

2. Proposition~\ref{prop:3x3infinitefield} contrasts the result of \cite{ZL20} where it was shown that the involutory monoid $(T_2(F), {}^S)$ is nonfinitely based for any field $F$.
\end{remks}

\subsubsection{Matrices over additively idempotent semirings} Here we address the finite basis problem for involutory monoids $(ET_n(\mathbb{S}), {}^S)$, where $\mathbb{S}$ is an additively idempotent semiring with commuting multiplicative idempotents. By Proposition~\ref{prop:etinvol}, this problem has a negative solution for all $n\ge 3$, albeit under the additional assumption that the multiplicative idempotents of $\mathbb{S}$ are small. We now turn to the remaining case $n=2$.

We will use a sufficient condition for nonfinite basedness of an involutory semigroup  established in~\cite{HZL21}. It involves two series of terms in the free involutory semigroup indexed by integers $k\ge 2$:
\[
\mathbf{p}_k:=\mathbf{a}_k\mathbf{c}_k\mathbf{b}_k\quad \text{and}\quad \mathbf{q}_k:=\mathbf{a}_k\mathbf{d}_k\mathbf{b}_k,
\]
where
\[
\begin{split}
\mathbf{a}_k:=b_{k-1}\cdots b_1a_{k-1}\cdots a_1, & \qquad \mathbf{c}_k:=x_1a_1\cdots x_{k-1}a_{k-1}x_kx^*_1b_1\cdots x^*_{k-1}b_{k-1}x^*_k,\\
\mathbf{b}_k:=a_1\cdots a_{k-1}b_1\cdots b_{k-1}, & \qquad  \mathbf{d}_k:=x_kb_{k-1}x_{k-1}\cdots b_1x_1x^*_ka_{k-1}x^*_{k-1}\cdots a_1x^*_1.
\end{split}
\]

\begin{prop}\cite[Theorem 4.1]{HZL21}\label{prop:T2BD}
An involution monoid $(M, {}^*)$  is nonfinitely based if it:
\begin{enumerate}
    \item[\emph{1)}] satisfies the identities $\mathbf{p}_k\approx \mathbf{q}_k$ for all $k\ge 2$, and
    \item[\emph{2)}] generates the variety containing the involution monoid $(T_2(\mathbb{B}), {}^S)$.
\end{enumerate}
\end{prop}

\begin{prop}\label{prop:ET2D}
Let \/ $\mathbb{S}$ be an additively idempotent semiring whose multiplicative idempotents commute. Then the involutory monoid $(ET_2(\mathbb{S}), {}^S)$ is nonfinitely based.
\end{prop}

\begin{proof}
Clearly, $(T_2(\mathbb{B}), {}^S)$ is an involutory submonoid of $(ET_2(\mathbb{S}), {}^S)$ and therefore belongs to the variety generated by $(ET_2(\mathbb{S}), {}^S)$. In view of Proposition \ref{prop:T2BD}, it remains to verify that the involutory monoid $(ET_2(\mathbb{S}), {}^S)$ satisfies each identity $\mathbf{p}_k\approx \mathbf{q}_k$ with $k\ge 2$.
	
Let $\varphi\colon\mathcal{X}\to ET_2(\mathbb{S})$ be any substitution. Since the diagonal entries of the matrices in $ET_2(\mathbb{S})$ are commuting idempotents and $\mathsf{con}(\mathbf{p}_k)=\mathsf{con}(\mathbf{q}_k)$, it follows that $\varphi(\mathbf{p}_k)_{ii}=\varphi(\mathbf{q}_k)_{ii}$ for $i=1, 2$. To prove  $\varphi(\mathbf{p}_k)=\varphi(\mathbf{q}_k)$, it therefore remains to show that $\varphi(\mathbf{p}_k)_{12}=\varphi(\mathbf{q}_k)_{12}$.

Let $A:=\varphi(\mathbf{a}_k)$, $B:=\varphi(\mathbf{b}_k)$, $C:=\varphi(\mathbf{c}_k)$, and $D:=\varphi(\mathbf{d}_k)$. Note that $\mathsf{con}(\mathbf{a}_k)=\mathsf{con}(\mathbf{b}_k)$ and $\mathsf{con}(\mathbf{c}_k)=\mathsf{con}(\mathbf{d}_k)$. Since the diagonal entries of the matrices in $ET_2(\mathbb{S})$ are commuting idempotents, it follows that $A_{ii}=B_{ii}$ and $C_{ii}=D_{ii}$ for $i=1, 2$. A direct calculation yields
\[
\begin{split}
	\varphi(\mathbf{p}_k)_{12} & =A_{11}C_{11}B_{12}+A_{11}C_{12}B_{22}+A_{12}C_{22}B_{22},\\
	\varphi(\mathbf{q}_k)_{12} & =A_{11}D_{11}B_{12}+A_{11}D_{12}B_{22}+A_{12}D_{22}B_{22}.
\end{split}
\]
Thus it suffices to show that $A_{11}C_{12}B_{22}=A_{11}D_{12}B_{22}$.
	
From the proof of Lemma \ref{lem:ijexpansion}, all terms in the expansion of $C_{12}$ (resp., $D_{12}$) are in one-to-one correspondence with all walks of length $4k-1$ from $1$ to $2$. Each term in the expansion of $C_{12}$ is of the form
\[
\varphi(x_1)_{11}\cdots\varphi(y)_{12}\cdots \varphi(x^*_k)_{22}=\varphi(x_1)_{11}\cdots \varphi(w)_{11}\varphi(y)_{12}\varphi(z)_{22}\cdots \varphi(x^*_k)_{22}.
\]
Let $F:=\{a_1, \cdots a_{k-1}, b_1, \cdots, b_{k-1}\}$ and $K:=\{x_1,\cdots, x_k, x^*_1,\cdots, x^*_k\}$. The expansion of $C_{12}$ divides into two sums according to the value of $y$:
\[
C_{12}=\sum_{y\in F} \varphi(x_1)_{11}\cdots \varphi(y)_{12}\cdots \varphi(x^*_k)_{22} +\sum_{y\in K}\varphi(x_1)_{11}\cdots \varphi(y)_{12}\cdots \varphi(x^*_k)_{22}.
\]
Similarly,
\[
D_{12}=\sum_{y\in F} \varphi(x_k)_{11}\cdots \varphi(y)_{12}\cdots \varphi(x^*_1)_{22} +\sum_{y\in K}\varphi(x_k)_{11}\cdots \varphi(y)_{12}\cdots \varphi(x^*_1)_{22}.
\]
For $i=1, \ldots, k-1$, since $a_i, b_i \in \mathsf{con}(\mathbf{a}_k)=\mathsf{con}(\mathbf{b}_k)$ and the diagonal entries of each matrix are commuting idempotent, we have
\[
A_{11}\varphi(a_i)_{11}=A_{11}=A_{11}\varphi(b_i)_{11}\quad \mbox{and}\quad \varphi(a_i)_{22}B_{22}=B_{22}=\varphi(b_i)_{22}B_{22}.
\]
It follows that
\begin{equation}\label{4.9}
\begin{split}
& A_{11}(\sum_{y\in K}\varphi(x_1)_{11}\cdots \varphi(y)_{12}\cdots \varphi(x^*_k)_{22})B_{22}=A_{11}\varphi(x_1\cdots x_kx^*_1\cdots x^*_k)_{12}B_{22}\\
= & A_{11}\varphi(x_k\cdots x_1x^*_k\cdots x^*_1)_{12}B_{22}=A_{11}\sum_{y\in K}\varphi(x_k)_{11}\cdots \varphi(y)_{12}\cdots \varphi(x^*_1)_{22}B_{22}.
\end{split}
\end{equation}
	
Now, observe that $\varphi(x_i)_{11}=\varphi(x^*_i)_{22}$ and $\varphi(x_i)_{22}=\varphi(x^*_i)_{11}$. Let
\[
X:=\sum\limits_{y\in F} \varphi(x_1)_{11}\cdots \varphi(y)_{12}\cdots \varphi(x^*_k)_{22}\quad \mbox{and}\quad Y:=\sum\limits_{y\in F} \varphi(x_k)_{11}\cdots \varphi(y)_{12}\cdots \varphi(x^*_1)_{22}.
\]
Consider a generic summand $\varphi(x_1)_{11}\cdots \varphi(y)_{12}\cdots \varphi(x^*_k)_{22}$ in $X$, where, without loss of generality, assume $y=a_{i-1}$ (the case $y=b_{i-1}$ is similar). Then
\[
\begin{split}
& A_{11}(\varphi(x_1)_{11}\cdots \varphi(a_{i-1})_{12}\cdots \varphi(x^*_k)_{22})B_{22}\\
= &  A_{11}(\varphi(x_1\cdots x_{i-1})_{11}\varphi(a_{i-1})_{12}\varphi(x_i\cdots x_kx^*_1\cdots x^*_k)_{22})B_{22}\\
= & A_{11}(\varphi((x_1\cdots x_{i-1})^*)_{22}\varphi(a_{i-1})_{12}\varphi((x_i\cdots x_kx^*_1\cdots x^*_k)^*)_{11})B_{22}\\
=& A_{11}(\varphi(x_k\cdots x_1x^*_k\cdots x^*_i)_{11}\varphi(a_{i-1})_{12}\varphi(x^*_{i-1}\cdots x^*_1)_{22}B_{22}\\
= & A_{11}(\varphi(x_k)_{11}\cdots \varphi(a_{i-1})_{12}\cdots \varphi(x^*_1)_{22})B_{22}
\end{split}
\]
Thus each summand of $A_{11}XB_{22}$ corresponds to a summand of $A_{11}YB_{22}$. The converse correspondence holds similarly, so $A_{11}XB_{22}=A_{11}YB_{22}$. Consequently,
\[
A_{11}C_{12}B_{22}=A_{11}D_{12}B_{22},
\]
which together with (\ref{4.9}) yields $\varphi(\mathbf{p}_k)_{12}=\varphi(\mathbf{q}_k)_{12}$, as required.
\end{proof}

We note that, in contrast to Proposition~\ref{prop:etinvol}, Proposition~\ref{prop:ET2D} does not require the assumption that the multiplicative idempotents of $\mathbb{S}$ are small.

Recall that $ET_n(L)=T_n(L)$ for any bounded distributive lattice $L$. Hence,  Proposition~\ref{prop:ET2D} implies the following:

\begin{cor}\label{cor:BLT}
For every bounded distributive lattice $L$, the involutory monoid $(T_2(L), {}^S)$, is nonfinitely based.
\end{cor}

Combining this fact with Proposition~\hyperref[prop:multiprop]{3.4.14}, we obtain a complete solution to the Finite Basis Problem for involutory monoids of triangular matrices over lattices.

\begin{prop}\label{prop:BLT}
For every bounded distributive lattice $L$, the involutory monoid $(T_n(L), {}^S)$ is nonfinitely based, and for $n\ge 3$, it is inherently nonfinitely based.
\end{prop}

\subsection{Open problems}
\label{subsec:involopen}
Here, in parallel to Section~\ref{subsec:plainopen}, we summarize the as yet unsettled instances of the Finite Basis Problem for monoids of triangular matrices equipped with skew transposition.

For involutory monoids of triangular matrices over a finite field, only the case of  $3\times 3$ matrices remains open.

\begin{problem}
Determine the status of the Finite Basis Problem for the involutory monoid of all $3\times 3$ triangular matrices over a finite field.
\end{problem}

\begin{problem}
 Determine the status of the Finite Basis Problem for the involutory monoid of all $3\times 3$ triangular matrices with diagonal entries $0$ and $1$ over a finite field.
\end{problem}

For involutory monoids of the form $(T_n(F),{}^S)$, where $F$ is an infinite field, the Finite Basis Problem is completely solved by Proposition~\ref{prop:3x3infinitefield}, \cite[Corollary 14]{ZJL17} and \cite[Corollary 4.12]{ZL20}. In contrast, involutory monoids of the form $(U^0T_n(F),{}^S)$, where $F$ is an infinite field, the answer to the Finite Basis Problem is known only for $n=2$ and, if $\mathsf{char}\,F= 0$, for $n=3$.

\begin{problem}
Determine the status of the Finite Basis Problem for the involutory monoid of all $n\times n$ triangular matrices with diagonal entries $0$ and $1$ over an infinite field for $n\ge 4$ and for $n=3$ if the field has prime characteristic.
\end{problem}

The monoid family $U^{\pm}T_n(F)$ is defined over any field $F$ of characteristic different from $2$. Proposition~\ref{prop:pminvol} shows that the  involutory monoid $(U^{\pm}T_n(F),{}^S)$  is nonfinitely based when $F$ has characteristic $0$ and inherently nonfinitely based when $n\ge 4$ and $F$ is a (finite or infinite) field of odd characteristic.

\begin{problem}
Determine the status of the Finite Basis Problem for the involutory monoid of all $2\times 2$ and $3\times 3$ triangular matrices with diagonal entries $0$ and $\pm1$ over a field of odd characteristic.
\end{problem}

Concerning the Finite Basis Problem for the involutory monoid $(T_n(\mathbb{T}),{}^S)$ of all $n\times n$ triangular tropical matrices, no results are currently known beyond the cases $n=2,3$, in which the Transfer Lemma applies; see Proposition~\hyperref[prop:multiprop]{3.4.10}.

\begin{problem}
Determine the status of the Finite Basis Problem for the involutory monoid of all $n\times n$ triangular matrices over the tropical semiring for $n\ge 4$.
\end{problem}

The monoid family $ET_n(\mathbb{S})$ is defined over any semiring whose multiplicative idempotents commute, but our results on the Finite Basis Problem for the involutory monoid $(ET_n(\mathbb{S}),{}^S)$  (Propositions~\ref{prop:etinvol} and~\ref{prop:ET2D}) cover only the cases where the idempotents in $\mathbb{S}$ are small or $n=2$.

\begin{problem}
Determine the status of the Finite Basis Problem for the involutory monoid of all $n\times n$ triangular matrices $(n\ge 3)$ with idempotent diagonal entries over additively idempotent semirings  whose multiplicative idempotents commute but are not necessarily small.
\end{problem}

\section*{Conclusion}
Studying triangular matrices from the viewpoint of the Finite Basis Problem for identities involving either multiplication alone or multiplication together with skew transposition reveals a number of interesting phenomena. In particular, the boundaries between different gradations of nonfinite basedness turn out to be remarkably subtle: even changing the admissible values of a single matrix entry from 0,1 to $\pm1$ may result in monoids with very different finite basis behavior.

A classification of methods for studying the Finite Basis Problem in semigroups was given in the third-named author's 2001 survey~\cite{Vol01}; see also \cite[Section 1.4.2]{Lee23a}. Since then, new useful methods have been developed, in particular, in~\cite{ACHLV15}, \cite{Lee17}, and~\cite{GSV25}. All the methods surveyed in~\cite{Vol01}, as well as the newer methods, have been employed in the present study, and the limitations of these methods become apparent in several cases. Since many open problems remain, the need for new techniques is evident, especially for addressing the Finite Basis Problem for monoids of matrices over the tropical semiring.

\medskip

\noindent\textbf{Acknowledgments.} The authors are indebted to Marianne Johnson and Edmond W. H. Lee for their comments and valuable suggestions.


\begin{thebibliography}{99}
\bibitem{AVG09}
J. Almeida,  M.V. Volkov, S. V. Goldberg, Complexity of the identity checking problem for finite semigroups. Zap. Nauchn. Sem. POMI \textbf{358}, 5--22 (2008) [In Russian; English translation: J. Math. Sci. \textbf{158}(5), 605--614 (2009)]

\bibitem{Ami74}
S. A. Amitsur, Polynomial identities. Israel J. Math. \textbf{19}, 183--199 (1974)

\bibitem{ACHLV15}
K. Auinger, Y. Z. Chen, X Hu, Y. F. Luo, M. V. Volkov, The finite basis
problem for Kauffman monoids. Algebra Universalis \textbf{74}(3–4), 333--350 (2015)

\bibitem{ADPV14}
K. Auinger, I. Dolinka, T. V. Pervukhina, M. V. Volkov, Unary enhancements of inherently non-finitely based semigroups. Semigroup Forum \textbf{89}(1), 41--51 (2014)

\bibitem{ADV12}
K. Auinger, I. Dolinka, M. V. Volkov, Matrix identities involving multiplication and transposition. J. Eur. Math. Soc.  \textbf{14}(3), 937--969 (2012).

\bibitem{AV20}
K. Auinger, M. V. Volkov, Combinatorial skeletons of 2-cobordism and annular categories with applications to equational logic. Preprint arXiv 2002.01016v2 (2026) https://doi.org/10.48550/arXiv.2002.01016

\bibitem{Bis04}
S. Bistarelli, Semirings for Soft Constraint Solving and Programming. Lecture Notes in Computer Science, vol. 2962. Springer, Berlin (2004)

\bibitem{Br68}
T.C. Brown, On locally finite semigroups. Ukrainsk. Mat. Zh. \textbf{20}, 732--738 (1968). [In Russian; English translation: Ukrainian Math. J. \textbf{20}, 631--636 (1968)]

\bibitem{Br71}
T.C. Brown, An interesting combinatorial method in the theory of locally finite semigroups. Pacific J. Math. \textbf{36}, 285--289 (1971)

\bibitem{BuSa81}
S. Burris, H.P. Sankappanavar, A Course in Universal Algebra. Springer, Berlin (1981)

\bibitem{CKR84}
Z. Q. Cao, K. H. Kim, F. W. Roush, Incline Algebra and Applications. Ellis Horwood, Chichester (1984)

\bibitem{Cha70}
M. Chacron, A note on matrices with entries in a distributive lattice. Bull. Soc. Math. Belg. \textbf{22}, 143--145 (1970)

\bibitem{CHL16}
Y. Z. Chen, X Hu, Y. F. Luo, The finite basis property of a certain semigroup of upper triangular matrices over a field. J. Algebra Appl. \textbf{15}(9), article no. 1650177 (2016)

\bibitem{CHLS16}
Y. Z. Chen, X. Hu, Y. F. Luo, O. Sapir, The finite basis problem for the monoid of two-by-two upper triangular tropical matrices. Bull. Aust. Math. Soc. \textbf{94}, 54--64 (2016)

\bibitem{CMZ17}
A. E. Clement, S. Majewicz, M. Zyman, The Theory of Nilpotent Groups. Birkh\"auser, Cham (2017)

\bibitem{DKL06}
O. M. Di Vincenzo, P. Koshlukov, R. La Scala, Involutions for upper triangular matrix algebras. Adv. Appl. Math. \textbf{37}, 541--568 (2006)

\bibitem{Edm77}
C. C. Edmunds, On certain finitely based varieties of semigroups. Semigroup Forum \textbf{15}, 21--39 (1977)

\bibitem{Faj72}
S. Fajtlowicz, Equationally complete semigroups with involution. Algebra Universalis \textbf{1}(1), 355--358 (1972)

\bibitem{GLLZ26}
M. Gao, E. W. H. Lee, Y. F. Luo, W. T. Zhang, Finite basis problem for involution semigroups of order four, Pacific J. Math. \textbf{342}(1), 163--206  (2026)

\bibitem{GZL20}
M. Gao, W. T. Zhang, Y. F. Luo, The monoid of $2\times 2$ triangular boolean matrices under skew transposition is non-finitely based. Semigroup Forum \textbf{100}, 153--168 (2020)

\bibitem{GZL202}
M. Gao, W. T. Zhang, Y. F. Luo, A nonfinitely based involutory semigroup of order five. Algebra Universalis \textbf{81}(3), article no. 31 (2020)

\bibitem{Gau96}
S. Gaubert, On the Burnside problem for semigroups of matrices in the (max, +) algebra. Semigroup Forum \textbf{52}, 271--292 (1996)

\bibitem{Give64}
Y. Give'on, Lattice matrices. Inform. Control \textbf{7}, 477--484 (1964)

\bibitem{GM78}
I. Z. Golubchik, A. V. Mikhalev, A note on varieties of semiprime rings with semigroup identities. J. Algebra \textbf{54}, 42--45 (1978)

\bibitem{Guna98}
J. Gunawardena (ed.). Idempotency. Publications of the Newton Institute, vol. 11. Cambridge University Press, Cambridge (1998)

\bibitem{GLV22}
S. V. Gusev, E. W.H. Lee, B. M. Vernikov, The lattice of varieties of monoids. Japan. J. Math. \textbf{17}, 117--183 (2022)

\bibitem{GSV25}
S. V. Gusev, O. B. Sapir,  M. V. Volkov, Strongly nonfinitely based monoids. J. Comb. Algebra \textbf{9}(1-2), 129--144 (2025)

\bibitem{GV*}
S. V. Gusev, M. V. Volkov, Strongly nonfinitely based monoids of small order (in preparation)

\bibitem{Hall:1959}
M. Hall, Jr., The Theory of Groups. The Macmillan Company, New York (1959)

\bibitem{HM79}
T. E. Hall, W. D. Munn, Semigroups satisfying minimal conditions. II.  Glasgow Math. J. \textbf{20}, 133--140 (1979)

\bibitem{HZL21}
B. B. Han, W. T. Zhang, Y. F. Luo, Equational theories of upper triangular tropical matrix semigroups. Algebra Universalis \textbf{82}(3), article no. 44 (2021)

\bibitem{Howie:1995}
J. M. Howie, Fundamentals of Semigroup Theory. Clarendon Press, Oxford (1995)

\bibitem{Isb70}
J. R. Isbell, Two examples in varieties of monoids. Proc. Cambridge Philos. Soc. \textbf{68}(2), 265--266 (1970)

\bibitem{Izh14}
Z. Izhakian, Semigroup identities in the monoid of triangular tropical matrices. Semigroup Forum \textbf{88}(1), 145--161 (2014)

\bibitem{JF19}
M. Johnson, P. Fenner, Identities in unitriangular and gossip monoids. Semigroup Forum \textbf{98}, 338--354 (2019)

\bibitem{Kam22}
M. Kambites, Free objects in triangular matrix varieties and quiver algebras over semirings. J. Algebra \textbf{590}, 439--462 (2022)

\bibitem{KR04}
K. H. Kim, F. W. Roush, Inclines and incline matrices: a survey. Linear Algebra Appl. \textbf{379}, 457--473 (2004)

\bibitem{KV24}
N. V. Kitov, M. V. Volkov, Identities in twisted Brauer monoids. In A. A. Ambily, V. B. Kiran Kumar (eds.), {Semigroups, Algebras and Operator Theory}. Springer Proceedings in Mathematics \& Statistics, vol. {436}, pp. 79--103. Springer, Singapore (2024)

\bibitem{Kra91}
A. N. Krasilnikov, The identities of a group with nilpotent commutator subgroup are finitely based. {Izv. Akad. Nauk SSSR, Ser. Mat.} \textbf{54}(6), 1181--1195 (1991). [In Russian; English translation: {Mathematics of the USSR--Izv.} \textbf{37}(3),  539--553 (1991)]

\bibitem{Lee76}
A. Lee, Secondary symmetric, skewsymmetric and orthogonal matrices. Period. Math. Hungarica \textbf{7}, 63--70 (1976)

\bibitem{Lee08}
E. W. H. Lee, On the variety generated by some monoid of order five. Acta Sci. Math. (Szeged) \textbf{74}, 509--537 (2008)

\bibitem{Lee12}
 E. W. H. Lee, {A sufficient condition for the non-finite basis property of semigroups}. Monatsh. Math. \textbf{168}(3-4), 461--472  (2012)

\bibitem{Lee16}
 E. W. H. Lee, Finitely based finite involution semigroups with non-finitely based reducts. Quaest. Math. \textbf{39}(2),  217--243 (2016)

\bibitem{Lee17}
E. W. H. Lee, Equational theories of unstable involution semigroups. Electron. Res. Announc. Math. Sci. \textbf{24}, 10--20 (2017)

\bibitem{Lee23a}
E. W. H. Lee, Advances in the Theory of Varieties of Semigroups. Frontiers in Mathematics. Birkh\"auser, Cham (2023)

\bibitem{Lee23b}
E. W. H. Lee,  Embedding finite involution semigroups in matrices with transposition. Discrete Appl. Math. 340, 327--330 (2023)

\bibitem{LZ15}
 E. W. H. Lee, W. T. Zhang, Finite basis problem for semigroups of order six. LMS J. Comput. Math. \textbf{18}(1), 1--129  (2015)

\bibitem{LL11}
J. R. Li, Y. F. Luo, On the finite basis problem for the monoids of triangular boolean matrices. Algebra Universalis \textbf{65}, 353--362 (2011)

\bibitem{LiMa:2005}
G. L. Litvinov, V. P. Maslov (eds.), Idempotent Mathematics and Mathematical Physics. Contemporary Mathematics, vol. 377. American Mathematical Society, Providence, RI (2005)

\bibitem{LS77}
R. C. Lyndon, P. E. Schupp, Combinatorial Group Theory. Ergebnisse der Mathematik und ihrer Grenzgebiete. 2. Folge, vol. 89. Springer, Berlin (1977)

\bibitem{Mal53}
A. I. Mal'cev, Nilpotent semigroups. Ivanov. Gos. Ped. Inst. U\v{c}. Zap. Fiz.-Mat. Nauki \textbf{4} (1953), 107--111 [In Russian; reprinted in: Selected Works. Vol. 1: Classical Algebra, pp. 335--339. Nauka, Moscow (1976)]

\bibitem{Mal67}
A. I. Mal'cev, Multiplication of classes of algebraic systems. Sibirsk. Mat. \v Z. \textbf{8}, 346--365 (1967) [In Russian; reprinted in: Selected Works. Vol. 2: Mathematical Logic and General Theory of Algebraic Systems, pp. 355--372. Nauka, Moscow (1976); English translation: Siberian Math.\ J. \textbf{8}, 254--267 (1967); reprinted in: The Metamathematics of Algebraic Systems. Collected Papers: 1936–1967. Studies in Logic and the Foundations of Mathematics, vol. 66, pp. 422--446. North Holland, Amsterdam (1971)]

\bibitem{Mal71}
Yu. N. Mal’cev, A basis for the identities of the algebra of upper-triangular matrices. Algebra i Logika \textbf{10}(4), 393--400 (1971). [In Russian; English translation:  Algebra and Logic \textbf{10}, 242--247 (1971)]

\bibitem{Mace4}
W. McCune, Prover9 and Mace4. (2009). http://www.cs.unm.edu/~mccune/prover9/


\bibitem{Neu67}
H. Neumann, Varieties of Groups. Ergebnisse der Mathematik und ihrer Grenzgebiete. 2. Folge, vol. 37. Springer, Berlin (1967)

\bibitem{Okn98}
J. Okni\'nski, Semigroups of Matrices. Series in Algebra, vol. 6. World Scientific, Singapore (1998)

\bibitem{Okn15}
J. Okni\'nski,  Identities of the semigroup of upper triangular tropical matrices. Commun.
Algebra \textbf{43}(10), 4422--4426 (2015)

\bibitem{Per69}
P. Perkins, Bases for equational theories of semigroups. J. Algebra \textbf{11},  298--314 (1969)

\bibitem{Pol80}
S. V. Polin, Identities of the algebra of triangular matrices. Sibirsk. Mat. Zh. \textbf{21}(4), 206--215 (1980).  [In Russian; English translation: Siberian Math. J. \textbf{21}  638--645  (1980)]


\bibitem{Sap87a}
M. V. Sapir, Problems of Burnside type and the finite basis property in varieties of semigroups. {Izv. Akad. Nauk SSSR, Ser. Mat.} \textbf{51}(2), 319--340 (1987). [In Russian; English translation: {Mathematics of the USSR--Izv.} \textbf{30}(2), 295--314 (1988)]

\bibitem{Sap87b}
M. V. Sapir, Inherently nonfinitely based finite semigroups. {Mat. Sb.} \textbf{133}(2), 154--166 (1987). [In Russian; English translation: Mathematics of the USSR--Sb. \textbf{61}, 155--166 (1988)]

\bibitem{SV23}
O. B. Sapir, M. V. Volkov, Catalan monoids inherently nonfinitely based relative to finite $\mathrsfs{R}$-trivial semigroups. J. Algebra \textbf{633},  138--171 (2023)

\bibitem{Sh94}
L. N. Shevrin, On the theory of epigroups. I.  {Mat. Sb.} \textbf{185}(8), 129--160 (1994). [In Russian; English translation: Russian Acad. Sci. Sb. Math., \textbf{82}(2), 485--512 (1995)]

\bibitem{SVV09}
L. N. Shevrin, B. M. Vernikov, M.V. Volkov, Lattices of semigroup varieties. Izv. Vyssh. Uchebn. Zaved. Mat. no. 3, 3--36 (2009). [In Russian; English translation: Russian Math. (Iz. VUZ), \textbf{53}(3), 1--28 (2009)]

\bibitem{Sid81}
P. N. Siderov, A basis for the identities of an algebra of triangular matrices over an arbitrary field. Pliska Stud. Math. Bulgar. \textbf{2}, 143--152 (1981). [In Russian]

\bibitem{Sko86}
L. A. Skornyakov, Invertible matrices over distributive structures. Sibirsk. Mat. Zh. \textbf{27}(2), 182–-185 (1986). [In Russian; English translation: Siberian Math. J. \textbf{27}(2), 289--292 (1986)]

\bibitem{St69}
W. Stephenson, Unique addition rings. Canadian J. Math. \textbf{21}, 1455--1461 (1969)

\bibitem{Tan02}
Yi-jia Tan, On compositions of lattice matrices. Fuzzy Sets Syst. \textbf{129}(1), 19--28 (2002)

\bibitem{Vol89}
M. V. Volkov, On finite basedness of semigroup varieties. Mat. Zametki \textbf{45}(3), 12--23 (1989) [In Russian; English translation (entitled `The finite basis question for varieties of semigroups'): Math. Notes \textbf{45}, 187--194 (1989)]

\bibitem{Vol00}
M. V. Volkov, The finite basis problem for finite semigroups: a survey. In: M.~P.~Smith, E.~Giraldes, P.~Mendes Martins (eds.), {Semigroups}, pp. 244--279. World Scientific, Singapore (2000)

\bibitem{Vol01}
M. V. Volkov, The finite basis problem for finite semigroups. {Sci. Math. Jpn.} \textbf{53}, 171--199 (2001)

\bibitem{Vol04}
M. V. Volkov, Reflexive relations, extensive transformations and piecewise testable languages of a given height. Int. J. Algebra and Computation \textbf{14}(5-6), 817--827 (2004)

\bibitem{Vol15}
 M. V. Volkov, A nonfinitely based semigroup of triangular matrices. In: J. Meakin,  A.R.Rajan, P. G. Romeo (eds.), Semigroups, Algebra and Operator Theory, [Springer Proc. Math. \& Statistics. 142], pp. 27--38. Springer India, New Delhi (2015)

 \bibitem{Vol25}
 M. V. Volkov, Identities of triangular Boolean matrices. In: G. Badia, A. Blass, N. Dershowitz, M. Droste (eds.), Fields of Logic and Computation IV: Essays Dedicated to Yuri Gurevich on the Occasion of His 85th Birthday, Lecture Notes in Computer Science. Springer, Cham (accepted). Preprint arXiv:2412.16113v4 (2025).  https://doi.org/10.48550/arXiv.2412.16113

\bibitem{VG03}
 M.V. Volkov, I.A. Goldberg, Identities of semigroups of triangular matrices over finite fields. Mat. Zametki \textbf{73}(4), 502--510 (2003). [In Russian; English translation: Math. Notes \textbf{73}(4), 474--481 (2003)]

\bibitem{VG04}
 M.V. Volkov, I.A. Goldberg, The finite basis problems for monoids of triangular boolean matrices. In: Algebraic Systems, Formal Languages, and Conventional and Unconventional Computation Theory, RIMS Kokyuroku \textbf{1366}, 205--214 (2004)

\bibitem{ZJL17}
W. T. Zhang, Y. D. Ji, Y. F. Luo, The finite basis problem for infinite involution semigroups of triangular $2\times 2$ matrices. Semigroup Forum \textbf{94}, 426--441 (2017)

\bibitem{ZLL12}
W. T. Zhang, J. R. Li,  Y. F. Luo, On the variety generated by the monoid of triangular $2 \times 2$ matrices over two-element field. Bull. Aust. Math. Soc. \textbf{86}, 64--77 (2012)

\bibitem{ZLL13}
W. T. Zhang, J. R. Li,  Y. F. Luo, Hereditarily finitely based semigroups of triangular matrices over finite fields. Semigroup Forum \textbf{86}(2), 229--261 (2013)

\bibitem{ZL11}
W. T. Zhang, Y. F. Luo, A new example of a minimal nonfinitely based semigroup. Bull. Aust. Math. Soc. \textbf{84}(3), 484--491 (2011)

\bibitem{ZL20}
W. T. Zhang, Y. F. Luo, The finite basis problem for involutory semigroups of triangular $2\times 2$ matrices. Bull. Aust. Math. Soc. \textbf{101}, 88--104 (2020)

\bibitem{ZLW20}
W. T. Zhang, Y. F. Luo, N. Wang, Finite basis problem for involution monoids of unitriangular boolean matrices. Algebra Universalis \textbf{81}(1), article no. 7 (2020)

\bibitem{Zim80}
A. I. Zimin, Semigroups that are nilpotent in the sense of Mal'cev, Izv. Vyssh. Uchebn. Zaved. Mat. no.6, 23--29  (1980). [In Russian]

\end{thebibliography}
\end{document}